\documentclass [12pt]{article}
 \usepackage{amssymb}
 \usepackage{amsmath,amsthm}
 \usepackage{mathrsfs}
 \usepackage{color}
 \usepackage{graphicx}

\normalsize

\newcommand\codeg{\operatorname{codeg}}
\newcommand\Bin{\operatorname{Bin}}
\renewcommand{\dfrac}[2]{\lower0.15ex\hbox{\large$\textstyle\frac{#1}{#2}$}}

\newtheorem{theorem}{Theorem}[section]
\newtheorem{lemma}[theorem]{Lemma}

\newtheorem{corollary}[theorem]{Corollary}

\newtheorem{remark}[theorem]{Remark}

\numberwithin{equation}{section}

\begin{document}
\title
{Asymptotic enumeration of
linear hypergraphs with given number of vertices and edges}
\author
{{Brendan D. McKay$^1$\qquad Fang Tian$^2$}\\
{\small $^1$Research School of Computer Science}\\%
{\small Australian National University, Canberra ACT 2601, Australia} \\
{\small\tt brendan.mckay@anu.edu.au}\\[1ex]
{\small $^2$Department of Applied Mathematics}\\
{\small Shanghai University of Finance and Economics, Shanghai, 200433, China}\\
{\small\tt tianf@mail.shufe.edu.cn}
}

\date{}
\maketitle

 \begin{abstract}For $n\geq 3$, let $r=r(n)\geq 3$ be an
integer. A hypergraph is $r$-uniform
 if each edge is a set of $r$ vertices,
and is said to be linear if two edges intersect in at most
 one vertex. In this paper, the number of linear
 $r$-uniform hypergraphs on $n\to\infty$ vertices is determined asymptotically
 when the number of edges is $m(n)=o(r^{-3}n^{ \frac32})$.
 As one application, we find the probability of linearity for the
 independent-edge model of random $r$-uniform hypergraph when the
 expected number of edges is $o(r^{-3}n^{ \frac32})$.
 We also find the probability that a random $r$-uniform linear
 hypergraph with a given number of edges contains a given subhypergraph.
 \end{abstract}

%\begin{keywords}
%asymptotic enumeration,
%linear uniform hypergraph,  switching method, random hypergraph
%\end{keywords}

\section{Introduction}\label{s:1}

 For $n\geq 3$, let $r=r(n)$ and $\ell$ be integers such that $r=r(n)\geq 3$ and $2\leq\ell\leq r-1$.
A hypergraph ${H}$ on vertex set $[n]$ is an \textit{$r$-uniform hypergraph}
(\textit{$r$-graph} for short) if each edge is a set of  $r$ vertices.
An $r$-graph is  called a  partial Steiner $(n,r,\ell)$-system
if every subset of size $\ell$ is contained in at most one edge of $H$.
In particular, $(n,r,2)$-systems are also called  \textit{linear hypergraphs}, which implies that
 any two edges intersect in at most one vertex.
 Partial Steiner $(n,r,\ell)$-systems and the stronger version, Steiner $(n,r,\ell)$-systems,
where every $\ell$-set is contained in  precisely one edge of ${H}$,
are widely studied combinatorial designs. Little is known about the number of %
 distinct partial Steiner $(n,r,\ell)$-systems, denoted by $s(n,r,\ell)$.
 Grable and Phelps~\cite{grable96}
 used the R\"{o}dl nibble algorithm~\cite{rodl85} to obtain
 an asymptotic formula for $\log s(n,r,\ell)$ as $\ell\leq r-1$
 and $n\to\infty$.
 Asratian and Kuzjurin gave another proof~\cite{asas00}.

An interesting problem is the enumeration of hypergraphs with
given number of edges. Let $\mathcal{H}_r(n,m)$ denote the set of $r$-graphs on
 the vertex set $[n]$ with $m$  edges, and let  $\mathcal{L}_r(n,m)$ denote
 the set of all  linear hypergraphs in $\mathcal{H}_r(n,m)$.
%\red{For $r\geq 3$ a fixed integer, Karo\'{n}ski and {\L}uczak~\cite{mika97}
%derived an asymptotic formula for the number of connected $r$-graphs
%in $\mathcal{H}_r(n,m)$ as $m= \frac{n}{r-1}+o( \frac{\ln n}{\ln\ln n})$ and
%$n\to\infty$. Behrisch, Coja-Oghlan and Kang~\cite{kang1415}
%obtained their asymptotic formula for connected $r$-graphs in $\mathcal{H}_r(n,m)$
%as $m= \frac{n}{r-1}+\Theta(n)$ and  $n\to\infty$.
%Sato and Wormald~\cite{sato14} gave an asymptotic formula for connected  $3$-graphs
%in $\mathcal{H}_3(n,m)$ as $m= \frac{n}{2}+R$  and  $n\to\infty$,
%where $R=o(n)$ and $R=\omega(n^{1/3}\ln^2 n)$.  Bollob\'{a}s and Riordan~\cite{bollobas16}
%complemented the results in~\cite{kang1415,sato14} to cover the range for all
%$m= \frac{n}{r-1}+R$ as long as $R\to\infty$ and $R=o(n)$.
%Their proofs are all based on probabilistic methods  as $m$ is relatively large,
%rather than on enumerative methods.}

Dudek et al.~\cite{dudek13} used the switching method
to obtain the asymptotic number of $k$-regular $r$-graphs for fixed $r$
and $n\to\infty$ with $k=o(n^{1/2})$.
For $r=r(n)\geq 3$ an integer and a sequence of positive integers
$\boldsymbol{k}=\boldsymbol{k}(n)=(k_1,\ldots,k_n)$, define $M=M(n)=\sum_{i=1}^nk_i$.
Let $\mathcal{H}_r(\boldsymbol{k})$ denote the set of $r$-graphs on the
 vertex set $[n]$ with  degree sequence $\boldsymbol{k}$, and
 $\mathcal{L}_r(\boldsymbol{k})$ denote the set of all  linear hypergraphs
 in $\mathcal{H}_r(\boldsymbol{k})$.
Blinovsky and Greenhill~\cite{vlaejoc,valelec}  extended  the asymptotic
enumeration result on the number of $k$-regular $r$-graphs
 to the general $\mathcal{H}_r(\boldsymbol{k})$
when $r^4k_{\rm max}^3=o(M)$ and $n\to\infty$.
By relating the incidence matrix of a hypergraph to the biadjacency matrix of a
bipartite graph, they used switching arguments  together with previous
enumeration results for bipartite graphs to obtain the asymptotic
enumeration formula for $\mathcal{L}_r(\boldsymbol{k})$
provided $r^4k_{\rm max}^4(k_{\rm max}+r)=o(M)$ and $n\to\infty$~\cite{valelec}.
Recently, Balogh and Li~\cite{balgoh17} obtained an upper bound on the
total number of linear $r$-graphs with given girth for fixed $r\geq 3$.

Apart from these few results, the literature on the enumeration of linear
hypergraphs is very sparse.
In particular, there seems to be no asymptotic enumeration of linear
hypergraphs by the number of edges, which is the subject of this paper.
The result of Blinovsky and Greenhill~\cite{valelec} could in principle
be summed over degree sequences to obtain $\mathcal{L}_r(n,m)$
for $m=o\bigl( \min\{r^{-2}n^{\frac54},r^{-\frac83}n^{\frac43}\}\bigr)$,
but we prefer a direct switching approach.
Note that $m=O(r^{-2}n^2)$ for all linear hypergraphs; we get as far
as $m=o(r^{-3}n^{\frac32})$.

Our
application of the switching method combines several different switching
operations into a single computation, which was previously used in~\cite{green06} to
 count sparse $0$-$1$ matrices with irregular row and column sums, in~\cite{green08} to
 count sparse nonnegative integer matrices with specified row and column sums, and
 in~\cite{green13} to count sparse multigraphs with given degrees.

We will use the falling factorial $[x]_t=x(x-1)\cdots(x-t+1)$
and adopt $N$ as an abbreviation for $\binom nr$.
All asymptotics are with respect to $n\to\infty$.
Our main theorem is the following.

\begin{theorem}\label{t1.1}
Let $r=r(n)\geq 3$
and $m=m(n)$ be integers with $m=o(r^{-3}n^{ \frac32})$.
Then, as $n\to \infty$,
\begin{align}
|\mathcal{L}_r(n,m)|&=
 \binom{N}{m}
\exp\biggl[- \frac{[r]_2^2[m]_2}{4n^2}- \frac{[r]_2^3(3r^2-15r+20)m^3}{24n^4}+O\Bigl( \frac{r^6m^2}{n^3}\Bigr)\biggr] \label{maineqn}\\
&={ \frac{N^m}{m!}}
\exp\biggl[- \frac{[r]_2^2[m]_2}{4n^2}- \frac{[r]_2^3(3r^2-15r+20)m^3}{24n^4}+O\Bigl( \frac{r^6m^2}{n^3}\Bigr)\biggr]. \notag
\end{align}
\end{theorem}

\begin{proof}[Proof of Theorem~\ref{t1.1}]
Note that the condition $m=o(r^{-3}n^{ \frac32})$ implies that
either $m=0$ or $r=o(n^{\frac12})$. In the former case the theorem
is trivially true, while in the latter we can apply
Remark~\ref{r4.5} and Lemma~\ref{l6.6} to obtain~\eqref{maineqn}
for $r^{-2}n\leq m=o(r^{-3}n^{ \frac32})$.
Equivalent expressions follow from Remark~\ref{r7.3} and
Lemma~\ref{l7.7} when $\log(r^{-2}n)\leq m=O(r^{-2}m)$
and from Remark~\ref{r8.2} and Lemma~\ref{l8.6}
when $1\leq m=O(\log(r^{-2}n))$.
%
%By Theorem~\ref{t7.8} and
%Theorem~\ref{t8.7}, we obtain
%\[
%|\mathcal{L}_r(n,m)|=
% \binom{N}{m}
% \exp\biggl[- \frac{[r]_2^2[m]_2}{4n^2}+O\Bigl( \frac{r^6m^2}{n^3}\Bigr)\biggr]\\
%\]
%for $m=O(r^{-2}n)$, which is equivalent to~\eqref{maineqn} in that case.
%Finally, note that
%\[
% \binom{N}{m}
%%= \frac{N^m}{m!}\exp\biggl[- \frac{{m\choose 2}}{N}+O\biggl( \frac{m^3}{N^2}\biggr)\biggr]
%=  \frac{N^m}{m!}\exp\biggl[ O\Bigl( \frac{r^6m^2}{n^3}\Bigr)\biggr]. \qedhere
%\]
\end{proof}

Let $\mathbb{P}_r(n,m)$ denote the probability that an
$r$-graph $H\in \mathcal{H}_r(n,m)$ chosen uniformly at random is linear. Then
 \[
 |\mathcal{L}_r(n,m)|= \binom{N}{m} \mathbb{P}_r(n,m).
 \]
Hence, our task is reduced to computing $\mathbb{P}_r(n,m)$ and
it suffices to show that $\mathbb{P}_r(n,m)$ equals the exponential
 factor in Theorem~\ref{t1.1}.

Recall that a random $r$-graph ${H}_r(n,p)=([n],E_{n,p})$ refers to an $r$-graph
on the vertex set $[n]$, where each $r$-set is an edge randomly and
independently with probability~$p$.  Also it might be surmised that random
hypergraphs with edge probability~$p$ have about the same probability of
being linear as a random hypergraph with $Np$ edges, that is not the case
when $Np$ is moderately large.
Let $\mathcal{L}_r(n)$ be the set of all
linear $r$-graphs with $n$ vertices.

%By the law of total probability, \red{include this?}
%\begin{align*}
%\mathbb{P}\bigl[H_r(n,p)\in \mathcal{L}_r\bigr]
%&=\sum_{m=0}^{N}\mathbb{P}\bigl[H_r(n,p)\in \mathcal{L}_r\mid|E_{n,p}|=m\bigr]\mathbb{P}\bigl[|E_{n,p}|=m\bigr]\\
%&=\sum_{m=0}^{N}\mathbb{P}\bigl[H\in \mathcal{L}_r(n,m)\bigr]\mathbb{P}\bigl[|E_{n,p}|=m\bigr]\\
%&=\sum_{m=0}^{N}\mathbb{P}_r(n,m) \binom{N}{m}p^m(1-p)^{N-m}.
%\end{align*}
%As an application of Theorem~\ref{t1.1}, we obtain  the probability that a random $r$-graph
%$H_r(n,p)$ is to be linear.

\begin{theorem}\label{t1.3}
Let $r=r(n)\geq 3$ and let
$pN=m_0$ with $m_0=o(r^{-3}n^{ \frac32})$.
Then, as $n\to\infty$,
\begin{align*}
\mathbb{P}&[H_r(n,p)\in \mathcal{L}_r(n)]\\
&=\begin{cases}
\exp\Bigl[- \frac{[r]_2^2m_0^2}{4n^2}+O\bigl( \frac{r^6m_0^2}{n^3}\bigr)\Bigr],
 &\text{ if }m_0=O(r^{-2}n);\\[2ex]
\exp\Bigl[- \frac{[r]_2^2m_0^2}{4n^2}+ \frac{[r]_2^3(3r-5)m_0^3}{6n^4}+
O\bigl(\frac{\log^3(r^{-2}n)}{\sqrt{m_0}} + \frac{r^6m_0^2}{n^3}\bigr)\Bigr],
&\text{ if }r^{-2}n\leq m_0=o(r^{-3}n^{ \frac32}).
\end{cases}
\end{align*}
\end{theorem}

From the calculations in the proof of Theorem~\ref{t1.3}, we have a corollary about the distribution on the
number of edges of $H_r(n,p)$ conditioned on it being linear.

\begin{corollary}\label{c1.4}
Let $r=r(n)\geq 3$ and let $Np=m_0$
with $r^{-2}n\leq m_0=o(r^{-3}n^{ \frac32})$.
Suppose that $n\to \infty$.
Then the number of edges of $H_r(n,p)$ conditioned on being
linear converges in distribution to the normal distribution with mean $m_0- \frac{[r]_2^2m_0^2}{2n^2}$
and variance $m_0$.
\end{corollary}

Consider  $H\in\mathcal{L}_r(n,m)$ chosen uniformly at random.
Using a similar switching method, we also obtain the
probability that $H$ contains a given hypergraph as a subhypergraph.

\begin{theorem}\label{t1.5}
Let $r=r(n)\geq 3$, $m=m(n)$ and $k=k(n)$ be  integers with
 $m=o(r^{-3}n^{ \frac32})$ and
% $k=o\bigl(\min\bigl\{ \frac{n^2}{r^5m}, \frac{n^3}{r^6m^2} \bigr\}\bigr)$.
  $k=o\bigl(\frac{n^3}{r^6m^2}\bigr)$.
 Let $K=K(n)$ be a linear $r$-graph on $n$ vertices with $k$ edges.
Let $H\in\mathcal{L}_r(n,m)$ be chosen uniformly at random. Then, as $n\to\infty$,
\[
\mathbb{P}[K\subseteq H]= \frac{[m]_k}{N^k}\exp\biggl[ \frac{[r]_2^2k^2}{4n^2}+
 O\Bigl( \frac{r^4k}{n^2}+ \min\Bigl\{ \frac{r^6m^2k}{n^3}, \frac{r^5mk}{n^2}\Bigr\}\Bigr)\biggr].
\]
\end{theorem}

The remainder of the paper is structured as follows. Notation and  auxiliary results
are presented in Section~\ref{s:2}. From Section~\ref{s:3} to Section~\ref{s:6},
we mainly consider the case $r^{-2}n\leq m=o(r^{-3}n^{ \frac32})$.
In Section~\ref{s:3}, we define subsets $\mathcal{H}^+_r(n,m)$ and
$\mathcal{H}^{++}_r(n,m)$ of $\mathcal{H}_r(n,m)$ and show that they
are almost all of $\mathcal{H}_r(n,m)$.  In Section~\ref{s:4}, we
show that the same is true when $\mathcal{H}^+_r(n,m)$ and
$\mathcal{H}^{++}_r(n,m)$ are restricted by certain counts of clusters of
edges that overlap in more than one vertex.
We define four other kinds of
switchings on $r$-graphs in $\mathcal{H}_r^{+}(n,m)$
which are used to remove some hyperedges with two or more common vertices, and
analyze these switchings  in Section~\ref{s:5}.  In Section~\ref{s:6}, we complete
the enumeration for the case $r^{-2}n\leq m=o(r^{-3}n^{ \frac32})$
with the help of some calculations performed in~\cite{green06,green08,green13}.
In Sections~\ref{s:7}--\ref{s:8}, we consider the cases $\log (r^{-2}n)\leq m= O(r^{-2}n)$
and $1\le m=O(\log (r^{-2}n))$, respectively. In Section~\ref{s:9}, we prove Theorem~\ref{t1.3},
while in Section~\ref{s:10}, we prove Theorem~\ref{t1.5}.

\section{Notation and  auxiliary results}\label{s:2}

To state our results precisely, we need some definitions.
Let $H$ be an $r$-graph in $\mathcal{H}_r(n,m)$.
For $U\subseteq [n]$, the \textit{codegree} of $U$ in $H$, denoted by $\codeg({U})$,
is the number of edges of $H$ containing~$U$. In particular, if $U=\{v\}$ for
 $v\in [n]$ then $\codeg({U})$ is the degree of $v$ in $H$, denoted by $\deg (v)$.
Any $2$-set $\{x,y\}\subseteq [n]$ in an
edge $e$ of $H$ is called a \textit{link} of $e$ if $\codeg ({x,y})\geq 2$.
Two edges $e_i$ and $e_j$ in $H$ are called \textit{linked edges}
if $|e_i\cap e_j|=2$.

Let $G_H$ be a simple graph whose vertices are the edges of~$H$,
with two vertices of $G$ adjacent iff the corresponding edges of~$H$ are linked.
An edge induced subgraph of $H$ corresponding to a non-trivial
 component of $G_H$ is called a \textit{cluster} of $H$.
The standard asymptotic notations $o$ and $O$ refer to $n\to\infty$.
The floor and ceiling signs are omitted whenever they are not crucial.

In order to identify several events which have low probabilities
in the uniform probability space $\mathcal{H}_r(n,m)$ as
$m=o(r^{-3}n^{ \frac32})$, the following lemmas will be useful.

\begin{lemma}\label{l2.1}
Let $r=r(n)\geq 3$, $t=t(n)\geq 1$ be integers.
Let $e_1,\ldots,e_{t}$ be distinct $r$-sets of $[n]$ and $H$ be an
$r$-graph that is chosen uniformly at random from $\mathcal{H}_r(n,m)$.
Then the probability that $e_1,\ldots,e_t$ are  edges of $H$ is at most
$\bigl( \frac{m}{N}\bigr)^t$.
\end{lemma}

\begin{proof} Since $H$ is an $r$-graph
 that is chosen uniformly at random from $\mathcal{H}_r(n,m)$, the probability
 that $e_1,\ldots,e_t$ are edges of $H$ is
\begin{align*}
 \frac{ \binom{N-t}{m-t}}{ \binom{N}{m}}&= \frac{[m]_t}{[N]_t}=
\prod_{i=0}^{t-1} \frac{m-i}{N-i}\leq\Bigl( \frac{m}{N}\Bigr)^t. \qedhere
\end{align*}
\end{proof}

\begin{lemma}\label{l2.2}
Let $r=r(n)\geq 3$ be an integer with $r=o(n^{ \frac12})$.
Let $t$ and $\alpha$ be integers such that $t=O(1)$ and $0\leq\alpha\leq rt$.
If a hypergraph $H$ is chosen uniformly at random from $\mathcal{H}_r(n,m)$,
then the expected number of sets of $t$ edges of $H$ whose union has
$rt-\alpha$ or fewer vertices is $O\bigl( t^\alpha r^{2\alpha}m^t n^{-\alpha})$.
\end{lemma}

\begin{proof}
Let $e_1,\ldots,e_t$ be distinct $r$-sets of  $[n]$.
 According to Lemma~\ref{l2.1}, the probability that $e_1,\ldots,e_t$
 are edges of $H$ is at most $(m/N)^t$.
 Here, we firstly count
 how many $e_1,\ldots,e_t$ such that $|e_1\cup\cdots\cup e_t|=rt-\beta$
 for some $\beta\geq \alpha$.

 Suppose that there is a sequence among the edges $\{e_1,\ldots,e_t\}$ and
 we have chosen the edges $\{e_1,\ldots,e_{i-1}\}$, where $2\leq i\leq t$.
 Let $a_i=|(e_1\cup\cdots\cup e_{i-1})\cap e_i|$, thus we have
 $\sum_{i=2}^ta_i=\beta$ and
 \[
O\biggl(N\prod_{i=2}^{t}(tr)^{a_i} \binom{n}{r-a_i}\biggr)
 \]
 ways to choose $\{e_1,\ldots,e_t\}$. The expected number of these
$t$ edges is
\begin{align*}
O\biggl(\Bigl( \frac{m}{N}\Bigr)^tN\prod_{i=2}^{t}(tr)^{a_i} \binom{n}{r-a_i}\biggr)
% &=O\biggl(m^tt^\beta r^\beta\prod_{i=2}^{t} \frac{ \binom{n}{r-a_i}}{N}\biggr)\\
&=O\biggl(m^tt^\beta r^\beta\prod_{i=2}^{t} \frac{[r]_{a_i}}{[n-r]_{a_i}}\biggr)\\
% &=O\biggl(m^tt^\beta r^\beta\prod_{i=2}^{t}\prod_{j=0}^{a_i-1} \frac{r-j}{n-r-j}\biggr)\\
&=O\biggl(m^tt^\beta r^\beta\prod_{i=2}^{t}\Bigl( \frac{r}{n-r}\Bigr)^{a_i}\biggr)
% &=O\biggl(m^tt^\beta r^\beta\prod_{i=2}^{t}\Bigl( \frac{r}{n}\Bigr)^{a_i}\biggr)\\
=O\biggl( \frac{t^\beta m^t r^{2\beta}}{n^\beta}\biggr),
\end{align*}
where we use the fact that $\prod_{j=0}^{a_i-1} \frac{r-j}{n-r-j}\leq\bigl( \frac{r}{n-r}\bigr)^{a_i}$
is true as $r\leq  \frac{n}{2}$ and $\bigl( \frac{r}{n-r}\bigr)^{a_i}=O(( \frac{r}{n})^{a_i})$ because $a_i< r$,
$r=o(n^{ \frac12})$ and $a_i r=o(n)$.

The expected number of sets of $t$ edges whose union has at most
$rt-\alpha$ vertices is
\[
O\biggl(\sum_{\beta\geq\alpha} \frac{t^\beta m^t r^{2\beta}}{n^\beta}\biggr)
=O\biggl( \frac{t^\alpha m^t r^{2\alpha}}{n^\alpha}\biggr),
\]
because $\beta=\alpha$ corresponds to the
largest term as $t=O(1)$ and $r=o(n^{ \frac12})$.\end{proof}

\begin{remark}\label{r2.3}
Throughout the following sections we assume that $n\to\infty$.
From Sections~\ref{s:3}--\ref{s:6},
we assume that $r^{-2}n \leq m=o(r^{-3}n^{ \frac32})$.
 In Sections~\ref{s:7}--\ref{s:8} , we assume
that $\log (r^{-2}n)\leq m=O(r^{-2}n)$
and $1\le m=O(\log (r^{-2}n))$, respectively.
Recall that these conditions all imply that $r=o(n^{ \frac12})$.
\end{remark}

\section{Two important subsets of $\mathcal{H}_r(n,m)$}\label{s:3}

Define
\begin{equation}\label{e3.1}
\begin{split}
M_0^*&=\biggl\lceil\log (r^{-2}n)+ \frac{3^4r^2m}{n}\biggr\rceil, \\
M_0&=\biggl\lceil\log (r^{-2}n)+ \frac{3^4r^2m}{n}\biggr\rceil+3, \\
M_1&=\biggl\lceil\log(r^{-2}n)+ \frac{3^4r^8 m^3}{2n^4}\biggr\rceil, \\
M_2&=\biggl\lceil\log(r^{-2}n)+ \frac{3^4r^7m^3}{2n^4}\biggr\rceil, \\
M_3&=\biggl\lceil\log(r^{-2}n)+ \frac{3^4r^6m^3}{2n^4}\biggr\rceil, \\
M_4&=\biggl\lceil\log(r^{-2}n)+ \frac{3^4r^4m^2}{2n^2}\biggr\rceil.
\end{split}
\end{equation}
%We make no attempt to optimize the coefficients in all terms. This is for
%the interest of understanding the role of these constants played
%in the calculations.

Now define $\mathcal{H}_r^+(n,m)\subseteq\mathcal{H}_r(n,m)$
to be the set of $r$-graphs $H$ which satisfy the following
properties $\bf(a)$ to $\bf(g)$.

$\bf(a)$\  The intersection of any two edges contains at most two vertices.

$\bf(b)$\  $H$ only contains the four types of clusters that are shown in Figure~\ref{fig:1}.
(This implies that any three edges of $H$ involve at least $3r-4$ vertices and
any four edges involve at least $4r-5$ vertices. Thus, if there are
three edges of $H$, for example $\{e_1,e_2,e_3\}$, such that $|e_1\cup e_2\cup e_3|=3r-4$,
then $|e\cap (e_1\cup e_2\cup e_3)|\leq 1$  for any
edge $e$ other than $\{e_1,e_2,e_3\}$ of $H$.)
\begin{figure}[!htb]
\centering
\includegraphics[width=1.0\textwidth]{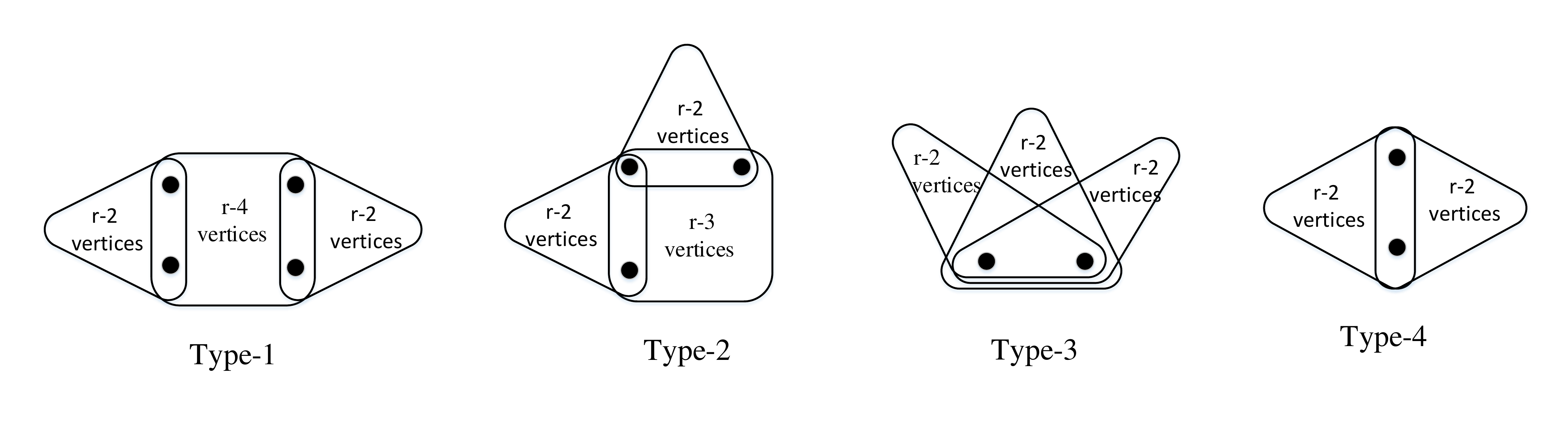}
\caption{The four types of clusters allowed in $H\in\mathcal{H}_r^+(n,m)$.\label{fig:1}}
\end{figure}

$\bf(c)$\ The intersection of any two clusters contains at most one vertex.

$\bf(d)$\ Any three distinct Type-$1$, Type-$2$ or Type-$3$ clusters
 involve at least $9r-13$ vertices. (Together with~$\bf(c)$,
this implies that if a pair of Type-$1$, Type-$2$ or Type-$3$ clusters
have exactly one common vertex, then any other
Type-$1$, Type-$2$ or Type-$3$ clusters
of $H$  must be vertex-disjoint from them.)

$\bf(e)$\ Any three distinct Type-$4$ clusters involve at least $6r-8$ vertices. %If
(Together with~$\bf(c)$,
this implies that if a pair of Type-$4$ clusters of $H$ have exactly one common vertex, then any other Type-$4$
cluster of $H$ shares at most one vertex with them.)

$\bf(f)$\ There are at most $M_i$  Type-$i$ clusters, for $1\leq i\leq 4$.

$\bf(g)$\ $\deg (v)\leq M_0$ for every vertex $v\in [n]$.

We further define  $\mathcal{H}_r^{++}(n,m)\subseteq\mathcal{H}_r^+(n,m)$
to be the set of $r$-graphs $H$
by replacing the property $\bf(g)$ with a stronger constraint $\bf(g^*)$.

$\bf(g^*)$\ $\deg (v)\leq M_0^*$ for every vertex $v\in [n]$.

\begin{remark}\label{r3.1}
From property $\bf(g)$, it is natural to obtain
\[
\sum_{v\in [n]} \binom{\deg (v)}{2}=O\Bigl(M_0\sum_{v\in [n]}\deg (v)\Bigr)=O\bigl(rmM_0\bigr)
=O\Bigl(rm\log (r^{-2}n)+ \frac{r^3m^2}{n}\Bigr)
\]
for $H\in\mathcal{H}_r^+(n,m)$.
\end{remark}

We now  show that the  expected number of $r$-graphs %$H$
 in $\mathcal{H}_r(n,m)$
not satisfying the properties of $\mathcal{H}^+_r(n,m)$ and  $\mathcal{H}^{++}_r(n,m)$ is quite small,
which implies that these $r$-graphs make asymptotically insignificant contributions.
The removal of these $r$-graphs from our main proof will lead
to some welcome simplifications.

\begin{theorem}\label{t3.2}
Suppose that $r^{-2}n\leq m=o(r^{-3}n^{ \frac32})$ and $n\to \infty$. Then
\[
 \frac{|\mathcal{H}^+_r(n,m)|}{|\mathcal{H}_r(n,m)|}=1-O\Bigl( \frac{r^6m^2}{n^3}\Bigr),\quad  \frac{|\mathcal{H}^{++}_r(n,m)|}{|\mathcal{H}_r(n,m)|}=1-O\Bigl( \frac{r^6m^2}{n^3}\Bigr).
\]
\end{theorem}

\begin{proof} Consider $H\in \mathcal{H}_r(n, m)$ chosen uniformly at random. It's
enough for us to prove the second equation. We apply Lemma~\ref{l2.2} several times
to show that  $H$ satisfies the properties $\bf(a)$-$\bf(g^*)$ with probability $1-O\bigl( \frac{r^6m^2}{n^3}\bigr)$.

$\bf(a)$\  Applying Lemma~\ref{l2.2} with $t=2$ and $\alpha=3$, the expected number of
two edges involving at most $2r-3$ vertices is $O\bigl( \frac{r^6m^2}{n^3}\bigr)$.
 Hence,  the property $\bf(a)$ holds with probability $1-O\bigl( \frac{r^6m^2}{n^3}\bigr)$.

$\bf(b)$\ Applying Lemma~\ref{l2.2} with $t=3$ and $\alpha=5$, the expected number of
three edges involving at most $3r-5$ vertices is
\[
O\Bigl( \frac{r^{10}m^3}{n^{5}}\Bigr)
=O\Bigl( \frac{r^6m^2}{n^3}\Bigr),
\]
where the last equality is true because $m=o(r^{-3}n^{ \frac32})$.

Similarly, applying Lemma~\ref{l2.2} with $t=4$ and $\alpha=6$,  the expected
number of four edges involving at most $4r-6$ vertices is
\[
O\Bigl( \frac{r^{12}m^4}{n^{6}}\Bigr)
=O\Bigl( \frac{r^6m^2}{n^3}\Bigr).
\]

If there is a cluster with four or more edges, then there must be
four edges involving at most $4r-6$ vertices.
Hence, every cluster contains at most three edges and the property
$\bf(b)$ holds with probability $1-O\bigl( \frac{r^6m^2}{n^3}\bigr)$.

$\bf(c)$\ Applying Lemma~\ref{l2.2}, the expected number of two clusters such
that their intersection contains
two or more vertices is
\[
O\Bigl( \frac{r^{20}m^6}{n^{10}}+ \frac{r^{16}m^5}{n^8}+ \frac{r^{12}m^4}{n^6}
\Bigr)=O\Bigl( \frac{r^6m^2}{n^3}\Bigr),
\]
where the first term  arises from the case between Type-$1$ (or Type-$2$
or Type-$3$) cluster and Type-$1$ (or Type-$2$ or Type-$3$)
cluster and there are at most $6r-10$ vertices if their intersection
contains at least two vertices, the second term arises from the case
between Type-$1$ (or Type-$2$ or Type-$3$) cluster and Type-$4$ cluster
and there are at most $5r-8$ vertices if their intersection contains
at least two vertices, while the last term arises from the case between
Type-$4$ cluster and Type-$4$ cluster and there are at most $4r-6$ vertices
if their intersection contains at least two vertices.
Hence, the property $\bf(c)$ holds with probability $1-O\bigl( \frac{r^6m^2}{n^3}\bigr)$.

$\bf(d)$\ The expected number of three Type-$1$, Type-$2$ or Type-$3$ clusters
involving at most $9r-14$ vertices is
\[
O\Bigl( \frac{r^{28}m^9}{n^{14}}\Bigr)
=O\Bigl( \frac{r^6m^2}{n^3}\Bigr).
\]
Hence, the property $\bf(d)$ holds with probability $1-O\bigl( \frac{r^6m^2}{n^3}\bigr)$.

$\bf(e)$\  The expected number of  three Type-$4$
clusters involving at most $6r-9$ vertices is
\[
O\Bigl( \frac{r^{18}m^6}{n^{9}}\Bigr)
=O\Bigl( \frac{r^6m^2}{n^3}\Bigr).
\]
Hence,   the property $\bf(e)$ holds with probability $1-O\bigl( \frac{r^6m^2}{n^3}\bigr)$.

$\bf(f)$\ Define four events $\mathcal{E}_i$ as
\[
 \mathcal{E}_i=\bigl\{\text{There are at most $M_i$ Type-$i$ clusters in $H$}\bigr\},
\]
and $\mathcal{\overline{E}}_i$ as the complement of
the event $\mathcal{E}_i$, where $1\leq i\leq 4$.
We show that $\mathbb{P}[\mathcal{E}_i]=1-O\bigl( \frac{r^6m^2}{n^3}\bigr)$
for $1\leq i\leq 4$.

Let
\[
Q_1=\biggl\lceil \log (r^{-2}n)+\frac{3^{4}r^8 m^3}{4n^4}\biggr\rceil
\]
and define $\ell_1=Q_1+1$. We first show that the expected number of
sets of $\ell_1$ vertex-disjoint Type-$1$ clusters in $H$ is $O\bigl( \frac{r^6}{n^3}\bigr)$.
These $\ell_1$ vertex-disjoint Type-$1$ clusters  contain  $3\ell_1$ edges and $(3r-4)\ell_1$
vertices. Note that Lemma~\ref{l2.2} is not appropriate here because $3\ell_1$ is not a constant.

By Lemma~\ref{l2.1}, it follows that the expected number of sets of  $\ell_1$
vertex-disjoint Type-$1$ clusters in $H$ is at most
\begin{align*}
 \binom{n}{(3r-4)\ell_1} &\binom{(3r-4)\ell_1}{3r-4,\ldots,3r-4} \frac{1}{\ell_1!}
\biggl[ \binom{3r-4}{r} \binom{r}{4} \binom{4}{2} \binom{2r-4}{r-2}\biggr]^{\ell_1}\biggl( \frac{m}{N}\biggr)^{3\ell_1}\\
&=O\biggl(\biggl( \frac{[r]_4[r]_2^2em^3}{4\ell_1n^4}\biggr)^{\ell_1}\biggr)
=O\Bigl(\Bigl( \frac{e}{3^4}\Bigr)^{\ell_1}\Bigr)
=O\Bigl( \frac{r^6}{n^3}\Bigr),
\end{align*}
where the first two equalities are true because $\ell_1!\geq \bigl( \frac{\ell_1}{e}\bigr)^{\ell_1}$
and $\ell_1> \frac{3^4r^8 m^3}{4n^4}$, and the
last equality is true because of the assumption $\ell_1>\log (r^{-2}n)$.

Assuming that property $\bf(d)$ holds, any Type-$1$ cluster is
 either vertex-disjoint from all other Type-$1$ clusters of $H$, or
 shares one vertex with precisely one
 other Type-$1$ cluster of~$H$, then it follows that
 $\mathbb{P}[\mathcal{\overline{E}}_1\mid\text{Property {\bf(d)} holds}]
 =O\bigl( \frac{r^6}{n^3}\bigr)$.
By the total probability formula, we have
\[
\mathbb{P}[\mathcal{\overline{E}}_1]
\leq\mathbb{P}[\text{Property {\bf(d)} does not hold}]+
\mathbb{P}[\mathcal{\overline{E}}_1\mid\text{Property {\bf(d)} holds}]
=O\Bigl( \frac{r^6m^2}{n^3}\Bigr)
\]
and $\mathbb{P}\left[\mathcal{E}_1\right]=1-O\bigl( \frac{r^6m^2}{n^3}\bigr)$.
Similarly, we  have $\mathbb{P}\left[\mathcal{E}_2\right]=1-O\bigl( \frac{r^6m^2}{n^3}\bigr)$, $\mathbb{P}\left[\mathcal{E}_3\right]=1-O\bigl( \frac{r^6m^2}{n^3}\bigr)$.

At last, we show that
$\mathbb{P}[\mathcal{E}_4]=1-O\bigl( \frac{r^6m^2}{n^3}\bigr)$ is also true.
Let $\ell_4=M_4+1$. Let $\{x_i,y_i\}\in  \binom{[n]}{2}$ be a set of $\ell_4$ links
with edges $e_i$ and $e_i'$, here $i\in\{1,\ldots,\ell_4\}$.
These $\ell_4$ links  are called \textit{paired-distinct} if these $2\ell_4$ edges
are all distinct. Assuming the property $\bf(b)$ holds, note that the number of
Type-$4$ clusters is no greater than the number of paired-distinct links.
Define a event $\mathcal{E}_4'$ as
\[
\mathcal{E}_4'=\bigl\{\text{There are at most $M_4$ paired-distinct links in $H$}\bigr\},
\]
and $\mathcal{\overline{E}}_4'$ as the complement of the event $\mathcal{E}_4'$.
We firstly show that $\mathbb{P}[\mathcal{\overline{E}}_4']=O\bigl( \frac{r^6}{n^3}\bigr)$ by
\[
\mathbb{P}\bigl[\mathcal{\overline{E}}_4'\bigr]
=O\biggl( \binom{n}{r-2}^{2\ell_4} \binom{ \binom{n}{2}}{\ell_4}\biggl( \frac{m}{N}\biggr)^{2\ell_4}\biggr)
=O\biggl(\Bigl( \frac{r^4e m^2}{2\ell_4n^2}\Bigr)^{\ell_4}\biggr)
=O\Bigl( \frac{r^6}{n^3}\Bigr),
\]
where the last two equalities are true because $\ell_4> \frac{3^4r^4m^2}{2n^2}$
and $\ell_4>\log(r^{-2}n)$.
Then it follows that $\mathbb{P}[\mathcal{\overline{E}}_4\mid\text{Property {\bf(b)}
 holds}]\leq\mathbb{P}[\mathcal{\overline{E}}_4']=O\bigl( \frac{r^6}{n^3}\bigr)$.
By the law of total probability,
\[
\mathbb{P}[\mathcal{\overline{E}}_4]
\leq\mathbb{P}[\text{Property {\bf(b)} does not hold}]+
\mathbb{P}[\mathcal{\overline{E}}_4\mid\text{Property {\bf(b)} holds}]
=O\Bigl( \frac{r^6m^2}{n^3}\Bigr).
\]

$\bf(g^*)$\  Define $d=M_0^*+1$.
The expected number of vertices $v$ such that $\deg (v)\geq d$ is
\[
n \binom{n-1 }{r-1}^d \frac{1}{d!}\Bigl( \frac{m}{N}\Bigr)^{d}
=O\Bigl(n\Bigl( \frac{emr}{dn}\Bigr)^{d}\Bigr)
=O\Bigl(n\Bigl( \frac{e}{3^4r}\Bigr)^{d}\Bigr)
=O\Bigl( \frac{r^6}{n^3}\Bigr),
\]
where the second equality is true because $d!\geq \bigl( \frac{d}{e}\bigr)^{d}$,
and the last equality is true because of the assumption $d>M_0^*$ made in~\eqref{e3.1}.
Thus, there are no  vertices with degree at least $d$
 in $H$ holds with  probability $1-O\bigl( \frac{r^6}{n^3}\bigr)$.

This completes the proof of Theorem~\ref{t3.2}.
\end{proof}

\begin{remark}\label{r3.3}
For nonnegative integers $h_1$,
$h_2$, $h_3$ and $h_4$, define $\mathcal{C}_{h_1,h_2,h_3,h_4}^{+}$
to be the set of $r$-graphs in $H\in \mathcal{H}_r^+(n,m)$ with
exactly $h_i$ clusters of Type~$i$, for $1\leq i\leq 4$.
Define $\mathcal{C}_{h_1,h_2,h_3,h_4}^{++}$ similarly.
By the definitions of $H\in \mathcal{H}_r^+(n,m)$ and
$H\in \mathcal{H}_r^{++}(n,m)$ we have
\begin{align*}
|\mathcal{H}_r^+(n,m)|&=\sum_{h_4=0}^{M_4}\sum_{h_3=0}^{M_3}
\sum_{h_2=0}^{M_2}\sum_{h_1=0}^{M_1}|\mathcal{C}_{h_1,h_2,h_3,h_4}^{+}|,\\
|\mathcal{H}_r^{++}(n,m)|&=\sum_{h_4=0}^{M_4}\sum_{h_3=0}^{M_3}
\sum_{h_2=0}^{M_2}\sum_{h_1=0}^{M_1}|\mathcal{C}_{h_1,h_2,h_3,h_4}^{++}|.
\end{align*}
\end{remark}

We will estimate the relative sizes of these subsets by means of
switching operations.  The following is an essential tool that we
will use repeatedly.

\begin{lemma}\label{l3.5}
Assume that $r^{-2}n\leq m=o(r^{-3}n^{ \frac32})$.
Let $H\in \mathcal{H}_r^{+}(n,m-\xi)$
and $N_t$ be the set of $t$-sets of $[n]$ of which no two vertices belong to
the same edge of $H$, where $r\leq t\leq 3r-4$ and $\xi=O(1)$.
Suppose that $n\to\infty$. Then
\[
|N_t|=\biggl[ \binom{n}{t}- \binom{r}{2}m \binom{n-2}{t-2}\biggr]
  \biggl(1+O\Bigl( \frac{r^6mn\log (r^{-2}n)+r^8m^2}{n^{4}}\Bigr)\biggr).
\]
\end{lemma}

\begin{proof} We will use inclusion-exclusion.  Let $A_{e(i,j)}$ be the
event that a $t$-set of $[n]$ contains two vertices $i$ and $j$ of the edge $e$.
Thus, we have
\[
 \binom{n}{t}-\sum_{\{e,\{i,j\}\}}|A_{e(i,j)}|\leq|N_t|\leq \binom{n}{t}-\sum_{\{e,\{i,j\}\}}|A_{e(i,j)}|+\sum_{\{e,\{i,j\}\}\neq \{e',\{i',j'\}\}}|A_{e(i,j)}\cap A_{e'(i',j')}|.
\]
Clearly, $|A_{e(i,j)}|= \binom{n-2}{t-2}$ for each edge $e$
and $\{i,j\}\subset e$.   We have $|N_t|\geq \binom{n}{t}- \binom{r}{2}(m-\xi) \binom{n-2}{t-2}$.

Now we consider the upper bound. For the case $e=e'$, we have
\begin{align}\label{e3.2}
\sum_{\{e,\{i,j\}\}\neq \{e',\{i',j'\}\}} & |A_{e(i,j)}\cap A_{e'(i',j')}| \notag \\
&=(m-\xi) \binom{r}{3} \binom{3}{2} \binom{n-3}{t-3}+ \frac12(m-\xi) \binom{r}{4} \binom{4}{2} \binom{n-4}{t-4} \notag \\
&=O\biggl((m-\xi)r^3 \binom{n-3}{t-3}\biggr).
\end{align}
For the case $e\neq e'$ and $\{i,j\}\cap \{i',j'\}=\emptyset$, we have
\begin{align}\label{e3.3}
\sum_{\{e,\{i,j\}\}\neq \{e',\{i',j'\}\}} & |A_{e(i,j)}\cap A_{e'(i',j')}| \notag\\
&=O\biggl( \binom{m-\xi}{2} \binom{r}{2}^2 \binom{n-4}{t-4}\biggr) \notag\\
&=O\biggl((m-\xi)^2r^4 \binom{n-4}{t-4}\biggr).
\end{align}
For the case $e\neq e'$ and $|\{i,j\}\cap \{i',j'\}|=1$,  by Remark~\ref{r3.1}, we have
\begin{align}\label{e3.4}
\sum_{\{e,\{i,j\}\}\neq \{e',\{i',j'\}\}} & |A_{e(i,j)}\cap A_{e'(i',j')}| \notag\\
&=\sum_{v\in [n]} \binom{\deg(v)}{2}(r-1)^2 \binom{n-3}{t-3} \notag\\
&=O\biggl(\Bigl( \frac{r^3(m-\xi)n\log (r^{-2}n) +r^5(m-\xi)^2}{n}\Bigr) \binom{n-3}{t-3}\biggr).
\end{align}
For the case $e\neq e'$ and $\{i,j\}=\{i',j'\}$, since there are at most $M_s$ Type-$s$
clusters with $1\leq s\leq 4$ in $H$, as the equation shown in~\eqref{e3.1}, we have
\begin{align}\label{e3.5}
\sum_{\{e,\{i,j\}\}\neq \{e',\{i',j'\}\}} & |A_{e(i,j)}\cap A_{e'(i',j')}| \notag\\
&=O\biggl(\sum_{s=1}^4M_s \binom{n-2}{t-2}\biggr) \notag\\
&=O\biggl(\Bigl(\log (r^{-2}n)+ \frac{r^4(m-\xi)^2}{n^2}\Bigr) \binom{n-2}{t-2}\biggr).
\end{align}

By equations~\eqref{e3.2}--\eqref{e3.5} and the assumption $m\geq r^{-2}n$, we have
\[
\sum_{\{e,\{i,j\}\}\neq \{e',\{i',j'\}\}}|A_{e(i,j)}\cap A_{e'(i',j')}|
=O\biggl(\Bigl( \frac{r^3mn\log (r^{-2}n) +r^5m^2}{n}\Bigr) \binom{n-3}{t-3}\biggr).
\]
Thus, we complete the proof of Lemma~\ref{l3.5} by
\[
 \frac{ \binom{n-3}{t-3}}{ \binom{n}{t}}= \frac{[t]_3}{[n]_3}=O\Bigl( \frac{r^3}{n^3}\Bigr),
 \frac{ \binom{r}{2}\xi \binom{n-2}{t-2}}{ \binom{n}{t}}=O\Bigl( \frac{r^4}{n^2}\Bigr)
=O\Bigl( \frac{r^8m^2}{n^4}\Bigr),
\]
because $r\leq t\leq 3r-4$, $m\geq r^{-2}n$ and $n\to\infty$.
\end{proof}

\begin{lemma}\label{l3.6}
Assume $r^{-2}n\leq m=o(r^{-3}n^{ \frac32})$.
Let $H\in\mathcal{H}_r^{+}(n,m-\xi)$ and $N_{2r-2}'$ be  the set of $(2r-2)$-sets of $[n]$
of which exactly two vertices belong to the same edge of $H$, where $\xi=O(1)$.
Suppose that $n\to\infty$. Then
\[
|N_{2r-2}'|= \binom{r}{2}m \binom{n-2}{2r-4}
\biggl(1+O\Bigl( \frac{r^2n\log(r^{-2}n)+r^4m}{n^2}\Bigr)\biggr).
\]
\end{lemma}

\begin{proof} It is clear that
\[
\sum_{\{e,\{i,j\}\}}|A_{e(i,j)}|-\sum_{\{e,\{i,j\}\}
\neq \{e',\{i',j'\}\}}|A_{e(i,j)}\cap A_{e'(i',j')}|\leq |N_{2r-2}'|\leq \sum_{\{e,\{i,j\}\}}|A_{e(i,j)}|.
\]
Therefore, as shown in the proof of Lemma~\ref{l3.5}, we have
\begin{align*}
|N_{2r-2}'|&\leq \binom{r}{2}(m-\xi) \binom{n-2}{2r-4},\\
|N_{2r-2}'|&\geq \binom{r}{2}(m-\xi) \binom{n-2}{2r-4}-
O\biggl(\Bigl( \frac{r^3mn\log (r^{-2}n) +r^5m^2}{n}\Bigr) \binom{n-3}{2r-5}\biggr).
\end{align*}
We complete the proof by noting that
$\frac{ \binom{n-3}{2r-5}}{ \binom{r}{2}m \binom{n-2}{2r-4}}
=O\bigl( \frac{1}{rmn}\bigr)$ and $ \frac{\xi}{m}=O\bigl( \frac{r^4m}{n^2}\bigr)$.
\end{proof}

The switching method relies on the fact that the ratio of the sizes of the two
parts of a bipartite graph is reciprocal to the ratio of their average degrees.
For our purposes we need a generalization as given in the following lemma.

\begin{lemma}\label{l3.7}
Let $G$ be a bipartite graph with vertex sets $A$ and $B$, where $A=A_1\cup A_2$ with $A_1\cap A_2=\emptyset$ and
 $B=B_1\cup B_2$ with $B_1\cap B_2=\emptyset$. Let
$d_{min}^{A_1}$ and $d_{min}^{B_1}$ be the minimum degrees of vertices in $A_1$ and $B_1$, respectively.
Let $d_{max}^A$ and $d_{max}^B$ be the maximum degrees of vertices in $A$ and $B$, respectively. Then
\[
 \frac{d_{min}^{B_1}}{d_{max}^A}\biggl(1+ \frac{|A_2|}{|A_1|}\biggr)^{\!\!-1}
 \leq  \frac{|A_1|}{|B_1|}
 \leq  \frac{d_{max}^B}{d_{min}^{A_1}}\biggl(1+ \frac{|B_2|}{|B_1|}\biggr).
\]
\end{lemma}

\begin{proof} Let $E$ be the set of  edges between $A_1$ and $B_1$ in $G$. We have
\begin{align*}
 |A_1|d_{min}^{A_1}-|B_2|d_{max}^B \leq |E| &\leq |A_1|d_{max}^A,~~\text{and}\\
|B_1|d_{min}^{B_1}-|A_2|d_{max}^A \leq |E| &\leq |B_1|d_{max}^B.
\end{align*}
Combining these inequalities, we have
\[
 \frac{|A_1|d_{min}^{A_1}-|B_2|d_{max}^B}{|B_1|d_{max}^B}\leq1
\]
which gives the upper bound on $ \frac{|A_1|}{|B_1|}$, and
\[
 \frac{|B_1|d_{min}^{B_1}-|A_2|d_{max}^A}{|A_1|d_{max}^A}\leq1
\]
which gives the lower bound.
\end{proof}

\section{Partitions in $\mathcal{H}_r^{+}(n,m)$ and $\mathcal{H}_r^{++}(n,m)$}\label{s:4}

We firstly show that $|\mathcal{C}_{0,0,0,0}^{++}|$ and
 $|\mathcal{L}_r(n,m)|$ are almost equal.

\begin{theorem}\label{t4.1}
Assume that $r^{-2}n\leq m=o(r^{-3}n^{ \frac32})$
and $n\to\infty$. Then
\[
|\mathcal{C}_{0,0,0,0}^{++}|=\Bigl(1-O\Bigl( \frac{r^6}{n^3}\Bigr)\Bigr)|\mathcal{L}_r(n,m)|.
\]
\end{theorem}

\begin{proof}
Consider $H\in \mathcal{L}_r(n,m)$ chosen uniformly at random.
We show that  there are no vertices with degree greater
than $M_0^*$ with probability $1-O\bigl( \frac{r^6}{n^3}\bigr)$.
Fix $v\in[n]$.

Assume $\deg (v)=d$ for some integer $d\geq 1$. Define
the set $\mathcal{S}_v(d\to d-1)$ to be the set of switching operations
that consist of removing one edge containing $v$ and placing it somewhere
else such that it doesn't contain $v$ but the linearity property is preserved.

Applying Lemma~\ref{l3.5} to $H-\{v\}$, we have
$|\mathcal{S}_v(d\to d-1)|\geq d\bigl[ \binom{n-1}{r}- \binom{r}{2}m \binom{n-3}{r-2}\bigr]$.
In the other direction, assume $\deg (v)=d-1$  for some integer
$d\geq 1$. Define the set $\mathcal{S}_v(d-1\to d)$ as the set of
operations inverse to $\mathcal{S}_v(d\to d-1)$.
Clearly, $|\mathcal{S}_v(d-1\to d)|\leq m \binom{n-1}{r-1}$.

Thus, we have
\begin{align*}
 \frac{\mathbb{P}[\deg (v)=d]}{\mathbb{P}[\deg (v)=d-1]}
&= \frac{|\mathcal{S}_v(d-1\to d)|}{|\mathcal{S}_v(d\to d-1)|}\\
&\leq  \frac{m \binom{n-1}{r-1}}{d\bigl[ \binom{n-1}{r}- \binom{r}{2}m \binom{n-3}{r-2}\bigr]}\\
&= \frac{rm}{dn}\Bigl(1+O\Bigl( \frac{r^4m}{n^2}\Bigr)\Bigr)
< \frac{2rm}{dn},
\end{align*}
where the last equality is true because $r^{-2}n\leq m=o(r^{-3}n^{ \frac32})$.
If $d>M_0^*$, following the recursive relation as above, then we have
\[
\mathbb{P}[\deg (v)=d]\leq \prod_{i=4\lceil \frac{r^2m}{n}\rceil+1}^d\Bigl( \frac{2rm}{in}\Bigr)
\mathbb{P}\Bigl[\deg (v)=4\Bigl\lceil \frac{r^2m}{n}\Bigr\rceil\Bigr]
\leq\Bigl( \frac{1}{2r}\Bigr)^{d-4\bigl\lceil \frac{r^2m}{n}\bigr\rceil}.
\]
The probability that $\deg (v)>M_0^*$ is
\[
\mathbb{P}[\deg (v)>M_0^*]\leq \sum_{d>M_0^*}\Bigl( \frac{1}{2r}\Bigr)^{d-4
\bigl\lceil \frac{r^2m}{n}\bigr\rceil}
<2\Bigl( \frac{1}{2r}\Bigr)^{M_0^*-4\bigl\lceil \frac{r^2m}{n}\bigr\rceil}
=O\Bigl( \frac{r^6}{n^4}\Bigr).
\]
The probability that  there is  a vertex with degree greater than $M_0^*$ is
\[
\mathbb{P}\bigl[\exists_{v\in[n]}:\deg (v)>M_0^*\bigr]
 \leq n\mathbb{P}\bigl[\deg (v)>M_0^*\bigr]
=O\Bigl(\frac{r^6}{n^3}\Bigr),
\]
to complete the proof.
\end{proof}

Next we show that
$|\mathcal{C}_{h_1,h_2,h_3,h_4}^{+}|$
 and $|\mathcal{C}_{h_1,h_2,h_3,h_4}^{++}|$
are almost equal if $h_i\leq M_i$ for $1\leq i\leq 4$.
%For nonnegative integers $h_1$, $h_2$, $h_3$ and $h_4$, we show
%$\mathcal{C}_{h_1,h_2,h_3,h_4}^{+}$
% and $\mathcal{C}_{h_1,h_2,h_3,h_4}^{++}$ are equal with high probability.
Though the proof  is similar to that of Theorem~\ref{t4.1}, its switching
operations are different in order to keep the number of Type-$i$ clusters
intact for $1\leq i\leq 4$.

\begin{theorem}\label{t4.2}
Assume $r^{-2}n\leq m=o(r^{-3}n^{ \frac32})$ and that
$0\leq h_i\leq M_i$ for $1\leq i\leq 4$. Suppose that $n\to\infty$. Then
\[
|\mathcal{C}_{h_1,h_2,h_3,h_4}^{++}|=
\Bigl(1-O\Bigl( \frac{r^6}{n^3}\Bigr)\Bigr)|\mathcal{C}_{h_1,h_2,h_3,h_4}^{+}|.
\]
\end{theorem}

\begin{proof}
Consider $H\in \mathcal{C}_{h_1,h_2,h_3,h_4}^{+}$ chosen uniformly at random.
We will show that  there are no vertices with degree
greater than $M_0^*$ with probability $1-O\bigl( \frac{r^6}{n^3}\bigr)$.

Fix $v\in[n]$. Assume that $\deg (v)=d$ for some integer
$d\geq 1$. Let $d_i$ be the number of Type-$i$ clusters containing $v$.
By the property $\bf(c)$ for $\mathcal{H}_{r}^{+}(n,m)$, the
intersection of every two clusters here is only the vertex $v$. Let $d_0=d-d_1-d_2-d_3-d_4$.
Define the set $\mathcal{S}_v(d\to d-1)$
as a set of switching strategies to decrease the degree of $v$ from $d$
to $d-1$ while keeping the number of Type-$i$ clusters unchanged for
$1\leq i\leq 4$ in $H$.  Each strategy involves moving one edge or one cluster.
%We switch according to the following ways.
If we choose an edge containing $v$ in a Type-$1$, Type-$2$ or Type-$3$
cluster, then we switch this cluster to a $(3r-4)$-set of $[n]-\{v\}$ with
no two vertices in the same edge of $H-\{v\}$; if we choose an
edge containing $v$ in a Type-$4$ cluster, then we switch this cluster to
 a $(2r-2)$-set of $[n]-\{v\}$ with no two vertices in the same
edge of $H-\{v\}$; otherwise we switch the edge  containing $v$ to
an $r$-set of $[n]-\{v\}$ with no two vertices in the same edge
of $H-\{v\}$. Applying Lemma~\ref{l3.5} to $H-\{v\}$, we have
\begin{align*}
|\mathcal{S}_v(d\to d-1)|&\geq d_0\biggl[ \binom{n-1}{r}- \binom{r}{2}m \binom{n-3}{r-2}\biggr]\\
&{\qquad}+(d_1+d_2+d_3)\biggl[ \binom{n-1}{3r-4}- \binom{r}{2}m \binom{n-3}{3r-6}\biggr]\\
&{\qquad}+d_4\biggl[ \binom{n-1}{2r-2}- \binom{r}{2}m \binom{n-3}{2r-4}\biggr]\\
&>d\biggl[ \binom{n-1}{r}- \binom{r}{2}m \binom{n-3}{r-2}\biggr].
\end{align*}
Assume $\deg (v)=d-1$  for some integer $d\geq 1$. Define the
set $\mathcal{S}_v(d-1\to d)$ as the switching
 strategies inverse to the above to increase the degree of $v$ from $d-1$ to $d$.
 Clearly, $|\mathcal{S}_v(d-1\to d)|\leq m \binom{n-1}{r-1}$.

Thus, we have
\[
 \frac{\mathbb{P}[\deg (v)=d]}{\mathbb{P}[\deg (v)=d-1]}
= \frac{|\mathcal{S}_v(d-1\to d)|}{|\mathcal{S}_v(d\to d-1)|}
<  \frac{m \binom{n-1}{r-1}}{d\bigl[ \binom{n-1}{r}- \binom{r}{2}m \binom{n-2}{r-2}\bigr]}
= \frac{rm}{dn}\Bigl(1+O\Bigl( \frac{r^4m}{n^2}\Bigr)\Bigr),
\]
after which we complete the proof as in Theorem~\ref{t4.1}.
\end{proof}

\begin{remark}\label{r4.3}
From the proof of Theorem~\ref{t4.2}, for any given $h_1,\ldots,h_4$,
we have shown that $|\mathcal{C}_{h_1,h_2,h_3,h_4}^{+}|=0$ iff
$|\mathcal{C}_{h_1,h_2,h_3,h_4}^{++}|=0$.
%While if $|\mathcal{C}_{h_1,h_2,h_3,h_4}^{+}|\neq0$,
%then $|\mathcal{C}_{h_1,h_2,h_3,h_4}^{++}|\neq0$.
Similarly, from the proof of Theorem~\ref{t4.1}, we have shown if $|\mathcal{L}_r(n,m)|\neq0$,
then $|\mathcal{C}_{0,0,0,0}^{++}|\neq0$.
\end{remark}

\begin{remark}\label{r4.4}
In fact, by Theorem~\ref{t3.2}, we have $|\mathcal{H}^+_r(n,m)|\neq0$.
There exist $h_i$ with  $1\leq i\leq 4$ such that $|\mathcal{C}_{h_1,h_2,h_3,h_4}^{+}|\neq0$.
By the switching operations in Section~\ref{s:5} below, we obtain $|\mathcal{L}_r(n,m)|\neq0$.
\end{remark}

\begin{remark}\label{r4.5}
By Theorem~\ref{t3.2}, Remark~\ref{r3.3}, Theorem~\ref{t4.1}, Remark~\ref{r4.3} and Remark~\ref{r4.4}, it follows that
\begin{align*}
 \frac{1}{\mathbb{P}_r(n,m)}&= \biggl(1-O\Bigl( \frac{r^6m^2}{n^3}\Bigr)\biggr)
 \sum_{h_4=0}^{M_4}\sum_{h_3=0}^{M_3}\sum_{h_2=0}^{M_2}\sum_{h_1=0}^{M_1}
 \frac{|\mathcal{C}_{h_1,h_2,h_3,h_4}^{++}|}
{|\mathcal{L}_r(n,m)|}\\
&= \biggl(1-O\Bigl( \frac{r^6m^2}{n^3}\Bigr)\biggr)
 \sum_{h_4=0}^{M_4}\sum_{h_3=0}^{M_3}\sum_{h_2=0}^{M_2}\sum_{h_1=0}^{M_1} \frac{
|\mathcal{C}_{h_1,h_2,h_3,h_4}^{++}|}{|\mathcal{C}_{0,0,0,0}^{++}|}.
\end{align*}
We estimate the above sum using switching operations designed
to remove these Type-$i$ clusters for $1\leq i\leq 4$ in the next two sections.
\end{remark}

\section{Switchings on $r$-graphs in $\mathcal{H}_r^{+}(n,m)$}\label{s:5}

Now our task is reduced to calculating the ratio
$|\mathcal{C}_{h_1,h_2,h_3,h_4}^{++}|/|\mathcal{C}_{0,0,0,0}^{++}|$
when $0\leq h_i\leq M_i$ for $1\leq i\leq 4$.

\subsection{Switchings of Type-$1$ clusters}\label{s:5.1}

Let $H\in \mathcal{C}_{h_1,h_2,h_3,h_4}^{+}$. A \textit{Type-$1$ switching}
from $H$ is used to reduce the number of Type-$1$ clusters in $H$,
which is defined in the following four steps.

\noindent{\bf Step 0.}\  Take a Type-$1$ cluster
$\{e,f,g\}$ and remove it from~$H$.
Define $H_0$ with the same vertex set $[n]$ and the edge
set $E(H_0)=E(H)\backslash \{e,f,g\}$.

\noindent{\bf Step 1.}\ Take any $r$-set from $[n]$ of which
no two vertices belong to the same edge of $H_0$ and %
add it as a new edge.  The new graph is denoted by $H'$.

\noindent{\bf Step 2.}\ Insert another new edge at an $r$-set
of $[n]$ of which no two vertices belong to the same edge of $H'$.
The resulting graph is denoted by~$H''$.

\noindent{\bf Step 3.}\ Insert an edge at an $r$-set of which no
two vertices belong to
the same edge of $H''$. The resulting graph is denoted by $H'''$.

A Type-$1$ switching operation is illustrated in Figure~\ref{fig:2} below.
\begin{figure}[!htb]
\centering
\includegraphics[width=0.9\textwidth]{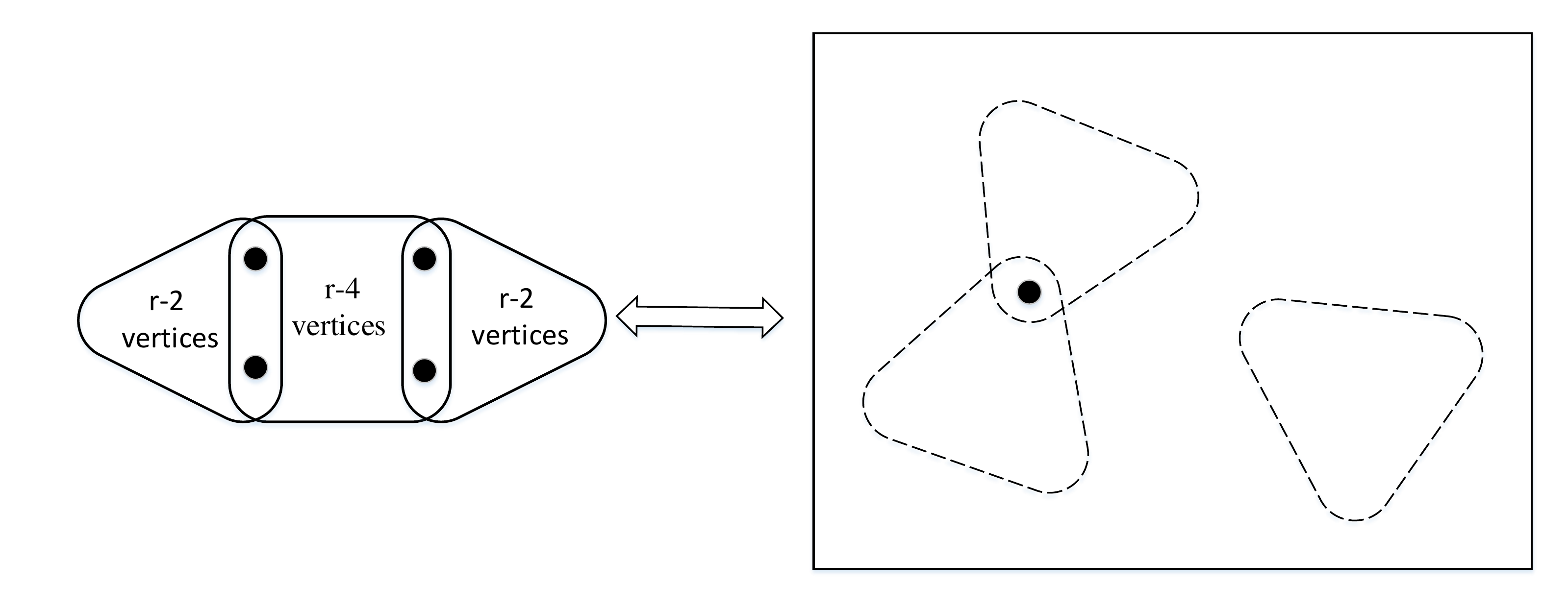}
\caption{An example of a Type-$1$ switching between $H$ and $H'''$\label{fig:2}}
\end{figure}
Note that any two new edges may or may not have a vertex in common.

\begin{remark}\label{r5.1}
A Type-$1$ switching reduces the number of Type-$1$ clusters in $H$ by
one without changing the other types of clusters.
Moreover, conditions $\bf(a)$--$\bf(f)$ remain true.
Since a vertex might gain degree during Steps 1--3, a Type-1 switching
does not necessarily map $\mathcal{C}_{h_1,h_2,h_3,h_4}^{+}$ into
$\mathcal{C}_{h_1-1,h_2,h_3,h_4}^{+}$ or map $\mathcal{C}_{h_1,h_2,h_3,h_4}^{++}$ into
$\mathcal{C}_{h_1-1,h_2,h_3,h_4}^{++}$.
However, it always maps $\mathcal{C}_{h_1,h_2,h_3,h_4}^{++}$ into
$\mathcal{C}_{h_1-1,h_2,h_3,h_4}^{+}$.
\end{remark}

A \textit{reverse Type-$1$ switching} is the reverse of a  Type-$1$ switching.
A reverse Type-$1$ switching from
$H'''\in \mathcal{C}_{h_1-1,h_2,h_3,h_4}^{+}$
is defined by sequentially removing  three edges of $H'''$ not containing a link, then
choosing a $(3r-4)$-set $T$ from $[n]$
of which no two vertices belong to any remaining  edges of $H'''$,
then inserting three edges into $T$ such that they
create a Type-$1$ cluster.
This operation is depicted in Figure~\ref{fig:2} by following the arrow in reverse.

\begin{remark}\label{r5.2}
If $H'''\in \mathcal{C}_{h_1-1,h_2,h_3,h_4}^{++}$, then a reverse
Type-$1$ switching from $H'''$ may also violate the condition $\bf(g^*)$
for $\mathcal{H}_r^{++}(n,m)$, but the resulting graph~$H$ is in
$\mathcal{C}_{h_1,h_2,h_3,h_4}^{+}$
because $\deg (v)\leq M_0^*+2<M_0$ for every vertex $v\in[n]$ of $H$.
\end{remark}

 Next we analyze Type-$1$ switchings to find a  relationship between the sizes of
 $\mathcal{C}_{h_1,h_2,h_3,h_4}^{++}$ and
 $\mathcal{C}_{h_1-1,h_2,h_3,h_4}^{++}$.

\begin{lemma}\label{l5.3}
Assume $r^{-2}n\leq m=o(r^{-3}n^{ \frac32})$ and
that $0\leq h_i\leq M_i$ for $1\leq i\leq 4$ with $h_1\geq 1$.\\
$(a)$\ Let $H\in \mathcal{C}_{h_1,h_2,h_3,h_4}^{+}$.
Then the number of Type-$1$ switchings for $H$ is
\[
h_1\biggl[N- \binom{r}{2}m \binom{n-2}{r-2}\biggr]^3
\biggl(1+O\Bigl( \frac{r^6mn\log (r^{-2}n)+r^8m^2}{n^{4}}\Bigr)\biggr).
\]
$(b)$\ Let $H'''\in \mathcal{C}_{h_1-1,h_2,h_3,h_4}^{+}$. The number of
reverse Type-$1$ switchings for $H'''$ is
\begin{align*}
&36 \binom{3r-4}{r} \binom{r}{4} \binom{2r-4}{r-2} \binom{m-3(h_1-1)-3h_2-3h_3-2h_4}{3}\\
&{\qquad}\times\biggl[ \binom{n}{3r-4}- \binom{r}{2}m \binom{n-2}{3r-6}\biggr]
\biggl(1+O\Bigl( \frac{r^6mn\log (r^{-2}n)+r^8m^2}{n^{4}}\Bigr)\biggr).
\end{align*}
\end{lemma}

\begin{proof} $(a)$\ Let $H\in \mathcal{C}_{h_1,h_2,h_3,h_4}^{+}$.
Let $\mathcal{S}(H)$ be set the of all Type-$1$ switchings which can
be applied to $H$. There are exactly $h_1$ ways to choose a Type-$1$ cluster.
In each of the steps 2--4 of the switching, by Lemma~\ref{l3.5}, there are
\[
\biggl[N- \binom{r}{2}m \binom{n-2}{r-2}\biggr]
\biggl(1+O\Bigl( \frac{r^6mn\log (r^{-2}n)+r^8m^2}{n^{4}}\Bigr)\biggr)
\]
ways to choose the new edge.
Thus, we have
\[
\bigl|\mathcal{S}(H)\bigl|=h_1\biggl[N- \binom{r}{2}m \binom{n-2}{r-2}\biggr]^3
\biggl(1+O\Bigl( \frac{r^6mn\log (r^{-2}n)+r^8m^2}{n^{4}}\Bigr)\biggr).
\]

$(b)$\ Conversely, suppose that $H'''\in \mathcal{C}_{h_1-1,h_2,h_3,h_4}^{+}$.
Similarly, let $\mathcal{S}'(H''')$ be the set of all reverse Type-$1$ switchings for $H'''$.
There are exactly $6 \binom{m-3(h_1-1)-3h_2-3h_3-2h_4}{3}$ ways to
delete three edges in sequence such that
none of them contain a link. By Lemma~\ref{l3.5}, there are
\[
\biggl[ \binom{n}{3r-4}- \binom{r}{2}m \binom{n-2}{3r-6}\biggr]
  \biggl(1+O\Bigl( \frac{r^6mn\log (r^{-2}n)+r^8m^2}{n^{4}}\Bigr)\biggr)
\]
ways to choose a $(3r-4)$-set $T$ from $[n]$ of which
no two vertices belong to any remaining edges of $H'''$. For every $T$,
there are $ \binom{3r-4}{r} \binom{r}{4} \binom{4}{2} \binom{2r-4}{r-2}$
ways to create a Type-$1$ cluster in $T$. Thus, we have
\begin{align*}
|\mathcal{S}'(H''')|&=36 \binom{3r-4}{r} \binom{r}{4} \binom{2r-4}{r-2} \binom{m-3(h_1-1)-3h_2-3h_3-2h_4}{3}\\
&{\qquad}\times\biggl[ \binom{n}{3r-4}- \binom{r}{2}m \binom{n-2}{3r-6}\biggr]\biggl(1+O\Bigl( \frac{r^6mn\log (r^{-2}n)+r^8m^2}{n^{4}}\Bigr)\biggr).\qedhere
\end{align*}
\end{proof}

\begin{corollary}\label{c5.4}
With notation as above,
for $0\leq h_i\leq M_i$ when $1\leq i\leq 4$, the following hold:\\
$(a)$\ If $\mathcal{C}_{h_1,h_2,h_3,h_4}^{++}=\emptyset$, then $\mathcal{C}_{h_1+1,h_2,h_3,h_4}^{++}=\emptyset$.
\\
$(b)$\
Let $h_1'=h_1'(h_2,h_3,h_4)$ be the first value of $h_1\leq M_1$
such that $\mathcal{C}_{h_1,h_2,h_3,h_4}^{++}=\emptyset$,
or $h_1'=M_1+1$ if no such value exists. Suppose that $n\to\infty$.
Then uniformly for $1\leq h_1< h_1'$,
\begin{align*}
 \frac{|\mathcal{C}_{h_1,h_2,h_3,h_4}^{++}|}{|\mathcal{C}_{h_1-1,h_2,h_3,h_4}^{++}|}&=
 \frac{36 \binom{3r-4}{r} \binom{r}{4} \binom{2r-4}{r-2} \binom{m-3(h_1-1)-3h_2-3h_3-2h_4}{3}
\bigl[ \binom{n}{3r-4}- \binom{r}{2}m \binom{n-2}{3r-6}\bigr]}{h_1\bigl[N- \binom{r}{2}m \binom{n-2}{r-2}\bigr]^3}\\
&{\qquad}\times\biggl(1+O\Bigl( \frac{r^6mn\log (r^{-2}n)+r^8m^2}{n^{4}}\Bigr)\biggr).
\end{align*}
\end{corollary}

\begin{proof} $(a)$\ Suppose $\mathcal{C}_{h_1+1,h_2,h_3,h_4}^{++}\neq\emptyset$
and let $H\in\mathcal{C}_{h_1+1,h_2,h_3,h_4}^{++}$.
We apply a Type-$1$ switching from $H$ to obtain an $r$-graph
$H'''\in\mathcal{C}_{h_1,h_2,h_3,h_4}^{+}$. By Remark~\ref{r4.3}, we have
$|\mathcal{C}_{h_1,h_2,h_3,h_4}^{++}|\neq0$.

$(b)$\ By $(a)$,  if $h_1'<M_1$ such that $\mathcal{C}_{h_1',h_2,h_3,h_4}^{++}=\emptyset$, then
$$\mathcal{C}_{h_1'+1,h_2,h_3,h_4}^{++},\ldots,\mathcal{C}_{M_1,h_2,h_3,h_4}^{++} =\emptyset.$$
By the definition of $h_1'$, the left hand ratio is well defined.
By Remark~\ref{r5.1} and Remark~\ref{r5.2}, we take
\begin{align*}
A_1&=\mathcal{C}_{h_1,h_2,h_3,h_4}^{++},&
A_2&=\mathcal{C}_{h_1,h_2,h_3,h_4}^{+}-\mathcal{C}_{h_1,h_2,h_3,h_4}^{++},\\
B_1&=\mathcal{C}_{h_1-1,h_2,h_3,h_4}^{++},&
B_2&=\mathcal{C}_{h_1-1,h_2,h_3,h_4}^{+}-\mathcal{C}_{h_1-1,h_2,h_3,h_4}^{++},
\end{align*}
in Lemma~\ref{l3.7}.
By Theorem~\ref{t4.2}, we have $ \frac{|B_2|}{|B_1|}=O\bigl( \frac{r^6}{n^3})$
and $ \frac{|A_2|}{|A_1|}=O\bigl( \frac{r^6}{n^3}\bigr)$. By Lemma~\ref{l5.3}, we also have
\[
 \frac{d_{min}^{B_1}}{d_{max}^A}= \frac{\min_{H'''\in B_1}|{\mathcal{S}}'(H''')|}{\max_{H\in A}|\mathcal{S}(H)|},\quad
 \frac{d_{max}^B}{d_{min}^{A_1}}= \frac{\max_{H'''\in B}|{\mathcal{S}}'(H''')|}{\min_{H\in A_1}|\mathcal{S}(H)|}
\]
to complete the proof of $(b)$, where
$O\bigl( \frac{r^6}{n^3}\bigr)$ is absorbed into
$O\bigl( \frac{r^6mn\log (r^{-2}n)+r^8m^2}{n^{4}}\bigr)$.
\end{proof}

\subsection{Switchings of Type-$2$ clusters}\label{s:5.2}

Let $H\in \mathcal{C}_{0,h_2,h_3,h_4}^{+}$. A \textit{Type-$2$ switching}
from $H$ is used to reduce the number of Type-$2$ clusters in $H$.
It is defined in the same manner as Steps $0$-$3$ in Section~\ref{s:5.1}.
Since we will only use this switching after all Type-$1$ clusters
have been removed, we can assume~$h_1=0$.

\begin{remark}\label{r5.5}
A Type-$2$ switching reduces the number of Type-$2$
clusters in $H$ by one without affecting other types of clusters.
If $H\in \mathcal{C}_{0,h_2,h_3,h_4}^{++}$, then a Type-$2$
switching from $H$ may violate property
$\bf(g^*)$ for $\mathcal{H}_r^{++}(n,m)$, but
the resulting graph $H'''$ is in $\mathcal{C}_{0,h_2-1,h_3,h_4}^{+}$
because $\deg (v)\leq M_0^*+3= M_0$
for every vertex $v\in[n]$ of $H'''$.
\end{remark}

A \textit{reverse Type-$2$ switching} is  the reverse of a  Type-$2$ switching.
A  reverse Type-$2$ switching from $H'''\in \mathcal{C}_{0,h_2-1,h_3,h_4}^{++}$
is defined by sequentially removing  three edges  not
containing a link, then choosing a $(3r-4)$-set $T$ from $[n]$
such that no two vertices belong to  any remaining edges of $H'''$, then inserting three edges into
$T$ such that they create a Type-$2$ cluster.

\begin{remark}\label{r5.6}
If $H'''\in \mathcal{C}_{0,h_2-1,h_3,h_4}^{++}$, then a reverse
Type-$2$ switching from $H'''$ may also
violate the property $\bf(g^*)$  for $\mathcal{H}_r^{++}(n,m)$, but
the resulting graph $H\in\mathcal{C}_{0,h_2,h_3,h_4}^{+}$ because $\deg (v)\leq M_0^*+3=M_0$
for every vertex $v\in[n]$ of $H$.
\end{remark}

 Next we analyze Type-$2$ switchings to find a  relationship
 between the sizes of $\mathcal{C}_{0,h_2,h_3,h_4}^{++}$ and
 $\mathcal{C}_{0,h_2-1,h_3,h_4}^{++}$.

\begin{lemma}\label{l5.7}
Assume $r^{-2}n\leq m=o(r^{-3}n^{ \frac32})$
and that $1\leq h_2\leq M_2$, $0\leq h_3\leq M_3$
and $0\leq h_4\leq M_4$.\\
$(a)$\ Let $H\in \mathcal{C}_{0,h_2,h_3,h_4}^{+}$.
Then the number of Type-$2$ switchings for $H$ is
\[
h_2\biggl[N- \binom{r}{2}m \binom{n-2}{r-2}\biggr]^3\biggl(1+O\Bigl( \frac{r^6mn\log (r^{-2}n)+r^8m^2}{n^{4}}\Bigr)\biggr).
\]
$(b)$\ Let $H'''\in \mathcal{C}_{0,h_2-1,h_3,h_4}^{+}$.
Then the number of reverse Type-$2$ switchings for $H'''$ is
\begin{align*}
&18 \binom{3r-4}{r} \binom{r}{3} \binom{2r-4}{r-2} \binom{m-3(h_2-1)-3h_3-2h_4}{3}\\
&{\qquad}\times\biggl[ \binom{n}{3r-4}- \binom{r}{2}m \binom{n-2}{3r-6}\biggr]
\biggl(1+O\Bigl( \frac{r^6mn\log (r^{-2}n)+r^8m^2}{n^{4}}\Bigr)\biggr).
\end{align*}
\end{lemma}

\begin{proof}
The proof follows the same logic as the proof of Lemma~\ref{l5.3},
so we omit it.
% $(a)$\ Let $H\in \mathcal{C}_{0,h_2,h_3,h_4}^{+}$.
%Define the set $\mathcal{S}(H)$ as all
%Type-$2$ switchings which can be applied to $H$.
%By the same proof  in Lemma~\ref{l5.3}, we have
%\[
%\begin{aligned}[b]
%|\mathcal{S}(H)|&=h_2\biggl[N- \binom{r}{2}m \binom{n-2}{r-2}\biggr]^3
%\biggl(1+O\Bigl( \frac{r^6mn\log (r^{-2}n)+r^8m^2}{n^{4}}\Bigr)\biggr).
%\end{aligned}
%\]
%
%
%$(b)$\ Conversely, suppose that $H'''\in \mathcal{C}_{0,h_2-1,h_3,h_4}^{+}$.
%Similarly,  let $\mathcal{S}'(H''')$
%be the set of all reverse Type-$2$ switchings for $H'''$.
%There are exactly $6 \binom{m-3(h_2-1)-3h_3-2h_4}{3}$ ways to delete
%three edges in sequence such that
%none of them contain a link.
%By the same proof  in Lemma~\ref{l5.3}, there are also
%\[
%\begin{aligned}[b]
%\biggl[ \binom{n}{3r-4}- \binom{r}{2}m \binom{n-2}{3r-6}\biggr]
%\biggl(1+O\Bigl( \frac{r^6mn\log (r^{-2}n)+r^8m^2}{n^{4}}\Bigr)\biggr)
%\end{aligned}
%\] ways to choose a $(3r-4)$-set $T$ from $[n]$ such that no two vertices
%are on the same edge of $H'''$. For every $T$,
% there are $ \binom{3r-4}{r} \binom{r}{3} \binom{3}{1} \binom{2r-4}{r-2}$
% ways to create a Type-$2$ cluster. Thus, we have
%\[
%\begin{aligned}
%|\mathcal{S}'(H''')|&=18 \binom{3r-4}{r} \binom{r}{3} \binom{2r-4}{r-2} \binom{m-3(h_2-1)-3h_3-2h_4}{3}\\
%&{}\times\biggl[ \binom{n}{3r-4}- \binom{r}{2}m \binom{n-2}{3r-6}\biggr]
%\biggl(1+O\Bigl( \frac{r^6mn\log (r^{-2}n)+r^8m^2}{n^{4}}\Bigr)\biggr).\qedhere
%\end{aligned}
%\]
\end{proof}

\begin{corollary}\label{c5.8}
With notation as above, for some $0\leq h_i\leq M_i$ with $2\leq i\leq 4$,\\
$(a)$\ If $\mathcal{C}_{0,h_2,h_3,h_4}^{++}=\emptyset$, then $\mathcal{C}_{0,h_2+1,h_3,h_4}^{++}=\emptyset$.
\\
$(b)$\
Let $h_2'=h_2'(h_3,h_4)$ be the first value of $h_2\leq M_2$
such that $\mathcal{C}_{0,h_2,h_3,h_4}^{++}=\emptyset$,
or $h_2'=M_2+1$ if no such value exists. Suppose that $n\to\infty$.
Then uniformly for $1\leq h_2< h_2'$,
\begin{align*}
 \frac{|\mathcal{C}_{0,h_2,h_3,h_4}^{++}|}{|\mathcal{C}_{0,h_2-1,h_3,h_4}^{++}|}
&= \frac{18 \binom{3r-4}{r} \binom{r}{3} \binom{2r-4}{r-2} \binom{m-3(h_2-1)-3h_3-2h_4}{3}
\bigl[ \binom{n}{3r-4}- \binom{r}{2}m \binom{n-2}{3r-6}\bigr]}{h_2\bigl[N- \binom{r}{2}m \binom{n-2}{r-2}\bigr]^3}\\
&{\qquad}\times\biggl(1+O\Bigl( \frac{r^6mn\log (r^{-2}n)+r^8m^2}{n^{4}}\Bigr)\biggr).
\end{align*}
\end{corollary}

\begin{proof}
This is proved in the same way as Corollary~\ref{c5.4}.
\end{proof}

\subsection{Switchings of Type-$3$ clusters}\label{s:5.3}

Let $H\in \mathcal{C}_{0,0,h_3,h_4}^{+}$. A \textit{Type-$3$ switching}
from $H$ is used to reduce the number of Type-$3$ clusters in $H$,
after the numbers of Type-$1$ and Type-$2$ clusters have been
reduced to zero.
It is  defined in the same manner as Steps $0$-$3$  in Section~5.1.

\begin{remark}\label{r5.9}
A Type-$3$ switching reduces the number of Type-$3$ clusters
in $H$ by one without affecting other types of  clusters.
If $H\in \mathcal{C}_{0,0,h_3,h_4}^{++}$, then a Type-$3$ switching
from $H$ may violate the property
$\bf(g^*)$ for $\mathcal{H}_r^{++}(n,m)$, but
the resulting graph $H'''$ is in $\mathcal{C}_{0,0,h_3-1,h_4}^{+}$
because $\deg (v)\leq M_0^*+3= M_0$
for every vertex $v\in[n]$ of $H'''$.
\end{remark}

A \textit{reverse Type-$3$ switching} is the reverse of a  Type-$3$ switching.
A  reverse Type-$3$ switching from $H'''\in \mathcal{C}_{0,0,h_3-1,h_4}^{++}$
is defined by sequentially removing  three edges not
containing a link, then choosing a $(3r-4)$-set $T$ from $[n]$
such that no two vertices belong to any remaining edges of $H'''$,
then inserting three edges into $T$ such that
 they create a Type-$3$ cluster.

\begin{remark}\label{r5.10}
If $H'''\in \mathcal{C}_{0,0,h_3-1,h_4}^{++}$, then a reverse Type-$3$ switching from $H'''$ may also
violate the property $\bf(g^*)$  for $\mathcal{H}_r^{++}(n,m)$, but
the resulting graph $H$ is $\mathcal{C}_{0,0,h_3-1,h_4}^{+}$
because $\deg (v)\leq M_0^*+3=M_0$ for every vertex $v\in[n]$ of $H$.
\end{remark}

 Next we analyze Type-$3$ switchings to find a  relationship
 between the sizes of $\mathcal{C}_{0,0,h_3,h_4}^{++}$ and
 $\mathcal{C}_{0,0,h_3-1,h_4}^{++}$.

\begin{lemma}\label{l5.11}
Assume $r^{-2}n\leq m=o(r^{-3}n^{ \frac32})$,
$1\leq h_3\leq M_3$ and $0\leq h_4\leq M_4$.\\
$(a)$\ Let $H\in \mathcal{C}_{0,0,h_3,h_4}^{+}$.
Then the number of Type-$3$ switchings for $H$ is
\begin{align*}
h_3\biggl[N- \binom{r}{2}m \binom{n-2}{r-2}\biggr]^3
\biggl(1+O\Bigl( \frac{r^6mn\log (r^{-2}n)+r^8m^2}{n^{4}}\Bigr)\biggr).
\end{align*}
$(b)$\ Let $H'''\in \mathcal{C}_{0,0,h_3-1,h_4}^{+}$. Then the number
of reverse Type-$3$ switchings for $H'''$ is
\begin{align*}
& \binom{3r-4}{r} \binom{r}{2} \binom{2r-4}{r-2} \binom{m-3(h_3-1)-2h_4}{3}\\
&{\qquad}\times\biggl[ \binom{n}{3r-4}- \binom{r}{2}m \binom{n-2}{3r-6}\biggr]
\biggl(1+O\Bigl( \frac{r^6mn\log (r^{-2}n)+r^8m^2}{n^{4}}\Bigr)\biggr).
\end{align*}
\end{lemma}

\begin{proof}
The proof follows the same logic as the proof of Lemma~\ref{l5.3},
so we omit it.
%$(a)$\  Let $H\in \mathcal{C}_{0,0,h_3,h_4}^{+}$.
% Define the set $\mathcal{S}(H)$ as  all Type-$3$ switchings which can be applied to $H$.
%By the same proof as shown in Lemma~\ref{l5.3} and Lemma~\ref{l5.7}, we have
%\[
%\begin{aligned}[b]
%|\mathcal{S}(H)|&=h_3\biggl[N- \binom{r}{2}m \binom{n-2}{r-2}\biggr]^3
%\biggl(1+O\Bigl( \frac{r^6mn\log (r^{-2}n)+r^8m^2}{n^{4}}\Bigr)\biggr).
%\end{aligned}
%\]
%
%
%$(b)$\ Conversely, suppose that $H'''\in \mathcal{C}_{0,0,h_3-1,h_4}^{+}$.
%Similarly,  let $\mathcal{S}'(H''')$ be the set of all reverse Type-$3$ switchings for $H'''$.
%There are exactly $6 \binom{m-3(h_3-1)-2h_4}{3}$ ways to delete three edges
%in sequence such that none of them contain a link in $H'''$. By the same
%proof as shown in Lemma~\ref{l5.3} and Lemma~\ref{l5.7},
%there are also
%\[
%\begin{aligned}[b]
%\biggl[ \binom{n}{3r-4}- \binom{r}{2}m \binom{n-2}{3r-6}\biggr]
%\biggl(1+O\Bigl( \frac{r^6mn\log (r^{-2}n)+r^8m^2}{n^{4}}\Bigr)\biggr)
%\end{aligned}
%\] ways to choose a $(3r-4)$-set $T$ from $[n]$ such that no two vertices
%are on the same edge of~$H'''$. For every $T$,
% there are $ \frac{1}{6} \binom{3r-4}{r} \binom{r}{2} \binom{2r-4}{r-2}$ ways to create a Type-$3$ cluster.
% Thus, we have
%\[
%\begin{aligned}
%|\mathcal{S}'(H''')|&= \binom{3r-4}{r} \binom{r}{2} \binom{2r-4}{r-2} \binom{m-3(h_3-1)-2h_4}{3}\\
%&{}\times\biggl[ \binom{n}{3r-4}- \binom{r}{2}m \binom{n-2}{3r-6}\biggr]
%\biggl(1+O\Bigl( \frac{r^6mn\log (r^{-2}n)+r^8m^2}{n^{4}}\Bigr)\biggr).\qedhere
%\end{aligned}
%\]
\end{proof}

\begin{corollary}\label{c5.12}
With notation as above, for some $0\leq h_i\leq M_i$ with $3\leq i\leq 4$,\\
$(a)$\ If $\mathcal{C}_{0,0,h_3,h_4}^{++}=\emptyset$,
then $\mathcal{C}_{0,0,h_3+1,h_4}^{++}=\emptyset$.\\
$(b)$\
Let $h_3'=h_3'(h_4)$ be the first value of $h_3\leq M_3$
such that $\mathcal{C}_{0,0,h_3,h_4}^{++}=\emptyset$,
or $h_3'=M_3+1$ if no such value exists. Suppose that $n\to\infty$.
Then uniformly for $1\leq h_3< h_3'$,
\begin{align*}
 \frac{|\mathcal{C}_{0,0,h_3,h_4}^{++}|}{|\mathcal{C}_{0,0,h_3-1,h_4}^{++}|}
&= \frac{ \binom{3r-4}{r} \binom{r}{2} \binom{2r-4}{r-2} \binom{m-3(h_3-1)-2h_4}{3}
\bigl[ \binom{n}{3r-4}- \binom{r}{2}m \binom{n-2}{3r-6}\bigr]}{h_3\bigl[N- \binom{r}{2}m \binom{n-2}{r-2}\bigr]^3}\\
&{\qquad}\times\biggl(1+O\Bigl( \frac{r^6mn\log (r^{-2}n)+r^8m^2}{n^{4}}\Bigr)\biggr).
\end{align*}
\end{corollary}

\begin{proof}
This is proved in the same way as Corollary~\ref{c5.4}.
\end{proof}

\subsection{Switchings of Type-$4$ clusters}\label{s:5.5}

Let $H\in \mathcal{C}_{0,0,0,h_4}^{+}$.
A \textit{Type-$4$ switching} from $H$
is used to reduce the number of Type-$4$ clusters in $H$ after the
Type-$1$, Type-$2$ and Type-$3$ clusters have been removed.

\noindent{\bf Step 0.}\  Take any one of these $h_4$ Type-$4$ clusters in $H$, denoted by $\{e,f\}$,
and remove it from $H$.
Define $H_0$ with the same vertex set $[n]$ and the edge set $E(H_0)=E(H)\backslash \{e,f\}$.

\noindent{\bf Step 1.}\ Take any $r$-set  from $[n]$ such that no two vertices
belong to the same edge of $H_0$ and insert one new edge
to the $r$-set.  The new graph is denoted by $H'$.

\noindent{\bf Step 2.}\ Repeat the process of Step 1 in $H'$, that is, insert another new edge to
an $r$-set of $[n]$ such that %
no two vertices belong to the same edge of $H'$.
The resulting graph is denoted by $H''$.

Note that the two new edges may or may not have a vertex in common.

\begin{remark}\label{r5.13}
If $H\in \mathcal{C}_{0,0,0,h_4}^{++}$, then a Type-$4$ switching
from $H$ may violate the property $\bf(g^*)$ for $\mathcal{H}_r^{++}(n,m)$, but
the resulting graph $H''$ is in $\mathcal{C}_{0,0,0,h_4-1}^{++}$
because $\deg (v)\leq M_0^*+2< M_0$ for each vertex $v\in[n]$ of  $H''$.
\end{remark}

A \textit{reverse Type-$4$ switching} is the reverse of a Type-$4$ switching.
A reverse Type-$4$ switching from $H''\in \mathcal{C}_{0,0,0,h_4-1}^{++}$
is defined by sequentially removing  two edges  not containing a
link in $H''$, then choosing a $(2r-2)$-set $T$ from $[n]$
such that at most two vertices belong to some remaining edge of $H''$,
then inserting  two edges into $T$ such that
they create a Type-$4$ cluster.

\begin{remark}\label{r5.14}
If $H''\in \mathcal{C}_{0,0,0,h_4-1}^{++}$, then a reverse
Type-$4$ switching from $H''$ may also violate the property $\bf(g^*)$
for $\mathcal{H}_r^{++}(n,m)$, but the resulting graph $H$ is
in $\in\mathcal{C}_{0,0,0,h_4}^{+}$
because $\deg (v)\leq M_0^*+2<M_0$ for each vertex $v\in[n]$ of $H$.
\end{remark}

Next we analyze Type-$4$ switchings to find a relationship
between the sizes of $\mathcal{C}_{0,0,0,h_4}^{++}$ and
$\mathcal{C}_{0,0,0,h_4-1}^{++}$.

\begin{lemma}\label{l5.15}
Assume that $r^{-2}n\leq m=o(r^{-3}n^{ \frac32})$ and $1\leq h_4\leq M_4$.\\
$(a)$\ Let $H\in \mathcal{C}_{0,0,0,h_4}^{+}$.
Then the number of Type-$4$ switchings for $H$ is \\
\[
h_4\biggl[N- \binom{r}{2}m \binom{n-2}{r-2}\biggr]^2
\biggl(1+O\Bigl( \frac{r^6mn\log (r^{-2}n)+r^8m^2}{n^{4}}\Bigr)\biggr).
\]
$(b)$\ Let $H''\in \mathcal{C}_{0,0,0,h_4-1}^{+}$. Then the number of reverse
Type-$4$ switchings for $H''$ is
\begin{align*}
& \binom{2r-2}{2} \binom{2r-4}{r-2} \binom{m-2(h_4-1)}{2}
\biggl[ \binom{n}{2r-2}- \frac{(r^2-r-1)}{(r-1)(2r-3)} \binom{r}{2}m \binom{n-2}{2r-4}\biggr]\\
&{\qquad}\times\biggl(1+O\Bigl( \frac{r^6mn\log (r^{-2}n)+r^8m^2}{n^{4}}\Bigr)\biggr).
\end{align*}
\end{lemma}

\begin{proof}
Since the likely number of Type-$4$ clusters in a random hypergraph is greater than
that of the other cluster types, our counting here must be more careful.

$(a)$\ Let $H\in \mathcal{C}_{0,0,0,h_4}^{+}$.
Define the set $\mathcal{S}(H)$ of all Type-$4$ switchings which can be applied to $H$.
By the same proof as given for Lemma~\ref{l5.3}, Lemma~\ref{l5.7} and Lemma~\ref{l5.11}, we have
\[
|\mathcal{S}(H)|=h_4\biggl[N- \binom{r}{2}m \binom{n-2}{r-2}\biggr]^2
\biggl(1+O\Bigl( \frac{r^6mn\log (r^{-2}n)+r^8m^2}{n^{4}}\Bigr)\biggr).
\]

$(b)$\ Conversely, suppose that $H''\in \mathcal{C}_{0,0,0,h_4-1}^{+}$.
Similarly,  let $\mathcal{S}'(H'')$ be the set of all reverse Type-$4$ switchings for $H''$.
There are exactly $2 \binom{m-2(h_4-1)}{2}$ ways to delete two edges in sequence
such that neither of them contain a link in $H''$.
Unlike the reverse Type-$1$, Type-$2$ and Type-$3$ switchings, the chosen $(2r-2)$-set
may include two vertices belong to the same edge of $H''$, as shown by the
dashed lines in Figure~\ref{fig:3}.
\begin{figure}[!htb]
\centering
\includegraphics[width=0.3\textwidth]{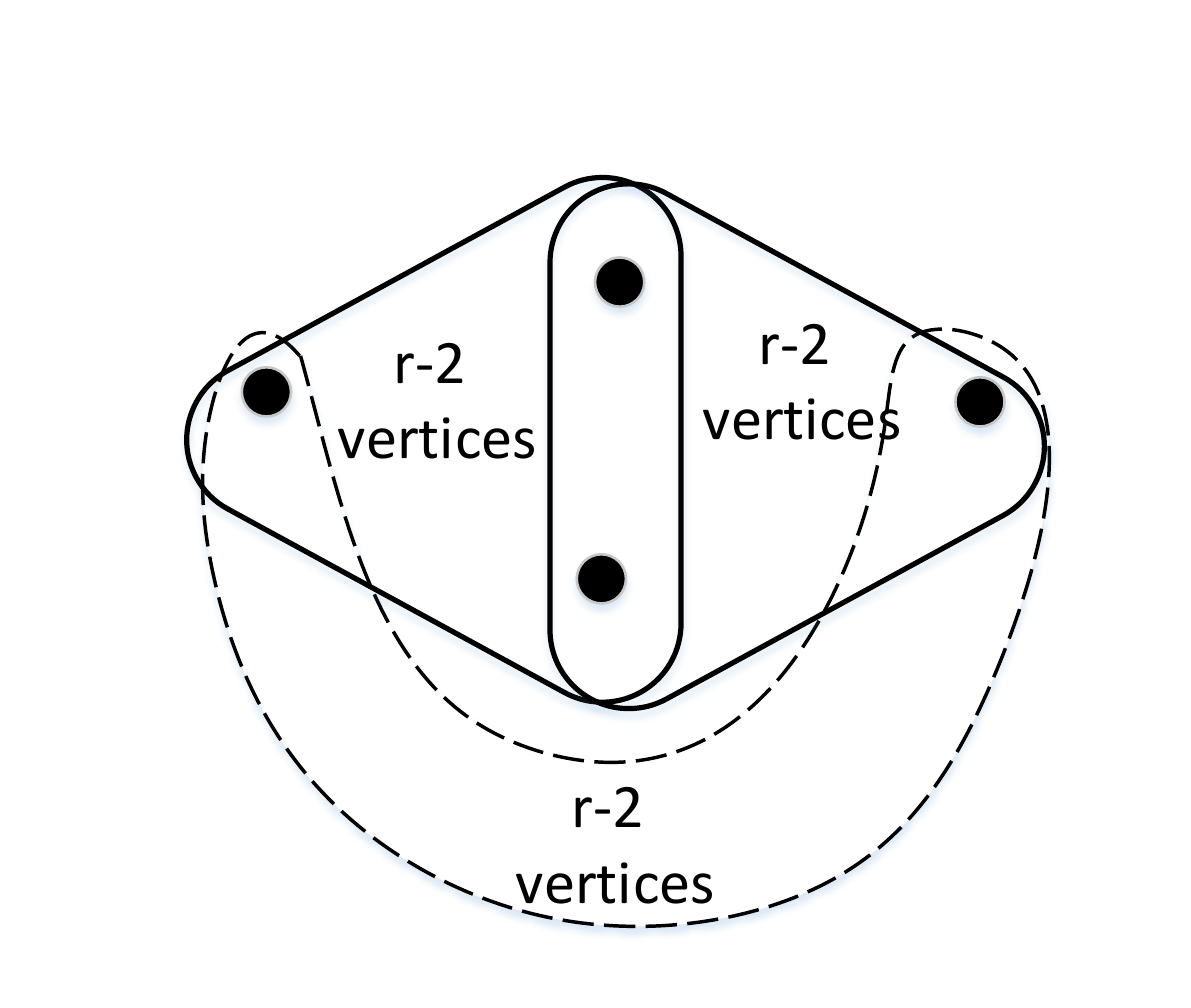}
\caption{$(2r-2)$-subset in the reverse Type-$4$ switching\label{fig:3}}
\end{figure}

Firstly, by Lemma~\ref{l3.5}, there are
\begin{align*}
\biggl[ \binom{n}{2r-2}- \binom{r}{2}m \binom{n-2}{2r-4}\biggr]
\biggl(1+O\Bigl( \frac{r^6mn\log (r^{-2}n)+r^8m^2}{n^{4}}\Bigr)\biggr)
\end{align*}
ways to choose a $(2r-2)$-set $T$ from $[n]$ such that no two vertices
belong to the same edge of~$H''$.
For every such~$T$,
there are $\frac12 \binom{2r-2}{2} \binom{2r-4}{r-2}$ ways to create a Type-$4$ cluster.

Secondly, by Lemma~\ref{l3.6}, there are
\begin{align*}
 \binom{r}{2}m \binom{n-2}{2r-4}\biggl(1+O\Bigl( \frac{r^2n\log(r^{-2}n)+r^4m}{n^2}\Bigr)\biggr)
\end{align*}
ways to choose a $(2r-2)$-set $T$ from $[n]$ such that exactly two vertices
belong to the same edge of $H''$. For every such $T$,
 there are $ \binom{2r-4}{2} \binom{2r-6}{r-3}$ ways to create a Type-$4$
 cluster.  Thus, we  have
\begin{align*}
|\mathcal{S}'&(H'')|\\
&=2 \binom{m-2(h_4-1)}{2}\biggl[ \frac{ \binom{2r-2}{2} \binom{2r-4}{r-2}}{2}
\biggl( \binom{n}{2r-2}- \binom{r}{2}m \binom{n-2}{2r-4}\biggr)\\
&{\qquad}+
 \binom{2r-4}{2} \binom{2r-6}{r-3} \binom{r}{2}m \binom{n-2}{2r-4}\biggr]
\biggl(1+O\Bigl( \frac{r^6mn\log (r^{-2}n)+r^8m^2}{n^{4}}\Bigr)\biggr)\\
&= \binom{2r-2}{2} \binom{2r-4}{r-2} \binom{m-2(h_4-1)}{2}
\biggl[ \binom{n}{2r-2}- \frac{(r^2-r-1)}{(r-1)(2r-3)} \binom{r}{2}m \binom{n-2}{2r-4}\biggr]\\
&{\qquad}\times\biggl(1+O\Bigl( \frac{r^6mn\log (r^{-2}n)+r^8m^2}{n^{4}}\Bigr)\biggr).\qedhere
\end{align*}
\end{proof}

\begin{corollary}\label{c5.16}
With notation as above, for some $0\leq h_4\leq M_4$,\\
$(a)$\ If $\mathcal{C}_{0,0,0,h_4}^{++}=\emptyset$, then $\mathcal{C}_{0,0,0,h_4+1}^{++}=\emptyset$.
\\
$(b)$\
Let $h_4'$ be the first value of $h_4\leq M_4$
such that $\mathcal{C}_{0,0,0,h_4}^{++}=\emptyset$,
or $h_4'=M_4+1$ if no such value exists. Suppose that $n\to\infty$.
Then uniformly for $1\leq h_4< h_4'$,
\begin{align*}
 \frac{|\mathcal{C}_{0,0,0,h_4}^{++}|}{|\mathcal{C}_{0,0,0,h_4-1}^{++}|}
  &=
 \frac{ \binom{2r-2}{2} \binom{2r-4}{r-2} \binom{m-2(h_4-1)}{2}
 \bigl[ \binom{n}{2r-2}- \frac{(r^2-r-1)}{(r-1)(2r-3)} \binom{r}{2}m \binom{n-2}{2r-4}\bigr]}
      {h_4\bigl[N- \binom{r}{2}m \binom{n-2}{r-2}\bigr]^2}\\
 &{\qquad}\times\biggl(1+O\Bigl( \frac{r^6mn\log (r^{-2}n)+r^8m^2}{n^{4}}\Bigr)\biggr).
\end{align*}
\end{corollary}

\begin{proof}
This is proved in the same way as Corollary~\ref{c5.4}.
\end{proof}

\section{Analysis of switchings}\label{s:6}

In this section, we estimate the sum
\[
\sum_{h_4=0}^{M_4}\sum_{h_3=0}^{M_3}\sum_{h_2=0}^{M_2}\sum_{h_1=0}^{M_1}
 \frac{|\mathcal{C}_{h_1,h_2,h_3,h_4}^{++}|}{|\mathcal{C}_{0,0,0,0}^{++}|}
\]
to finish the proof of the case
$r^{-2}n\leq m=o(r^{-3}n^{ \frac32})$ in accordance with Remark~\ref{r4.5}.
We will need the following summation lemmas from \cite{green06}, and
state them here for completeness.

\begin{lemma}[\cite{green06}, Corollary~4.3]\label{l6.1}
Let $K$, $N$ be integers with $N\geq 2$ and $0\leq K\leq N$. Suppose that
there are real numbers $\delta_i$ for $1\leq i\leq N$, and
$\gamma_i$ for $0\leq i\leq K$ such that
$\sum_{j=1}^{i}|\delta_j|\leq \sum_{j=0}^K \gamma_j [i]_j< \frac15$
for $1\leq i\leq N$. Let real numbers $A(i)$, $B(i)$ be given such that
$A(i)\geq 0$ and $1-(i-1)B(i)\geq 0$. Define $A_1=\min_{i=1}^NA(i)$, $A_2=\max_{i=1}^NA(i)$,
$B_1=\min_{i=1}^NB(i)$ and $B_2=\max_{i=1}^NB(i)$. Also suppose all $A\in [A_1,A_2]$,
$B\in [B_1,B_2]$ and $c$ be real numbers such that
$c>2e$, $0\leq Ac< N-K+1$ and $|BN|<1$. Define $n_0,n_1,\ldots, n_N$ by $n_0=1$ and
\[
 \frac{n_i}{n_{i-1}}= \frac{A(i)}{i}\bigl(1-(i-1)B(i)\bigr)\bigl(1+\delta_i\bigr)
\]
for $1\leq i\leq N$, with the following interpretation: if $A(i)= 0$ or
$1-(i-1)B(i)=0$, then $n_j=0$ for $i\leq j\leq N$. Then
\[
\Sigma_1\leq \sum_{i=0}^{N}n_i\leq \Sigma_2,
\]
where
\begin{align*}
\Sigma_1&=\exp\Bigl[A_1- \dfrac12A_1^2B_2-4\sum_{j=0}^{K}\gamma_j
\bigl(3A_1\bigr)^j\Bigr]- \dfrac14\bigl(2e/c\bigr)^N,\\
\Sigma_2&=\exp\Bigl[A_2- \dfrac12A_2^2B_1+ \dfrac12A_2^3B_1^2+
4\sum_{j=0}^{K}\gamma_j\bigl(3A_2\bigr)^j\Bigr]+ \dfrac14\bigl(2e/c\bigr)^N.
\end{align*}
\end{lemma}

\begin{lemma}[\cite{green06}, Corollary~4.5]\label{l6.2}
Let $N\geq 2$ be an integer, and for $1\leq i\leq N$, let real numbers
$A(i)$, $B(i)$ be given such that $A(i)\geq 0$ and $1-(i-1)B(i)\geq 0$.
Define $A_1=\min_{i=1}^NA(i)$, $A_2=\max_{i=1}^NA(i)$, $C_1=\min_{i=1}^NA(i)B(i)$
and $C_2=\max_{i=1}^NA(i)B(i)$. Suppose that there exists a real number  $\hat{c}$ with $0<\hat{c}< \frac{1}{3}$
such that $\max\{A/N,|C|\}\leq \hat{c}$ for all $A\in [A_1,A_2]$, $C\in[C_1,C_2]$.
Define $n_0$, $n_1$, $\ldots$, $n_N$ by $n_0=1$ and
\[
 \frac{n_i}{n_{i-1}}= \frac{A(i)}{i}\left(1-(i-1)B(i)\right)
\]
for $1\leq i\leq N$, with the following interpretation: if $A(i)= 0$ or $1-(i-1)B(i)=0$, then $n_j=0$
for $i\leq j\leq N$. Then
\[
\Sigma_1\leq \sum_{i=0}^{N}n_i\leq \Sigma_2,
\]
where
\begin{align*}
\Sigma_1&=\exp\Bigl[A_1- \dfrac12A_1C_2\Bigr]-\bigl(2e\hat{c}\bigr)^N,\\
\Sigma_2&=\exp\Bigl[A_2- \dfrac12A_2C_1+ \dfrac12A_2C_1^2\Bigr]+\bigl(2e\hat{c}\bigr)^N.
\end{align*}
\end{lemma}

\begin{lemma}\label{l6.3}
Assume
$r^{-2}n\leq m=o(r^{-3}n^{ \frac32})$.
If $\bigl|\mathcal{C}_{0,h_2,h_3,h_4}^{++}\bigl|\neq0$, then
\begin{align*}
\sum_{h_1=0}^{M_1} \frac{|\mathcal{C}_{h_1,h_2,h_3,h_4}^{++}|}{|\mathcal{C}_{0,h_2,h_3,h_4}^{++}|}
=\exp\biggl[ \frac{6 \binom{3r-4}{r} \binom{r}{4} \binom{2r-4}{r-2}m^3\bigl[ \binom{n}{3r-4}-
 \binom{r}{2}m \binom{n-2}{3r-6}\bigr]}{\bigl[N- \binom{r}{2}m \binom{n-2}{r-2}\bigr]^3}+O\Bigl( \frac{r^6m^2}{n^3}\Bigr)\biggr].
\end{align*}
\end{lemma}

\begin{proof} Let $h_1'=h_1'(h_2,h_3,h_4)$ be the first value of $h_1\leq M_1$
such that $\mathcal{C}_{h_1,h_2,h_3,h_4}^{++}=\emptyset$,
or $h_1'=M_1+1$ if no such value exists, which is  defined in Corollary~\ref{c5.4} $(b)$.
Since $|\mathcal{C}_{0,h_2,h_3,h_4}^{++}|\neq0$, we know $h_1'\geq1$.
Define $n_{1,0},\ldots,n_{1,M_1}$ by $n_{1,0}=1$,
\[
n_{1,h_1}= \frac{|\mathcal{C}_{h_1,h_2,h_3,h_4}^{++}|}{|\mathcal{C}_{0,h_2,h_3,h_4}^{++}|}
\]
for $1\leq h_1<h_1'$ and $n_{1,h_1}=0$ for $h_1'\leq h_1\leq M_1$.
Note that
\begin{align*}
& \binom{m-3(h_1-1)-3(h_2+h_3)-2h_4}{3}\\
&{\qquad}= \frac{m^3}{6}\biggl(1+O\Bigl(\frac{r^4m}{n^2}+
 \frac{\log(r^{-2}n)}{m}\Bigr)\biggr)\Bigl(1-(h_1-1) \frac{9}{m}\Bigr).
\end{align*}
By Corollary~\ref{c5.4} $(b)$, we have for $1\leq h_1< h_1'$,
\begin{equation}\label{e6.1}
\begin{split}
 \frac{|\mathcal{C}_{h_1,h_2,h_3,h_4}^{++}|}{|\mathcal{C}_{0,h_2,h_3,h_4}^{++}|}
&= \frac{1}{h_1} \frac{|\mathcal{C}_{h_1-1,h_2,h_3,h_4}^{++}|}{|\mathcal{C}_{0,h_2,h_3,h_4}^{++}|}
 \frac{6 \binom{3r-4}{r} \binom{r}{4} \binom{2r-4}{r-2}m^3\bigl[ \binom{n}{3r-4}-
 \binom{r}{2}m \binom{n-2}{3r-6}\bigr]}{\bigl[N- \binom{r}{2}m \binom{n-2}{r-2}\bigr]^3} \\
&{\qquad}\times\biggl(1+O\biggl(\frac{r^4m}{n^2}+ \frac{\log(r^{-2}n)}{m}
+ \frac{r^6mn\log (r^{-2}n)+r^8m^2}{n^{4}}\biggr)\biggr) \\
&{\qquad}\times\Bigl(1-(h_1-1) \frac{9}{m}\Bigr).
\end{split}
\end{equation}
Define
\begin{align*}
A(1,h_1)&= \frac{6 \binom{3r-4}{r} \binom{r}{4} \binom{2r-4}{r-2}m^3\bigl[ \binom{n}{3r-4}
- \binom{r}{2}m \binom{n-2}{3r-6}\bigr]}{\bigl[N- \binom{r}{2}m \binom{n-2}{r-2}\bigr]^3}\\
&{\qquad}\times\biggl(1+O\biggl(\frac{r^4m}{n^2}+\frac{\log(r^{-2}n)}{m}
+ \frac{r^6mn\log (r^{-2}n)+r^8m^2}{n^{4}}\biggr)\biggr)
\end{align*}
for $1\leq h_1< h_1'$  and $A(1,h_1)=0$ for $h_1'\leq h_1\leq M_1$, and
$B(1,h_1)= \frac{9}{m}$ for $1\leq h_1\leq M_1$.
Let $A_{1,1}=\min_{h_1=1}^{M_1}A(1,h_1)$,
$A_{1,2}=\max_{h_1=1}^{M_1}A(1,h_1)$, $C_{1,1}=\min_{h_1=1}^{M_1}A(1,h_1)B(1,h_1)$
and $C_{1,2}=\max_{h_1=1}^{M_1}A(1,h_1)B(1,h_1)$.
By~\eqref{e6.1} we  have
\[
 \frac{n_{1,h_1}}{n_{1,h_1-1}}= \frac{A(1,h_1)} {h_1}\bigl(1-(h_1-1)B(1,h_1)\bigr).
\]
Let $\hat{c}= \frac{1}{3^4}$, then we have  $\max\bigl\{A/M_1,|C|\bigr\}\leq \hat{c}< \frac{1}{3}$
for all $A\in [A_{1,1},A_{1,2}]$, $C\in[C_{1,1},C_{1,2}]$.
Therefore, Lemma~\ref{l6.2} applies and we obtain
\begin{align*}
\sum_{h_1=0}^{M_1}& \frac{|\mathcal{C}_{h_1,h_2,h_3,h_4}^{++}|}{|\mathcal{C}_{0,h_2,h_3,h_4}^{++}|}\\
&=\exp\biggl[ \frac{6 \binom{3r-4}{r} \binom{r}{4} \binom{2r-4}{r-2}m^3
\bigl[ \binom{n}{3r-4}- \binom{r}{2}m \binom{n-2}{3r-6}\bigr]}
{\bigl[N- \binom{r}{2}m \binom{n-2}{r-2}\bigr]^3}
\\&{\qquad}+O\biggl(\frac{r^{12}m^4}{n^6}+
 \frac{r^{8}m^2\log(r^{-2}n)}{n^4}+
 \frac{r^{14}m^4n\log(r^{-2}n)+r^{16}m^5}{n^8}\biggr)\biggr]+O\Bigl(\Bigl( \frac{2e}{3^4}\Bigr)^{M_1}\Bigr).
\end{align*}
Note that
$O\bigl( \frac{r^{12}m^4}{n^6}+ \frac{r^{8}m^2\log(r^{-2}n)}{n^4}+
 \frac{r^{14}m^4n\log(r^{-2}n)+r^{16}m^5}{n^8}\bigr)
=O\bigl( \frac{r^6m^2}{n^3}\bigr)$ and $O\bigl(\bigl( \frac{2e}{3^4}\bigr)^{M_1}\bigr)=O\bigl( \frac{r^6}{n^3}\bigr)$ as
$r^{-2}n\leq m=o(r^{-3}n^{ \frac32})$,
which gives the expression in the lemma.
%\[
%\begin{aligned}[b]
%\sum_{h_1=0}^{M_1} \frac{|\mathcal{C}_{h_1,h_2,h_3,h_4}^{++}|}{|\mathcal{C}_{0,h_2,h_3,h_4}^{++}|}
%&=\exp\biggl[ \frac{6 \binom{3r-4}{r} \binom{r}{4} \binom{2r-4}{r-2}m^3\bigl
%[ \binom{n}{3r-4}- \binom{r}{2}m \binom{n-2}{3r-6}\bigr]}{\bigl[N-
% \binom{r}{2}m \binom{n-2}{r-2}\bigr]^3}+O\Bigl( \frac{r^6m^2}{n^3}\Bigr)\biggr].\\
%\end{aligned}
%\]
\end{proof}

Note that the values of $h_2,h_3,h_4$ disappear into the error term
of Lemma~\ref{l6.3}. This means that Type-$2$ switchings can be
analysed in the same way using Corollary~\ref{c5.8}. The values
of $h_3$ and $h_4$ again disappear into the error term, so we
can analyse Type-$3$ switchings in the same way using Corollary~\ref{c5.12}.
As these two analyses are essentially the same as Lemma~\ref{l6.3}, we
will just state the results without proof.

\begin{lemma}\label{l6.4}
Assume $r^{-2}n\leq m=o(r^{-3}n^{ \frac32})$.
If $|\mathcal{C}_{0,0,h_3,h_4}^{++}|\neq0$, then
\begin{align*}
\sum_{h_2=0}^{M_2}&\sum_{h_1=0}^{M_1} \frac{|\mathcal{C}_{h_1,h_2,h_3,h_4}^{++}|}{|\mathcal{C}_{0,0,h_3,h_4}^{++}|}\\
&{}=\exp\biggl[ \frac{\bigl[6 \binom{r}{4}+3 \binom{r}{3}\bigr] \binom{3r-4}{r}
 \binom{2r-4}{r-2}m^3\bigl[ \binom{n}{3r-4}- \binom{r}{2}m \binom{n-2}{3r-6}\bigr]}
{\bigl[N- \binom{r}{2}m \binom{n-2}{r-2}\bigr]^3}+O\Bigl( \frac{r^6m^2}{n^3}\Bigr)\biggr].\quad\qed
\end{align*}
\end{lemma}

\begin{lemma}\label{l6.5}
Assume $r^{-2}n\leq m=o(r^{-3}n^{ \frac32})$.
If $|\mathcal{C}_{0,0,0,h_4}^{++}|\neq0$, then
\begin{align*}
\sum_{h_3=0}^{M_3}&\sum_{h_2=0}^{M_2}\sum_{h_1=0}^{M_1} \frac{|\mathcal{C}_{h_1,h_2,h_3,h_4}^{++}|}
{|\mathcal{C}_{0,0,0,h_4}^{++}|}\\
&=\exp\biggl[ \frac{\bigl[6 \binom{r}{4}+3 \binom{r}{3}+ \frac{1}{6} \binom{r}{2}\bigr]
 \binom{3r-4}{r} \binom{2r-4}{r-2}m^3\bigl[ \binom{n}{3r-4}- \binom{r}{2}m \binom{n-2}{3r-6}\bigr]}
{\bigl[N- \binom{r}{2}m \binom{n-2}{r-2}\bigr]^3}+O\Bigl( \frac{r^6m^2}{n^3}\Bigr)\biggr].\qed
\end{align*}
\end{lemma}

\begin{lemma}\label{l6.6}
Assume $r^{-2}n\leq m=o(r^{-3}n^{ \frac32})$.
Then
\begin{align*}
\sum_{h_4=0}^{M_4}\sum_{h_3=0}^{M_3}&\sum_{h_2=0}^{M_2}\sum_{h_1=0}^{M_1} \frac{
|\mathcal{C}_{h_1,h_2,h_3,h_4}^{++}|}{|\mathcal{C}_{0,0,0,0}^{++}|}\\
&=\exp\biggl[ \frac{ \binom{2r-2}{2} \binom{2r-4}{r-2}m^2\bigl[ \binom{n}{2r-2}-
 \frac{(r^2-r-1)}{(r-1)(2r-3)} \binom{r}{2}m \binom{n-2}{2r-4}\bigr]}{2\bigl[N- \binom{r}{2}m \binom{n-2}{r-2}\bigr]^2}\\
&{\qquad}- \frac{ \binom{2r-2}{2}^2 \binom{2r-4}{r-2}^2m^3\bigl[ \binom{n}{2r-2}- \frac{(r^2-r-1)}{(r-1)(2r-3)}
 \binom{r}{2}m \binom{n-2}{2r-4}\bigr]^2}{2\bigl[N- \binom{r}{2}m \binom{n-2}{r-2}\bigr]^4}\\
&{\qquad}+ \frac{\bigl[6 \binom{r}{4}+3 \binom{r}{3}+ \frac{1}{6} \binom{r}{2}\bigr] \binom{3r-4}{r}
 \binom{2r-4}{r-2}m^3\bigl[ \binom{n}{3r-4}- \binom{r}{2}m \binom{n-2}{3r-6}\bigr]}
{\bigl[N- \binom{r}{2}m \binom{n-2}{r-2}\bigr]^3}+O\Bigl( \frac{r^6m^2}{n^3}\Bigr)\biggr].
\end{align*}
\end{lemma}

\begin{proof}
Due to the larger number of Type-$4$ clusters, Lemma~\ref{l6.2} is not
accurate enough and we must use Lemma~\ref{l6.1}.
Let $h_4'$ be the first value of $h_4\leq M_4$
such that $\mathcal{C}_{0,0,0,h_4}^{++}=\emptyset$
or $h_4'=M_4+1$ if no such value exists, which is  defined in Corollary~\ref{c5.16} $(b)$.
Note that
\[
\sum_{h_4=0}^{M_4}\sum_{h_3=0}^{M_3}\sum_{h_2=0}^{M_2}\sum_{h_1=0}^{M_1} \frac{|\mathcal{C}_{h_1,h_2,h_3,h_4}^{++}|}{|\mathcal{C}_{0,0,0,0}^{++}|}
=\sum_{h_4=0}^{h_4'-1}\biggl( \frac{|\mathcal{C}_{0,0,0,h_4}^{++}|}{|\mathcal{C}_{0,0,0,0}^{++}|}
\sum_{h_3=0}^{M_3}\sum_{h_2=0}^{M_2}\sum_{h_1=0}^{M_1} \frac{|\mathcal{C}_{h_1,h_2,h_3,h_4}^{++}|}{|\mathcal{C}_{0,0,0,h_4}^{++}|}\biggr).
\]
Thus, by the definition of $h_4'$, $|\mathcal{C}_{0,0,0,h_4}^{++}|\neq0$ for $0\leq h_4<h_4'$.
By Lemma~\ref{l6.5}, we have
\begin{align*}
\sum_{h_4=0}^{M_4}&\sum_{h_3=0}^{M_3}\sum_{h_2=0}^{M_2}\sum_{h_1=0}^{M_1}
 \frac{|\mathcal{C}_{h_1,h_2,h_3,h_4}^{++}|}{|\mathcal{C}_{0,0,0,0}^{++}|}\\
&=\exp\biggl[ \frac{\bigl[6 \binom{r}{4}+3 \binom{r}{3}+ \frac{1}{6} \binom{r}{2}\bigr]
 \binom{3r-4}{r} \binom{2r-4}{r-2}m^3\bigl[ \binom{n}{3r-4}- \binom{r}{2}m \binom{n-2}{3r-6}\bigr]}
{\bigl[N- \binom{r}{2}m \binom{n-2}{r-2}\bigr]^3}+O\Bigl( \frac{r^6m^2}{n^3}\Bigr)\biggr] \\
&{\qquad}\times\sum_{h_4=0}^{h_4'-1}
 \frac{|\mathcal{C}_{0,0,0,h_4}^{++}|}{|\mathcal{C}_{0,0,0,0}^{++}|}.
\end{align*}

By Remark~\ref{r4.3} and Remark~\ref{r4.4}, we have $|\mathcal{C}_{0,0,0,0}^{++}|\neq0$.
Define $n_{4,0},\ldots,n_{4,M_4}$ by $n_{4,0}=1$,
\[
n_{4,h_4}= \frac{|\mathcal{C}_{0,0,0,h_4}^{++}|}{|\mathcal{C}_{0,0,0,0}^{++}|}
\]
for $1\leq h_4<h_4'$ and $n_{4,h_4}=0$ for $h_4'\leq h_4\leq M_4$,  where
$M_4$ is shown in the equation~\eqref{e3.1}.
Note that as $n\to\infty$, $r^{-2}n\leq m=o(r^{-3}n^{ \frac32})$ and $h_4\leq M_4$,
\[
 \binom{m-2(h_4-1)}{2}= \frac{m^2}{2}\biggl(1+O\Bigl( \frac{r^8m^2}{n^4}+
 \frac{1}{m}\Bigr)\biggr)\Bigl(1-(h_4-1) \frac{4}{m}\Bigr).
\]
By Corollary~\ref{c5.16} $(b)$, we have for $1\leq h_4\leq h_4'$,
\begin{equation}\label{e6.2}
\begin{split}
 \frac{|\mathcal{C}_{0,0,0,h_4}^{++}|}{|\mathcal{C}_{0,0,0,0}^{++}|}
&= \frac{1}{h_4} \frac{|\mathcal{C}_{0,0,0,h_4-1}^{++}|}{|\mathcal{C}_{0,0,0,0}^{++}|}
 \frac{ \binom{2r-2}{2} \binom{2r-4}{r-2}m^2\bigl[ \binom{n}{2r-2}- \frac{(r^2-r-1)}{(r-1)
(2r-3)} \binom{r}{2}m \binom{n-2}{2r-4}\bigr]}
 {2\bigl[N- \binom{r}{2}m \binom{n-2}{r-2}\bigr]^2} \\
&{}\times\biggl(1+O\biggl(\frac{r^8m^2}{n^4}+ \frac{1}{m}
+ \frac{r^6mn\log (r^{-2}n)+r^8m^2}{n^{4}}\biggr)\biggr)\Bigl(1-(h_4-1) \frac{4}{m}\Bigr).
\end{split}
\end{equation}
Define
\begin{align*}
A(4,h_4)&= \frac{ \binom{2r-2}{2} \binom{2r-4}{r-2}m^2\bigl[ \binom{n}{2r-2}-
 \frac{(r^2-r-1)}{(r-1)(2r-3)} \binom{r}{2}m \binom{n-2}{2r-4}\bigr]}
{2\bigl[N- \binom{r}{2}m \binom{n-2}{r-2}\bigr]^2}\\
&{\qquad}\times
\biggl(1+O\biggl( \frac{r^8m^2}{n^4}+ \frac{1}{m}+
 \frac{r^6mn\log (r^{-2}n)+r^8m^2}{n^{4}}\biggr)\biggr)
\end{align*}
for $1\leq h_4< h_4'$ and $A(4,h_4)=0$ for $h_4'\leq h_4\leq M_4$,
\[
B(4,h_4)= \frac{4}{m}\text{~~and~~}
\delta_{h_4}=0
\]
for $1\leq h_4\leq M_4$. Define $K=1$, $c=3^{4}$ and $\gamma_j=0$ for $0\leq j\leq K$.
By~\eqref{e6.2}, we further have
\[
 \frac{n_{4,h_4}}{n_{4,h_4-1}}=
 \frac{A(4,h_4)}{h_4}\bigl(1-(h_4-1)B(4,h_4)\bigr).
\]
Let $A_{4,1}=\min_{h_4=1}^{M_4}A(4,h_4)$, $A_{4,2}=\max_{h_4=1}^{M_4}A(4,h_4)$,
$B_{4,1}=\min_{h_4=1}^{M_4}B(4,h_4)= \frac{4}{m}$, $B_{4,2}=\max_{h_4=1}^{M_4}B(4,h_4)= \frac{4}{m}$.
Since $Ac\leq \frac{3^4r^4m^2}{4n^2}<M_4-K+1$ and $BM_4=o(1)<1$ for all $A\in[A_{4,1},A_{4,2}]$
and  $B\in[B_{4,1},B_{4,2}]$, Lemma~\ref{l6.1} applies and we obtain
\begin{align*}
\sum_{h_4=0}^{M_4} \frac{|\mathcal{C}_{0,0,0,h_4}^{++}|}{|\mathcal{C}_{0,0,0,0}^{++}|}
&=\exp\biggl[ \frac{ \binom{2r-2}{2} \binom{2r-4}{r-2}m^2\bigl[ \binom{n}{2r-2}-
 \frac{(r^2-r-1)}{(r-1)(2r-3)} \binom{r}{2}m \binom{n-2}{2r-4}\bigr]}{2\bigl[N- \binom{r}{2}m \binom{n-2}{r-2}\bigr]^2}\\
&{\qquad}- \frac{ \binom{2r-2}{2}^2 \binom{2r-4}{r-2}^2m^3\bigl[ \binom{n}{2r-2}-
 \frac{(r^2-r-1)}{(r-1)(2r-3)} \binom{r}{2}m \binom{n-2}{2r-4}\bigr]^2}{2\bigl[N- \binom{r}{2}m \binom{n-2}{r-2}\bigr]^4}\\
 &{\qquad}+O\biggl(\frac{r^{12}m^4}{n^6}+ \frac{r^4m}{n^2}+
  \frac{r^{10}m^3n\log(r^{-2}n) +r^{12}m^4}{n^6}\biggr)\biggr]+O\Bigl(\Bigl( \frac{2e}{3^{4}}\Bigr)^{M_4}\Bigr).
\end{align*}
Note that $O\bigl( \frac{r^{12}m^4}{n^6}+ \frac{r^4m}{n^2}+
 \frac{r^{10}m^3n\log(r^{-2}n) +r^{12}m^4}{n^6}\bigr)=O\bigl(
 \frac{r^6m^2}{n^3}\bigr)$  and $O\bigl(\bigl( \frac{2e}{3^{4}}\bigr)^{M_3}\bigr)=O\bigl( \frac{r^6}{n^3}\bigr)$,
which further leads to
\begin{align*}
\sum_{h_4=0}^{M_4} \frac{|\mathcal{C}_{0,0,0,h_4}^{++}|}{|\mathcal{C}_{0,0,0,0}^{++}|}
&=\exp\biggl[ \frac{ \binom{2r-2}{2} \binom{2r-4}{r-2}m^2\bigl[ \binom{n}{2r-2}-
 \frac{(r^2-r-1)}{(r-1)(2r-3)} \binom{r}{2}m \binom{n-2}{2r-4}\bigr]}{2\bigl[N-
 \binom{r}{2}m \binom{n-2}{r-2}\bigr]^2}\\
&{\qquad}- \frac{ \binom{2r-2}{2}^2 \binom{2r-4}{r-2}^2m^3
\bigl[ \binom{n}{2r-2}- \frac{(r^2-r-1)}{(r-1)(2r-3)} \binom{r}{2}m \binom{n-2}{2r-4}\bigr]^2}
{2\bigl[N- \binom{r}{2}m \binom{n-2}{r-2}\bigr]^4}+O\Bigl( \frac{r^6m^2}{n^3}\Bigr)\biggr].
\end{align*}
This gives the expression in the lemma statement.
%Then
%\[
%\begin{aligned}[b]
%&\sum_{h_4=0}^{M_4}\sum_{h_3=0}^{M_3}\sum_{h_2=0}^{M_2}\sum_{h_1=0}^{M_1}
% \frac{|\mathcal{C}_{h_1,h_2,h_3,h_4}^{++}|}{|\mathcal{C}_{0,0,0,0}^{++}|}\\
%&=\exp\biggl[ \frac{ \binom{2r-2}{2} \binom{2r-4}{r-2}m^2\bigl[ \binom{n}{2r-2}-
% \frac{(r^2-r-1)}{(r-1)(2r-3)} \binom{r}{2}m \binom{n-2}{2r-4}\bigr]}
%{2\bigl[N- \binom{r}{2}m \binom{n-2}{r-2}\bigr]^2}\\
%&- \frac{ \binom{2r-2}{2}^2 \binom{2r-4}{r-2}^2m^3\bigl[ \binom{n}{2r-2}- \frac{(r^2-r-1)}{(r-1)(2r-3)}
% \binom{r}{2}m \binom{n-2}{2r-4}\bigr]^2}{2\bigl[N- \binom{r}{2}m \binom{n-2}{r-2}\bigr]^4}\\
%&+ \frac{\bigl[6 \binom{r}{4}+3 \binom{r}{3}+ \frac{1}{6} \binom{r}{2}\bigr] \binom{3r-4}{r}
% \binom{2r-4}{r-2}m^3\bigl[ \binom{n}{3r-4}- \binom{r}{2}m \binom{n-2}{3r-6}\bigr]}{\bigl[N- \binom{r}{2}m \binom{n-2}{r-2}\bigr]^3}+O\Bigl( \frac{r^6m^2}{n^3}\Bigr)\biggr].\qedhere
%\end{aligned}
%\]
\end{proof}

%\red{Omit this and bypass it in the introduction.
%\begin{theorem}\label{t6.7}
%Let $r=r(n)\geq 3$ be an integer with $r=o(n^{ \frac12})$ and $m=m(n)$
%be an integer with $r^{-2}n\leq m=o(r^{-3}n^{ \frac32})$.
%Suppose that $n\to \infty$. Then
%\[
%\begin{aligned}[b]
%|\mathcal{L}_r(n,m)|
%&= \binom{N}{m}
%\exp\biggl[- \frac{ \binom{2r-2}{2} \binom{2r-4}{r-2}m^2\bigl[ \binom{n}{2r-2}- \frac{(r^2-r-1)}{(r-1)(2r-3)}
% \binom{r}{2}m \binom{n-2}{2r-4}\bigr]}{2\bigl[N- \binom{r}{2}m \binom{n-2}{r-2}\bigr]^2}\\&
%+ \frac{ \binom{2r-2}{2}^2 \binom{2r-4}{r-2}^2m^3\bigl[ \binom{n}{2r-2}- \frac{(r^2-r-1)}{(r-1)(2r-3)}
% \binom{r}{2}m \binom{n-2}{2r-4}\bigr]^2}{2\bigl[N- \binom{r}{2}m \binom{n-2}{r-2}\bigr]^4}\\
%&- \frac{\bigl[6 \binom{r}{4}+3 \binom{r}{3}+ \frac{1}{6} \binom{r}{2}\bigr]
% \binom{3r-4}{r} \binom{2r-4}{r-2}m^3\bigl[ \binom{n}{3r-4}- \binom{r}{2}
%m \binom{n-2}{3r-6}\bigr]}{\bigl[N- \binom{r}{2}m \binom{n-2}{r-2}\bigr]^3}
%+O\Bigl( \frac{r^6m^2}{n^3}\Bigr)\biggr].
%\end{aligned}
%\]
%\end{theorem}
%%
%\begin{proof}  By Remark~\ref{r4.5} and Lemma~\ref{l6.6},
%we obtain Theorem~\ref{t6.7}. \end{proof}}

\section{The case $\log\, (r^{-2}n)\leq m=O(r^{-2}n)$}\label{s:7}

We consider the case $\log (r^{-2}n)\leq m=O(r^{-2}n)$
in Theorem~\ref{t1.1}.
Recall that these inequalities imply $r=o(n^{\frac12})$.
Define
\begin{equation}\label{e7.1}
\begin{split}
M_0^*&=\Bigl\lceil\log (r^{-2}n)\Bigr\rceil, \\
M_0&=\Bigl\lceil\log (r^{-2}n)\Bigr\rceil+2, \\
M_4&=\Bigl\lceil\log (r^{-2}n)\Bigr\rceil.
\end{split}
\end{equation}
Let $\mathcal{H}_r^+(n,m)\subseteq\mathcal{H}_r(n,m)$
be the set of $r$-graphs $H$ which satisfy
properties $\bf(a)$ to $\bf(e)$:

$\bf(a)$\  The intersection of any two edges contains at most two vertices.

$\bf(b)$\  $H$ only contains one type of cluster (Type-$4$ cluster).
(This implies that any three edges involve at least $3r-3$ vertices. Thus, if there are
two edges, for example $\{e_1,e_2\}$, such that $|e_1\cup e_2|= 2r-2$, then $|(e_1\cup e_2)\cap e|\leq 1$  for any
edge $e$ other than $\{e_1,e_2\}$ of $H$.)

$\bf(c)$\ Any two distinct Type-$4$ clusters in $H$ are vertex-disjoint.
(This implies that any four edges involve at least $4r-4$ vertices.)

$\bf(d)$\ There are at most $M_4$  Type-$4$ clusters in $H$.

$\bf(e)$\ $\deg (v)\leq M_0$ for every vertex $v\in [n]$.

Similarly, we further define  $\mathcal{H}_r^{++}(n,m)\subseteq\mathcal{H}_r^+(n,m)$
to be the set of $r$-graphs $H$ obtained
by replacing the property $\bf(e)$
with a stronger constraint

$\bf(e^*)$\ $\deg (v)\leq M_0^*$ for every vertex $v\in [n]$.

\begin{remark}\label{r7.1}
From property $\bf(e)$, it easily follows that
\[
\sum_{v\in [n]} \binom{\deg (v)}{2}=O\Bigl(M_0\sum_{v\in [n]}\deg _0(v)\Bigr)
=O\Bigl(rm\log (r^{-2}n)\Bigr)
\]
for $H\in\mathcal{H}_r^+(n,m)$.
\end{remark}

Similarly to Section~\ref{s:3}, we find that the number of $r$-graphs
 in $\mathcal{H}_r(n,m)$
not satisfying the properties of $\mathcal{H}^+_r(n,m)$ and
$\mathcal{H}^{++}_r(n,m)$ is quite small.

\begin{theorem}\label{t7.2}
Assume that $\log (r^{-2}n)\leq m=O(r^{-2}n)$ and
$n\to \infty$. Then
\[
 \frac{|\mathcal{H}^+_r(n,m)|}{|\mathcal{H}_r(n,m)|}=1-O\Bigl( \frac{r^6m^2}{n^3}\Bigr),\quad  \frac{|\mathcal{H}^{++}_r(n,m)|}{|\mathcal{H}_r(n,m)|}=1-O\Bigl( \frac{r^6m^2}{n^3}\Bigr).
\]
\end{theorem}

\begin{proof}
The proof is much the same as that of Theorem~\ref{t3.2}, so we will
omit the proofs for $\bf(a)$--$\bf(d)$.

To prove property $\bf(e^*)$, define $d=M_0^*+1$.
The expected number of sets consisting of a vertex $v$
and $d$ edges that include $v$ is
\[
n \binom{n-1 }{r-1}^d \frac{1}{d!}\Bigl(\frac{m}{N}\Bigr)^{d}
=O\Bigl(n\Bigl( \frac{rem}{dn}\Bigr)^{d}\Bigr)
=O\Bigl(n\Bigl( \frac{e}{rd}\Bigr)^{d}\Bigr)
=O\Bigl( \frac{r^6}{n^3}\Bigr),
\]
where the second equality is true because $d!\geq \bigl( \frac{d}{e}\bigr)^{d}$ and $m=O(r^{-2}n)$,
and the last equality is true because of the choice $d>\log (r^{-2}n)$.
\end{proof}

\begin{remark}\label{r7.3}
Let
$\mathcal{C}_{h_4}^{+}$ {\rm(}resp. $\mathcal{C}_{h_4}^{++}${\rm)}
be the set of $r$-graphs $H\in \mathcal{H}_r^+(n,m)$ {\rm(}resp.
$H\in \mathcal{H}_r^{++}(n,m)${\rm)}
with exactly  $h_4$ Type-$4$ clusters. Then, by Theorem~\ref{t7.2},
\[
|\mathcal{H}_r^+(n,m)|=\sum_{h_4=0}^{M_4}|\mathcal{C}_{h_4}^{+}|,\quad
|\mathcal{H}_r^{++}(n,m)|=\sum_{h_4=0}^{M_4}|\mathcal{C}_{h_4}^{++}|.
\]
By the same discussion as in Section~\ref{s:4}, we also have
\[
|\mathcal{C}_{0}^{++}|=\Bigl(1-O\Bigl( \frac{r^6}{n^3}\Bigr)\Bigr)|\mathcal{L}_r(n,m)|\quad
\text{and}\quad |\mathcal{C}_{h_4}^{++}|=\Bigl(1-O\Bigl( \frac{r^6}{n^3}\Bigr)\Bigr)|\mathcal{C}_{h_4}^{+}|.
\]
We also have $|\mathcal{L}_r(n,m)|\neq0$
and $|\mathcal{C}_{0}^{++}|\neq0$.
Thus,
\[
 \frac{1}{\mathbb{P}_r(n,m)}=\sum_{h_4=0}^{M_4} \frac{|\mathcal{C}_{h_4}^{++}|}{|\mathcal{L}_r(n,m)|}
\Bigl(1-O\Bigl( \frac{r^6m^2}{n^3}\Bigr)\Bigr)
=\sum_{h_4=0}^{M_4} \frac{|\mathcal{C}_{h_4}^{++}|}{|\mathcal{C}_{0}^{++}|}
\Bigl(1-O\Bigl( \frac{r^6m^2}{n^3}\Bigr)\Bigr).
\]
\end{remark}

By Remark~\ref{r7.1} and the same arguments as used for Lemma~\ref{l3.5}
and Corollary~\ref{c5.4}, we also have the following two lemmas.

\begin{lemma}\label{l7.5}
Assume  $\log(r^{-2}n)\leq m=O(r^{-2}n)$.
Let $H\in \mathcal{H}_r^{+}(n,m-\xi)$
and let $N_t$ be  the set of $t$-sets of $[n]$ of which
no two vertices belong to the same edge of $H$, where
$r\leq t\leq 2r-2$ and $\xi=O(1)$. Then
\[
|N_t|=\biggl[ \binom{n}{t}- \binom{r}{2}m \binom{n-2}{t-2}\biggr]
\biggl(1+O\Bigl( \frac{r^4}{n^2}+ \frac{r^6m\log(r^{-2}n)}{n^3}\Bigr)\biggr).
\]
\end{lemma}

\begin{lemma}\label{l7.6}
Assume $\log(r^{-2}n)\leq m=O(r^{-2}n)$.
With notation as above, \\
$(a)$\ If $\mathcal{C}_{h_4}^{++}=\emptyset$, then $\mathcal{C}_{h_4+1}^{++}=\emptyset$.\\
$(b)$\ Let $h_4'$ be the first
value of $h_4\leq M_4$
such that $\bigl|\mathcal{C}_{h_4}^{++}\bigl|=0$, or $h_4'=M_4+1$ if no such value exists.
Suppose that $n\to\infty$, then uniformly for $1\leq h_4< h_4'$,
\begin{align*}
 \frac{|\mathcal{C}_{h_4}^{++}|}{|\mathcal{C}_{h_4-1}^{++}|}&=
 \frac{ \binom{2r-2}{2} \binom{2r-4}{r-2} \binom{m-2(h_4-1)}{2}\bigl[ \binom{n}{2r-2}
- \binom{r}{2}m \binom{n-2}{2r-4}\bigr]}{h_4\bigl[N- \binom{r}{2}m \binom{n-2}{r-2}\bigr]^2}\\
&{\qquad}\times\biggl(1+O\Bigl( \frac{r^4}{n^2}+ \frac{r^6m\log(r^{-2}n)}{n^3}\Bigr)\biggr).
\end{align*}
\end{lemma}

\begin{lemma}\label{l7.7}
Assume $\log(r^{-2}n)\leq m=O(r^{-2}n)$. Then
\[
\sum_{h_4=0}^{M_4} \frac{|\mathcal{C}_{h_4}^{++}|}{|\mathcal{C}_{0}^{++}|}
=\exp\biggl[ \frac{[m]_2 \binom{2r-2}{2} \binom{2r-4}{r-2}\bigl[ \binom{n}{2r-2}-
 \binom{r}{2}m \binom{n-2}{2r-4}\bigr]}{2\bigl[N- \binom{r}{2}m \binom{n-2}{r-2}\bigr]^2}
+O\Bigl( \frac{r^6m^2}{n^3}\Bigr)\biggr].
\]
\end{lemma}

\begin{proof} Let $h_4'$ be the first
value of $h_4\leq M_4$
such that $|\mathcal{C}_{h_4}^{++}|=0$ or $h_4'=M_4+1$ if no such value exists,
which is defined in Lemma~\ref{l7.5}. By Remark~\ref{r7.3},
we have $|\mathcal{C}_{0}^{++}|\neq0$.
Define $n_{0},\ldots,n_{M_4}$ by $n_{0}=1$,
\[
n_{h_4}= \frac{|\mathcal{C}_{h_4}^{++}|}{|\mathcal{C}_{0}^{++}|}
\]
for $1\leq h_4<h_4'$ and $n_{h_4}=0$ for $h_4'\leq h_4\leq M_4$,  where
$M_4$ is shown in~\eqref{e7.1}.

Note that
\[
 \binom{m-2(h_4-1)}{2}= \frac{[m]_2}{2}\Bigl(1-(h_4-1) \frac{4m-4h_4+2}{m(m-1)}\Bigr).
\]
By Lemma~\ref{l7.5}, for $1\leq h_4\leq h_4'$, we have
\begin{equation}\label{e7.2}
\begin{split}
 \frac{|\mathcal{C}_{h_4}^{++}|}{|\mathcal{C}_{0}^{++}|}&= \frac{1}{h_4} \frac{|\mathcal{C}_{h_4-1}^{++}|}
{|\mathcal{C}_{0}^{++}|}
 \frac{[m]_2 \binom{2r-2}{2} \binom{2r-4}{r-2}\bigl[ \binom{n}{2r-2}- \binom{r}{2}
m \binom{n-2}{2r-4}\bigr]}{2\bigl[N- \binom{r}{2}m \binom{n-2}{r-2}\bigr]^2} \\
&{\qquad}\times\biggl(1+O\Bigl(\frac{r^4}{n^2}+ \frac{r^6m\log(r^{-2}n)}
{n^3}\Bigr)\biggr)\biggl(1-(h_4-1) \frac{4m-4h_4+2}{m(m-1)}\biggr).
\end{split}
\end{equation}
Define
\[
A(h_4)= \frac{[m]_2 \binom{2r-2}{2} \binom{2r-4}{r-2}\bigl[ \binom{n}{2r-2}-
 \binom{r}{2}m \binom{n-2}{2r-4}\bigr]}{2\bigl[N- \binom{r}{2}m \binom{n-2}{r-2}\bigr]^2}
 \biggl(1+O\Bigl(\frac{r^4}{n^2}+ \frac{r^6m\log(r^{-2}n)}{n^3}\Bigr)\biggr)
\]
for $1\leq h_4< h_4'$, $A(h_4)=0$ for $h_4'\leq h_4\leq M_4$ and
$B(h_4)= \frac{4m-4h_4+2}{m(m-1)}=O\bigl( \frac{1}{m}\bigr)$ as $m\geq \log (r^{-2}n)$. Let $A_{1}=\min_{h_4=1}^{M_4}A(h_4)$,
$A_{2}=\max_{h_4=1}^{M_4}A(h_4)$, $C_{1}=\min_{h_4=1}^{M_4}A(h_4)B(h_4)$
and $C_{2}=\max_{h_4=1}^{M_4}A(h_4)B(h_4)$.

We further have
\[
 \frac{|\mathcal{C}_{h_4}^{++}|}{|\mathcal{C}_{0}^{++}|}=
\frac{A(h_4)}{h_4} \frac{|\mathcal{C}_{h_4-1}^{++}|}{|\mathcal{C}_{0}^{++}|}
   \bigl(1-(h_4-1)B(h_4)\bigr).
\]
Note that $\max\{A/M_4,|C|\}=o(1)$ for all $A\in [A_{1},A_{2}]$ and $C\in[C_{1},C_{2}]$
as $\log(r^{-2}n)\leq m=O(r^{-2}n)$ and $h_4\leq M_4=\lceil\log (r^{-2}n)\rceil$.
Let $\hat{c}= \frac{1}{2\cdot 3^4}$, then $\max\{A/M_4,|C|\}\leq \hat{c}< \frac{1}{3}$.
Lemma~\ref{l6.2} applies to obtain
\begin{align*}
\sum_{h_4=0}^{M_4} \frac{|\mathcal{C}_{h_4}^{++}|}{|\mathcal{C}_{0}^{++}|}
&=\exp\biggl[ \frac{[m]_2 \binom{2r-2}{2} \binom{2r-4}{r-2}\bigl[ \binom{n}{2r-2}- \binom{r}{2}
m \binom{n-2}{2r-4}\bigr]}{2\bigl[N- \binom{r}{2}m \binom{n-2}{r-2}\bigr]^2}\\
&{\qquad}+O\Bigl( \frac{r^8m^2}{n^{4}}+ \frac{r^{10}
m^3\log(r^{-2}n)}{n^5}\Bigr)\biggr]+O\Bigl(\Bigl( \frac{e}{3^{4}}\Bigr)^{M_4}\Bigr)\\
&=\exp\biggl[ \frac{[m]_2 \binom{2r-2}{2} \binom{2r-4}{r-2}\bigl[ \binom{n}{2r-2}-
 \binom{r}{2}m \binom{n-2}{2r-4}\bigr]}{2\bigl[N- \binom{r}{2}
m \binom{n-2}{r-2}\bigr]^2}+O\Bigl( \frac{r^6m^2}{n^3}\Bigr)\biggr],
\end{align*}
where the last equality is true because $O\bigl(  \frac{r^8m^2}{n^{4}}+
 \frac{r^{10}m^3\log(r^{-2}n)}{n^5}\bigr)
 =O\bigl( \frac{r^6m^2}{n^3}\bigr)$ as $m=O(r^{-2}n)$ and $O\bigl(\bigl( \frac{e}{3^{4}}\bigr)^{M_4}\bigr)=O\bigl( \frac{r^6}{n^3}\bigr)$.
\end{proof}

%\red{[Roll this into intro.]
%By Remark~\ref{r7.3}, we have the following theorem.
%\begin{theorem}\label{t7.8}
%Let $r=r(n)\geq 3$ be an integer with $r=o(n^{ \frac12})$ and $m=m(n)$ be an
%integer with $\log(r^{-2}n)\leq m=O(r^{-2}n)$.
%Suppose that $n\to \infty$. Then
%\[
%|\mathcal{L}_r(n,m)|&= \binom{N}{m}\exp\biggl[- \frac{[m]_2
% \binom{2r-2}{2} \binom{2r-4}{r-2}\bigl[ \binom{n}{2r-2}- \binom{r}{2}
%m \binom{n-2}{2r-4}\bigr]}{2\bigl[N- \binom{r}{2}m \binom{n-2}{r-2}\bigr]^2}
%+O\Bigl( \frac{r^6m^2}{n^3}\Bigr)\biggr].
%\]
%\end{theorem}}
%

\section{The case $1\leq m= O(\log (r^{-2}n))$}\label{s:8}

Let $\mathcal{H}_r^+(n,m)\subset\mathcal{H}_r(n,m)$  be the set of
$r$-graphs $H$ which satisfy properties $\bf(a)$ to $\bf(d)$:

$\bf(a)$\  The intersection of any two edges contains at most two vertices.

$\bf(b)$\  $H$  only contains one type of cluster (Type-$4$ cluster).
(This implies that any three edges involve at least $3r-3$ vertices. Thus, if there are
two edges, for example $\{e_1,e_2\}$, such that $|e_1\cup e_2|= 2r-2$,
then $|(e_1\cup e_2)\cap e|\leq 1$ for any
edge $e$ other than $\{e_1,e_2\}$ of $H$.)

$\bf(c)$\ Any two distinct Type-$4$ clusters in $H$ are vertex-disjoint.
(This implies that any four edges involve at least $4r-4$ vertices.)

$\bf(d)$\ There are at most two  Type-$4$ clusters in $H$. (This implies that
any six edges involve at least $6r-5$ vertices.)

\begin{theorem}\label{t8.1}
Assume $1\leq m= O(\log (r^{-2}n))$ and $n\to \infty$. Then
\[
 \frac{|\mathcal{H}^+_r(n,m)|}{|\mathcal{H}_r(n,m)|}=1-O\Bigl( \frac{r^6m^2}{n^3}\Bigr).
\]
\end{theorem}

\begin{proof} Consider $H\in \mathcal{H}_r(n, m)$ chosen uniformly at random.
We can apply Lemma~\ref{l2.2} several times to show that  $H$ satisfies properties
$\bf(a)$-$\bf(d)$ with probability $1-O\bigl( \frac{r^6m^2}{n^3}\bigr)$.
We only prove the property $\bf(d)$ because the proof of other conditions are
exactly same with the proof in Theorem~\ref{t3.2}.

Applying Lemma~\ref{l2.2} with $t=6$ and $\alpha=6$,
the expected number of sets of six edges involving at most $6r-6$ vertices is
\[
O\Bigl( \frac{r^{12}m^6}{n^{6}}\Bigr)=O\Bigl( \frac{r^6m^2}{n^3}\Bigr)
\]
because $m=O(\log (r^{-2}n))$.
Hence, property $\bf(d)$ holds with probability $1-O\bigl( \frac{r^6m^2}{n^3}\bigr)$.
\end{proof}

\begin{remark}\label{r8.2}
By Theorem~\ref{t8.1}, for a nonnegative integer  $h_4$, let $\mathcal{C}_{h_4}^{+}$
be the set of $r$-graphs $H\in \mathcal{H}_r^+(n,m)$
with exactly  $h_4$
Type-$4$ clusters. Thus, we have $|\mathcal{H}_r^+(n,m)|=\sum_{h_4=0}^2|\mathcal{C}_{h_4}^{+}|$,
$|\mathcal{C}_{0}^{+}|=|\mathcal{L}_r(n,m)|\neq\emptyset$ and
\[
 \frac{1}{\mathbb{P}_r(n,m)}=\sum_{h_4=0}^2 \frac{|\mathcal{C}_{h_4}^{+}|}
{|\mathcal{L}_r(n,m)|}\biggl(1-O\Bigl( \frac{r^6m^2}{n^3}\Bigr)\biggr)
=\sum_{h_4=0}^2 \frac{|\mathcal{C}_{h_4}^{+}|}{|\mathcal{C}_{0}^{+}|}
  \biggl(1-O\Bigl( \frac{r^6m^2}{n^3}\Bigr)\biggr).
\]
\end{remark}

By arguments similar to Lemma~\ref{l3.5} and
Lemma~\ref{l7.5}, Corollary~\ref{c5.4} and Lemma~\ref{l7.6}, we have
the following two lemmas.
\begin{lemma}\label{l8.4}
Assume $1\leq m=O(\log (r^{-2}n))$.
Let $H\in \mathcal{H}_r^{+}(n,m-\xi)$
and $N_t$ be  the set of $t$-sets of $[n]$ such that no two vertices belong to the same edge of $H$,
where $r\leq t\leq 2r-2$ and $\xi=O(1)$ are positive integers. Suppose that $n\to\infty$. Then
\[
|N_t|=\biggl[ \binom{n}{t}- \binom{r}{2}m \binom{n-2}{t-2}\biggr]\biggl(1+O\Bigl( \frac{r^4}{n^2}\Bigr)\biggr).
\]
\end{lemma}

\begin{lemma}\label{l8.5}
Assume $1\leq m=O(\log (r^{-2}n))$.
With notation as above, \\
$(a)$\ If $\mathcal{C}_{h_4}^{+}=\emptyset$, then $\mathcal{C}_{h_4+1}^{+}=\emptyset$.
\\
$(b)$\ Let $h_4'$ be the first
value of $h_4\leq 2$
such that $\bigl|\mathcal{C}_{h_4}^{+}\bigl|=0$ or $h_4'=3$ if no such value exists.
Suppose that $n\to\infty$, then uniformly for $1\leq h_4< h_4'$,
\[
 \frac{|\mathcal{C}_{h_4}^{+}|}{|\mathcal{C}_{h_4-1}^{+}|}=
 \frac{ \binom{2r-2}{2} \binom{2r-4}{r-2} \binom{m-2(h_4-1)}{2}\bigl[ \binom{n}{2r-2}
- \binom{r}{2}m \binom{n-2}{2r-4}\bigr]}{h_4\bigl[N- \binom{r}{2}m \binom{n-2}{r-2}\bigr]^2}\biggl(1+O\Bigl( \frac{r^4}{n^2}\Bigr)\biggr).
\]
\end{lemma}

\begin{lemma}\label{l8.6}
Assume $1\leq m=O(\log (r^{-2}n))$. Then
\begin{align*}
 \frac{|\mathcal{C}_{0}^{+}|+|\mathcal{C}_{1}^{+}|+|\mathcal{C}_{2}^{+}|}
{|\mathcal{C}_{0}^{+}|}
&=\exp\biggl[ \frac{[m]_2 \binom{2r-2}{2} \binom{2r-4}{r-2}\bigl[ \binom{n}{2r-2}- \binom{r}{2}m \binom{n-2}{2r-4}\bigr]}
{2\bigl[N- \binom{r}{2}m \binom{n-2}{r-2}\bigr]^2}+O\Bigl( \frac{r^6m^2}{n^3}\Bigr)\biggr].\\
\end{align*}
\end{lemma}

\begin{proof} By Remark~\ref{r8.2}, we have $|\mathcal{C}_{0}^{+}|\neq0$.
By Lemma~\ref{l8.4}, we have
\[
 \frac{|\mathcal{C}_{1}^{+}|}{|\mathcal{C}_{0}^{+}|}=
 \frac{[m]_2 \binom{2r-2}{2} \binom{2r-4}{r-2}\bigl[ \binom{n}{2r-2}- \binom{r}{2}m \binom{n-2}{2r-4}\bigr]}
{2\bigl[N- \binom{r}{2}m \binom{n-2}{r-2}\bigr]^2}\biggl(1+O\Bigl( \frac{r^4}{n^2}\Bigr)\biggr)
\]
and
\begin{align*}
 \frac{|\mathcal{C}_{2}^{+}|}{|\mathcal{C}_{1}^{+}|}=
 \frac{[m-1]_2 \binom{2r-2}{2} \binom{2r-4}{r-2}
   \bigl[ \binom{n}{2r-2}- \binom{r}{2}m \binom{n-2}{2r-4}\bigr]}
{4\bigl[N- \binom{r}{2}m \binom{n-2}{r-2}\bigr]^2}
  \biggl(1+O\Bigl( \frac{r^4}{n^2}\Bigr)\biggr).
\end{align*}
Thus, we have
\begin{align*}
\sum_{h_4=0}^2 \frac{|\mathcal{C}_{h_4}^{+}|}{|\mathcal{C}_{0}^{+}|}
&=\biggl(1+ \frac{[m]_2 \binom{2r-2}{2} \binom{2r-4}{r-2}\bigl[ \binom{n}{2r-2}- \binom{r}{2}
m \binom{n-2}{2r-4}\bigr]}{2\bigl[N- \binom{r}{2}m \binom{n-2}{r-2}\bigr]^2}\\
&{\qquad}+ \frac{[m]_4 \binom{2r-2}{2}^2 \binom{2r-4}{r-2}^2\bigl[ \binom{n}{2r-2}- \binom{r}{2}
m \binom{n-2}{2r-4}\bigr]^2}{8\bigl[N- \binom{r}{2}m \binom{n-2}{r-2}\bigr]^4}+O\Bigl( \frac{r^{8}m^2}{n^4}\Bigr)\biggr),
\end{align*}
which simplifies to the expression in the lemma statement.
\end{proof}

%\red{[Roll into intro]
%By Remark~\ref{r8.2}, we have the following theorem.
%\begin{theorem}\label{t8.7}
% Let $r=r(n)\geq 3$ be an integer with $r=o(n^{ \frac12})$ and
%$m=m(n)$ be an integer with $m= O(\log (r^{-2}n))$.
%Suppose that $n\to \infty$. Then
%\[
%|\mathcal{L}_r(n,m)|= \binom{N}{m}\,
%\exp\biggl[- \frac{[m]_2 \binom{2r-2}{2} \binom{2r-4}{r-2}\bigl[ \binom{n}{2r-2}- \binom{r}{2}m \binom{n-2}{2r-4}\bigr]}
%{2\bigl[ N- \binom{r}{2}m \binom{n-2}{r-2}\bigr]^2}+O\Bigl( \frac{r^6m^2}{n^3}\Bigr)\biggr].
%\]
%\end{theorem}}

\section{Proof of Theorem~\ref{t1.3}}\label{s:9}

As in the statement of Theorem~\ref{t1.3}, we use $m_0=Np$.
The binomial distribution is $\Bin(n,p)$.
Recall that $\mathbb{P}_r(n,m)$ denotes the probability that
an $r$-graph $H\in \mathcal{H}_r(n,m)$ chosen uniformly
 at random is linear. We found an expression for $\mathbb{P}_r(n,m)$
 in Theorem~\ref{t1.1}.
By the law of total probability, we have
\begin{equation}\label{e9.2}
\mathbb{P}\bigl[H_r(n,p)\in \mathcal{L}_r\bigr]=\sum_{m=0}^{N}\mathbb{P}_r(n,m) \binom{N}{m}p^mq^{N-m}.
\end{equation}

In the proof of Theorem~\ref{t1.3},
we need the following lemmas. We firstly show $\mathbb{P}_r(n,m)$
is decreasing in $m$ (Lemma~\ref{l9.1}).  Some approximations will make
use of the
Chernoff inequality (Lemma~\ref{l9.3}) and the normal approximation
of the binomial distribution (Lemmas~\ref{l9.4} and~\ref{l9.5}).

\begin{lemma}\label{l9.1}
 $\mathbb{P}_r(n,m)$ is a non-increasing function of $m$.
\end{lemma}

\begin{proof}
Choosing $m$ distinct edges at random gives the same distribution as
choosing $m-1$ distinct edges at random and then an $m$-th edge at
random distinct from the first $m-1$.  So $\mathbb{P}_r(n,m)\le\mathbb{P}_r(n,m-1)$.
%
%Consider the random hypergraph process
%$\bigl\{H_{r,t}\bigr\}=\bigl\{[n],E_t\bigr\}$, where $E_t=\bigl\{e_1,e_2,\ldots,e_t\bigr\}$ is obtained
%from $E_{t-1}$ by adding a random edge $e_t\notin E_{t-1}$ for $t=0,1,\ldots,N$. Let
%\[
%\tau_H=\min\bigl\{t\ |\ H_{r,t}{\rm{\ is\ not\ linear}}\bigr\}.
%\]
%Thus, we have
%\[
%\mathbb{P}_r(n,m)=\mathbb{P}[\tau_H>m]
%<\mathbb{P}[\tau_H>m-1]
%=\mathbb{P}_r(n,m-1). \qedhere
%\]
\end{proof}

%\begin{lemma}[\cite{mckay02,wong89}]\label{l9.2}
%Let $f(x)$ be a real-valued function such that $f^{(4)}(x)$ is absolutely
%integrable on $(0,\infty)$. Then for $m\geq 1$ we have
%\[
%\sum_{j=0}^{m}f(j)=\int_0^mf(x)dx+ \frac12\bigl(f(0)+f(m)\bigr)- \frac{1}{12}\bigl(f'(0)-f'(m)\bigr)+R(m),
%\]where
%\[
%|R(m)|\leq  \frac{1}{384}\int_0^m\left|f^{(4)}(x)\right|dx.
%\]
%\end{lemma}

\begin{lemma}[\cite{cher52}]\label{l9.3}
For $X\sim \Bin(N, p)$ and any $0<t\leq Np$,
\[
\mathbb{P}\bigl[|X - Np|>t\bigr]<2\exp\left[-t^2/\left(3Np\right)\right].
\]
\end{lemma}

\begin{lemma}\label{l9.4}
Let $r=r(n)\geq 3$ be an integer with $r=o(n^{ \frac12})$ and
$X\sim \Bin(N,p)$, where $r^{-2}n\leq Np=o(r^{-3}n^{ \frac32})$.
Then
\[
\mathbb{P}\bigl[X=\lfloor Np\rfloor+t\bigr]= \frac{1}{\sqrt{2\pi Np}}
\exp\biggl[- \frac{t^2}{2Np}+O\biggl( \frac{\log^3 (r^{-2}n)}{\sqrt{Np}}
+ \frac{r^6(Np)^2}{n^3}\biggr)\biggr],
\]
where $t\in \mathbb{Z}$ and
\[
t\in\biggl[-\log(r^{-2}n)\sqrt{Npq}- \frac{[r]_2^2(Np)^2}{2n^2},\log (r^{-2}n)\sqrt{Npq}\biggr].
\]
\end{lemma}

\begin{proof} Note that
\[
 \binom{N}{\lfloor Np\rfloor+t}p^{\lfloor Np\rfloor+t}q^{\lceil Nq\rceil -t}=
 \binom{N}{\lfloor Np\rfloor}p^{\lfloor Np\rfloor}q^{\lceil Nq\rceil }
 \frac{ \binom{N}{\lfloor Np\rfloor+t}p^tq^{-t}}{ \binom{N}{\lfloor Np\rfloor}}.
\]
By Stirling's formula for $\lfloor Np\rfloor!$ and $\lceil Nq\rceil !$,
%\[
%\begin{aligned}[b]
%\lfloor Np\rfloor !&=\biggl( \frac{\lfloor Np\rfloor}{e}\biggr)^{\lfloor Np\rfloor}
%\sqrt{2\pi  Np}\exp\Bigl[ \frac{\theta}{12Np}\Bigr],\\
%\lceil Nq\rceil !&=\biggl( \frac{\lceil Nq\rceil}{e}\biggr)^{\lceil Nq\rceil}
%\sqrt{2\pi  Nq}\exp\Bigl[ \frac{\theta}{12Nq}\Bigr],
%\end{aligned}
%\]
we have
\[
 \binom{N}{\lfloor Np\rfloor}p^{\lfloor Np\rfloor}q^{\lceil Nq\rceil}
= \frac{1}{\sqrt{2\pi Np}}\exp\Bigl(O\Bigl( \frac{1}{Np}\Bigr)\Bigr)
\]
and as $Np\to\infty$,
\begin{align}
 \frac{ \binom{N}{\lfloor Np\rfloor+t}p^tq^{-t}}{ \binom{N}{\lceil Np\rceil}}
&=\prod_{i=0}^{t-1} \frac{\Bigl(1- \frac{i}{\lceil Nq\rceil}\Bigr)}
        {\Bigl(1+ \frac{1+i}{\lfloor Np\rfloor}\Bigr)} \notag\\
&=\exp\biggl[\sum_{i=0}^{t-1}\ln\biggl( \frac{1- \frac{i}{\lceil Nq\rceil}}
        {1+ \frac{1+i}{\lfloor Np\rfloor}}\biggr)\biggr]\notag\\
&=\exp\biggl[- \frac{t^2}{2Np}+O\biggl(\frac{t}{Np}+ \frac{t^3}{(Np)^2}\biggr)\biggr].\label{e9.3}
\end{align}

Since
\[
t\in\biggl[-\log (r^{-2}n)\sqrt{Npq}- \frac{[r]_2^2(Np)^2}{2n^2},\log (r^{-2}n)\sqrt{Npq}\biggr],
\]
then we have
\[
O\biggl(\frac{t}{Npq}+ \frac{t^3}{(Npq)^2}\biggr)
=O\biggl(\frac{\log^3 (r^{-2}n)}{\sqrt{Np}}+
 \frac{r^6(Np)^2}{n^3}\biggr).\qedhere
\]
\end{proof}

By the proof of Lemma~\ref{l9.4}, we also have the following lemma.
\begin{lemma}\label{l9.5}
Let $r=r(n)\geq 3$ be an integer with $r=o(n^{ \frac12})$ and
$X\sim \Bin(N,p)$, where $r^{-2}n\leq Np=o(r^{-3}n^{ \frac32})$.
Then
\[
\mathbb{P}\bigl[X=\lfloor Np\rfloor+t\bigr]= \frac{1}{\sqrt{2\pi Np}}
 \,O\biggl(\exp\biggl(- \frac{t^2}{2Np}\biggr)\biggr),
\]
where $t\in \mathbb{Z}$ and
\[
t\in\biggl[- Np+\log (r^{-2}n),- \frac{[r]_2^2(Np)^2}{2n^2}-\log (r^{-2}n)\sqrt{Npq}\biggr].
\]
\end{lemma}

\begin{proof} By the same proof as Lemma~\ref{l9.4}, we find that~\eqref{e9.3}
still holds, which implies the lemma.
\end{proof}

We  prove Theorem~\ref{t1.3} separately for the two cases
$r^{-2}n\leq Np=o(r^{-3}n^{ \frac32})$ and $0<Np=O(r^{-2}n)$.

\begin{theorem}\label{t9.6}
Assume $Np=m_0$ with $r^{-2}n\leq m_0=o(r^{-3}n^{ \frac32})$.
Then
\[
\mathbb{P}\bigl[H_r(n,p)\in \mathcal{L}_r(n)\bigr]
=\exp\biggl[- \frac{[r]_2^2m_0^2}{4n^2}+
 \frac{[r]_2^3(3r-5)m_0^3}{6n^4}+
 O\biggl( \frac{\log^3 (r^{-2}n)}{\sqrt{m_0}}+ \frac{r^6m_0^2}{n^3}\biggr)\biggr].
\]
\end{theorem}

\begin{proof} Let
\[
m_0^*=m_0- \frac{[r]_2^2m_0^2}{2n^2}.
\]
Note that $m_0^*=m_0(1+o(1))$, enabling us to use
$m_0$ in place of $m_0^*$ in error terms.
We will divide the sum~\eqref{e9.2} into four domains:
\[
\mathbb{P}[H_r(n,p)\in \mathcal{L}_r]
=\sum_{m\in I_0\cup I_1\cup I_2\cup I_3}\mathbb{P}_r(n,m) \binom{N}{m}p^mq^{N-m},
\]
where
\begin{align*}
I_0&=\bigl[0,\log (r^{-2}n)\bigr], \\
I_1&=\bigl[\log (r^{-2}n), m_0^*-\log (r^{-2}n)\sqrt{m_0q}\bigr], \\
I_2&=\bigl[m_0^*-\log (r^{-2}n)\sqrt{m_0q},m_0+\log (r^{-2}n)\sqrt{m_0q}\bigr], \\
I_3&=\bigl[m_0+\log (r^{-2}n)\sqrt{m_0q},N\bigr].
\end{align*}

The theorem follows from a sequence of claims which we show next.

\bigskip
{\bf Claim 1}.~~$\displaystyle\mathbb{P}_r(n,m_0^*)=
\exp\Bigl[- \frac{[r]_2^2m_0^2}{4n^2}+ \frac{[r]_2^3(3r^2+9r-20)m_0^3}{24n^4}+O\Bigl( \frac{r^6m_0^2}{n^3}\Bigr)\Bigr].$

\begin{proof}[Proof of Claim 1] From Theorem~\ref{t1.1} we have
 \[
 \mathbb{P}_r(n,m_0^*)
  =\exp\biggl[- \frac{[r]_2^2\bigl[m_0- \frac{[r]_2^2m_0^2}{2n^2}\bigr]_2}{4n^2}- \frac{[r]_2^3(3r^2-15r+20)\bigl(m_0- \frac{[r]_2^2m_0^2}{2n^2}\bigr)^3}{24n^4}+
 O\Bigl( \frac{r^6{m_0^*}^2}{n^3}\Bigr)\biggr],
 \]
which simplifies to give the claim.
\end{proof}

\bigskip
{\bf Claim 2}.~~If
$m=m_0^*+s \in I_1\cup I_2$, then
\[
\mathbb{P}_r(n,m)=\mathbb{P}_r(n,m_0^*)\,
\exp\biggl[- \frac{[r]_2^2m_0^*s}{2n^2}- \frac{[r]_2^2s^2}{4n^2}
  +O\Bigl( \frac{r^6{m_0^2}}{n^3}\Bigr)\biggr].
\]

\begin{proof}[Proof of Claim 2]
 Since $m\in I_1\cup I_2$, we have
\[
s\in\Bigl[-m_0+ \frac{[r]_2^2m_0^2}{2n^2}+\log (r^{-2}n),\log (r^{-2}n)\sqrt{m_0q}+ \frac{[r]_2^2m_0^2}{2n^2}\Bigr].
\]
By Theorem~\ref{t1.1},  we have
\begin{align*}
\mathbb{P}_r(n,m)
&=\mathbb{P}_r(n,m_0^*+s)\\
&=\exp\biggl[- \frac{[r]_2^2[m_0^*+s]_2}{4n^2}
 - \frac{[r]_2^3(3r^2-15r+20)(m_0^*+s)^3}{24n^4}
 +O\Bigl( \frac{r^6(m_0^*+s)^2}{n^3}\Bigr)\biggr],
\end{align*}
which gives the claim because
$O\bigl( \frac{r^4(m_0^*+s)}{n^2}\bigr)=O\bigl( \frac{r^6{m_0^2}}{n^3}\bigr)$,
$O\bigl( \frac{r^8{m_0^*}^2s}{n^4}\bigr)=O\bigl( \frac{r^6{m_0^2}}{n^3}\bigr)$ and
$O\bigl( \frac{r^8{m_0^*}s^2}{n^4}\bigr)=O\bigl( \frac{r^6{m_0^2}}{n^3}\bigr)$.
\end{proof}

\bigskip
%In order to  show the key part
%\[
%\sum_{m\in I_2}\mathbb{P}_r(n,m) \binom{N}{m}p^mq^{N-m}.
%\]in Claim 4.
%Firstly, we prove the following claim.
%
%\vskip 0.3cm
{\bf Claim 3}.~~$\displaystyle
\sum_{s=-\log (r^{-2}n)\sqrt{m_0q}}^{\log (r^{-2}n)
\sqrt{m_0q}+ \frac{[r]_2^2m_0^2}{2n^2}}\exp\biggl[- \frac{s^2}{2m_0}\biggr]=\sqrt{2\pi m_0}
\biggl(1+O\Bigl( \frac{r^6m_0^2}{n^3}\Bigr)\biggr)$.

\begin{proof}[Proof of Claim 3]
This is an elementary summation that is easily proved either using the
Euler-Maclaurin summation formula or the Poisson summation formula.
\end{proof}

%Define $v_0=\log (r^{-2}n)\sqrt{m_0q}$ and $v_1= \frac{[r]_2^2m_0^2}{2n^2}$.
%Then, we have $s\in [-v_0,v_0+v_1]$ and  $|s|<m_0$ as $r^{-2}n\leq m_0=o(r^{-3}n^{ \frac32})$.
%Apply Lemma~\ref{l9.2} by taking $f(s)=\exp[- \frac{s^2}{2m_0}]$, then we have
%\[
%\begin{aligned}[b]
%\sum_{s=-v_0}^{v_0+v_1}\exp\Bigl[- \frac{s^2}{2m_0}\Bigr]
%&=\int_{-v_0}^{v_0+v_1}\exp\Bigl[- \frac{s^2}{2m_0}\Bigr]ds+
%O\biggl( \frac{\log (r^{-2}n)\sqrt{m_0q}+ \frac{[r]_2^2m_0^2}{2n^2}}{(r^{-2}n)^{\log (r^{-2}n)/2}}\biggr)\\
%\end{aligned}
%\]
%because $ \frac12\bigl(f(-v_0)+f(v_0+v_1)\bigr)=O\bigl((r^{-2}n)^{-\log (r^{-2}n)/2}\bigr)$, $ \frac{1}{12}\bigl(f'(-v_0)-f'(v_0+v_1)\bigr)=O\bigl((r^{-2}n)^{-\log (r^{-2}n)/2}\bigr)$
%and $f^{(4)}(s)=o\bigl(f(s)\bigr)$. %as $|s|<m_0$.
%
%
%Note that
%\[
%\begin{aligned}[b]
%\int_{-v_0}^{v_0+v_1}\exp\Bigl[- \frac{s^2}{2m_0}\Bigr]ds&=\sqrt{m_0}\int_{-\log (r^{-2}n)+o(1)}^{\log (r^{-2}n)+ \frac{[r]_2^2m_0^{ \frac32}}{2n^2}+o(1)}\exp\Bigl[- \frac{s^2}{2}\Bigr]ds\\
%&=\sqrt{2\pi m_0}
%\Bigl(1+O\Bigl((r^{-2}n)^{-\log(r^{-2}n)/2}\Bigr)\Bigr),
%\end{aligned}
%\]
%then we have
%\[
%\sum_{s=-\log (r^{-2}n)\sqrt{m_0q}}^{\log (r^{-2}n)\sqrt{m_0q}+
% \frac{[r]_2^2m_0^2}{2n^2}}\exp\biggl[- \frac{s^2}{2m_0}\biggr]=\sqrt{2\pi m_0}
%\biggl(1+O\biggl( \frac{\log (r^{-2}n)+ \frac{[r]_2^2m_0^{ \frac32}}{2n^2}}{(r^{-2}n)^{\log (r^{-2}n)/2}}\biggr)\biggr),
%\]where $O\biggl( \frac{\log (r^{-2}n)+ \frac{[r]_2^2m_0^{ \frac32}}{2n^2}}{(r^{-2}n)^{\log (r^{-2}n)/2}}\biggr)=O\bigl( \frac{r^6m_0^2}{n^3}\bigr)$.
%\end{proof}

\bigskip
{\bf Claim 4}.~~$\displaystyle
\sum_{m\in I_2} \mathbb{P}_r(n,m) \binom{N}{m}p^mq^{N-m} \newline
{\kern10em}=\exp\biggl[- \frac{[r]_2^2m_0^2}{4n^2}+ \frac{[r]_2^3(3r-5)m_0^3}{6n^4}+O\Bigl(\frac{\log^3 (r^{-2}n)}{\sqrt{m_0}}+
 \frac{r^6m_0^2}{n^3}\Bigr)\biggr]$.
%
%\begin{align*}
%\sum_{m\in I_2} & \mathbb{P}_r(n,m) \binom{N}{m}p^mq^{N-m}\\
%&=\exp\biggl[- \frac{[r]_2^2m_0^2}{4n^2}+ \frac{[r]_2^3(3r-5)m_0^3}{6n^4}+O\Bigl(\frac{\log^3 (r^{-2}n)}{\sqrt{m_0}}+
% \frac{r^6m_0^2}{n^3}\Bigr)\biggr].
%\end{align*}

\begin{proof}[Proof of Claim 4]
For $m=m_0- \frac{[r]_2^2m_0^2}{2n^2}+s\in I_2$, we have
\[
s\in\Bigl[-\log (r^{-2}n)\sqrt{m_0q},\log (r^{-2}n)\sqrt{m_0q}+ \frac{[r]_2^2m_0^2}{2n^2}\Bigr].
\]
By Lemma~\ref{l9.4}, we have
\[
 \binom{N}{m}p^mq^{N-m}= \frac{1}{\sqrt{2\pi m_0}}
\exp\biggl[- \frac{{\bigl(s- \frac{[r]_2^2m_0^2}{2n^2}\bigr)}^2}{2m_0}+O\Bigl( \frac{\log^3 (r^{-2}n)}{\sqrt{m_0}}+ \frac{r^6m_0^2}{n^3}\Bigr)\biggr].
\]
By Claim 2, we further have
\begin{align*}
\sum_{m\in I_2} & \mathbb{P}_r(n,m) \binom{N}{m}p^mq^{N-m} \\
&= \frac{\mathbb{P}_r(n,m_0^*)}{\sqrt{2\pi m_0}}\sum_{s=-\log (r^{-2}n)\sqrt{m_0q}}^{\log (r^{-2}n)\sqrt{m_0q}+ \frac{[r]_2^2m_0^2}{2n^2}}\exp\biggl[- \frac{{\bigl(s- \frac{[r]_2^2m_0^2}{2n^2}\bigr)}^2}{2m_0q}
- \frac{[r]_2^2s^2}{4n^2}- \frac{[r]_2^2\bigl(m_0- \frac{[r]_2^2m_0^2}{2n^2}\bigr)s}{2n^2} \\
&{\qquad}+O\Bigl( \frac{\log^3 (r^{-2}n)}{\sqrt{m_0}}+
 \frac{r^6m_0^2}{n^3}\Bigr)\biggr] \\
&= \frac{\mathbb{P}_r(n,m_0^*)}{\sqrt{2\pi m_0}}\exp\biggl[- \frac{[r]_2^4m_0^3}{8n^4}+O\Bigl( \frac{\log^3 (r^{-2}n)}{\sqrt{m_0}}+
 \frac{r^6m_0^2}{n^3}\Bigr)\biggr]\sum_{s=-\log (r^{-2}n)\sqrt{m_0q}}^{\log (r^{-2}n)\sqrt{m_0q}+ \frac{[r]_2^2m_0^2}{2n^2}}\!\!\!\exp\biggl[- \frac{s^2}{2m_0}\biggr],
\end{align*}
because $O\bigl( \frac{r^8{m_0^*}s^2}{n^4}\bigr)=O\bigl( \frac{r^6{m_0^2}}{n^3}\bigr)$ and $ \frac{[r]_2^2s^2}{4n^2}=O\bigl(\frac{\log^3 (r^{-2}n)}{\sqrt{m_0}}+
 \frac{r^6m_0^2}{n^3}\bigr)$.
Now we apply the value of $\mathbb{P}_r(n,m_0^*)$ from Claim~1 and the
summation from Claim~3.
%By Claim 1, we have
%\begin{align*}
%&\mathbb{P}_r(n,m_0^*)\exp\biggl[- \frac{[r]_2^4m_0^3}{8n^4}+
%O\Bigl( \frac{\log^3 (r^{-2}n)}{\sqrt{m_0}}+ \frac{r^6m_0^2}{n^3}\Bigr)\biggr]\\
%&{\qquad}=\exp\biggl[- \frac{[r]_2^2m_0^2}{4n^2}+ \frac{[r]_2^3(3r-5)m_0^3}{6n^4}
%+O\Bigl(\frac{\log^3 (r^{-2}n)}{\sqrt{m_0}}+ \frac{r^6m_0^2}{n^3}\Bigr)\biggr],
%\end{align*}
%together with Claim 3 into the equation~\eqref{e9.7} to complete the proof of Claim 4.
\end{proof}

Secondly, we show the value of
\[
\sum_{m\in I_1}\mathbb{P}_r(n,m) \binom{N}{m}p^mq^{N-m}
\]
in the following two claims.
\vskip 0.3cm
{\bf Claim 5}.~~$\displaystyle
 \frac{1}{\sqrt{m_0}}\!\!\sum_{s=-m_0+ \frac{[r]_2^2m_0^2}{2n^2}
 +\log (r^{-2}n)}^{-\log (r^{-2}n)\sqrt{m_0q}}\!\!
\exp\biggl[- \frac{[r]_2^2s^2}{4n^2}- \frac{s^2}{2m_0}
 + \frac{[r]_2^4m_0^2s}{4n^4}\biggr]=O\Bigl( \frac{r^6m_0^2}{n^3}\Bigr)$.

\begin{proof}[Proof of Claim 5]
Since the summand is increasing in the range of summation, it suffices to
take the number of terms times the last term.
%Note that $m_0=o(r^{-3}n^{ \frac32})$, then we have $\exp\bigl[- \frac{[r]_2^2s^2}{4n^2}- \frac{s^2}{2m_0}+ \frac{[r]_2^4m_0^2s}{4n^4}\bigr]
%=O\bigl(\exp\bigl[- \frac{s^2}{2m_0}+ \frac{[r]_2^4m_0^2s}{4n^4}\bigr]\bigr)$ and %. It is decreasing as
%%\[
%%s\in\Bigl[-m_0+ \frac{[r]_2^2m_0^2}{2n^2}+\log (r^{-2}n),-\log (r^{-2}n)\sqrt{m_0q}\Bigr],
%%\]then we have
%\[
%\begin{aligned}[b]
%& \frac{1}{\sqrt{m_0}}\sum_{s=-m_0+ \frac{[r]_2^2m_0^2}{2n^2}+\log (r^{-2}n)}^{-\log (r^{-2}n)
%\sqrt{m_0q}}\exp\biggl[- \frac{s^2}{2m_0}+ \frac{[r]_2^4m_0^2s}{4n^4}\biggr]\\
%&=O\biggl(\int_{-m_0+ \frac{[r]_2^2m_0^2}{2n^2}+\log (r^{-2}n)}^{-\log (r^{-2}n)\sqrt{m_0q}}
%\exp\biggl[- \frac{s^2}{2m_0}+ \frac{[r]_2^4m_0^2s}{4n^4}\biggr]d \frac{s}{\sqrt{ m_0}}\biggr)\\
%&=O\biggl(\int_{-m_0+ \frac{[r]_2^2m_0^2}{2n^2}+\log (r^{-2}n)}^{-\log (r^{-2}n)\sqrt{m_0q}}\exp\biggl[- \frac12\Bigl( \frac{s}{\sqrt{m_0}}+ \frac{[r]_2^4m_0^{ \frac{5}{2}}}{4n^4}\Bigr)^2\biggr]d \frac{s}{\sqrt{ m_0}}\biggr)\\
%&=O\biggl(\int_{-\sqrt{m_0}+ \frac{[r]_2^2m_0^{ \frac32}}{2n^2}+o(1)}^{-\log (r^{-2}n)+o(1)}\exp\biggl[- \frac12s^2\biggr]ds\biggr)\\
%&=O\biggl((r^{-2}n)^{-\log(r^{-2}n)/2}\biggr)\\
%&=O\Bigl( \frac{r^6m_0^2}{n^3}\Bigr)
%\end{aligned}
%\]to complete the proof of Claim 5.
\end{proof}

\bigskip
{\bf Claim 6}.~~$\displaystyle
\sum_{m\in I_1}\mathbb{P}_r(n,m) \binom{N}{m}p^mq^{N-m}=
O\Bigl( \frac{r^6m_0^2}{n^3}\Bigr)\sum_{m\in I_2}\mathbb{P}_r(n,m) \binom{N}{m}p^mq^{N-m}$.

%\ Let \[
%m=m_0- \frac{[r]_2^2m_0^2}{2n^2}+s\in I_1.
%\]
%Then
%\[
%\begin{aligned}[b]
%\sum_{m\in I_1}\mathbb{P}_r(n,m) \binom{N}{m}p^mq^{N-m}=
%O\Bigl( \frac{r^6m_0^2}{n^3}\Bigr)\sum_{m\in I_2}\mathbb{P}_r(n,m) \binom{N}{m}p^mq^{N-m}.
%\end{aligned}
%\]

\begin{proof}[Proof of Claim 6]
If $m=m_0^*+s\in I_1$, then we have
\[
s\in\Bigl[-m_0+ \frac{[r]_2^2m_0^2}{2n^2}+\log (r^{-2}n),-\log (r^{-2}n)\sqrt{m_0q}\Bigr].
\]
By Lemma~\ref{l9.5} and Claim 2, we have
\begin{align*}
&\sum_{m\in I_1}\mathbb{P}_r(n,m) \binom{N}{m}p^mq^{N-m}
   = \frac{\mathbb{P}_r(n,m_0^*)}{\sqrt{2\pi m_0}} \\
&{\qquad}\times\sum_{s=-m_0+ \frac{[r]_2^2m_0^2}{2n^2}+\log (r^{-2}n)}^{-\log (r^{-2}n)\sqrt{m_0q}}
\exp\biggl[- \frac{[r]_2^2m_0^*s}{2n^2}- \frac{[r]_2^2s^2}{4n^2}+O\Bigl( \frac{r^6m^2}{n^3}\Bigr)\biggr]
O\biggl(\exp\biggl[- \frac{\bigl(s- \frac{[r]_2^2m_0^2}{2n^2}\bigr)^2}{2m_0}\biggr]\biggr)
 \displaybreak[1]\\
&= \frac{\mathbb{P}_r(n,m_0^*)}{\sqrt{2\pi m_0}}
\kern-0.4em\sum_{s=-m_0+ \frac{[r]_2^2m_0^2}{2n^2}
  +\log (r^{-2}n)}^{-\log (r^{-2}n)\sqrt{m_0q}}\kern-0.4em
O\biggl(\exp\biggl[- \frac{[r]_2^2s^2}{4n^2}- \frac{s^2}{2m_0}+ \frac{[r]_2^4m_0^2s}{4n^4}- \frac{[r]_2^4m_0^3}{8n^4}+O\Bigl( \frac{r^6m^2}{n^3}\Bigr)\biggr]\biggr).
\end{align*}

By Claim 1 and Claim 4, we further have
\[
\mathbb{P}_r(n,m_0^*)\,O\biggl(\exp\Bigl[- \frac{[r]_2^4m_0^3}{8n^4}+O\Bigl( \frac{r^6m^2}{n^3}\Bigr)\Bigr]\biggr)=
O\biggl(\sum_{m\in I_2}\mathbb{P}_r(n,m) \binom{N}{m}p^mq^{N-m}\biggr),
\]
which completes the proof together with Claim~5.
\end{proof}

%At last, we consider
%\[
%\sum_{m\in I_0}\mathbb{P}_r(n,m) \binom{N}{m}p^mq^{N-m}\quad {\rm{and}}\quad
%\sum_{m\in I_3}\mathbb{P}_r(n,m) \binom{N}{m}p^mq^{N-m}
%\]
%in the following two claims.
%
%\vskip 0.3cm

\bigskip
{\bf Claim 7}.~~$\displaystyle
\sum_{m\in I_0}\mathbb{P}_r(n,m) \binom{N}{m}p^mq^{N-m}=O\Bigl( \frac{r^6m_0^2}{n^3}\Bigr)\sum_{m\in I_2}\mathbb{P}_r(n,m) \binom{N}{m}p^mq^{N-m}$.

%$m=m_0+t\in I_0.
%\] Then
%\[
%\begin{aligned}[b]
%\sum_{m\in I_0}\mathbb{P}_r(n,m) \binom{N}{m}p^mq^{N-m}=O\Bigl( \frac{r^6m_0^2}{n^3}\Bigr)\sum_{m\in I_2}\mathbb{P}_r(n,m) \binom{N}{m}p^mq^{N-m}.
%\end{aligned}
%\]

\begin{proof}[Proof of Claim 7]
Since $m=m_0+t\in I_0$, then we have
\[
t\in\Bigl[-m_0,-m_0+\log (r^{-2}n)\Bigr].
\]
Since $\mathbb{P}_r(n,m)\leq 1$ for $m\in I_0$,
Lemma~\ref{l9.3} gives
\[
\sum_{m\in I_0}\mathbb{P}_r(n,m) \binom{N}{m}p^mq^{N-m}
=O\bigl(\exp\bigl[- \dfrac13 m_0\bigr]\bigr).
\]
Together with Claim~4, this proves the required bound.
%By Claim 4 and  $r^{-2}n\leq m_0=o(r^{-3}n^{ \frac32})$,  we also have
%\[
%\begin{aligned}[b]
%\sum_{m\in I_0}\mathbb{P}_r(n,m) \binom{N}{m}p^mq^{N-m}=
%O\Bigl( \frac{r^6m_0^2}{n^3}\Bigr)\sum_{m\in I_2}\mathbb{P}_r(n,m) \binom{N}{m}p^mq^{N-m}. \qedhere
%\end{aligned}
%\]
\end{proof}

\bigskip
{\bf Claim 8}.~~$\displaystyle
\sum_{m\in I_3}\mathbb{P}_r(n,m) \binom{N}{m}p^mq^{N-m}=
O\Bigl( \frac{r^6m_0^2}{n^3}\Bigr)\sum_{m\in I_2}\mathbb{P}_r(n,m) \binom{N}{m}p^mq^{N-m}$.

%
%\ Let
%\[
%m=m_0+t\in I_3.
%\]Then
%\[
%\begin{aligned}[b]
%\sum_{m\in I_3}\mathbb{P}_r(n,m) \binom{N}{m}p^mq^{N-m}=
%O\Bigl( \frac{r^6m_0^2}{n^3}\Bigr)\sum_{m\in I_2}\mathbb{P}_r(n,m) \binom{N}{m}p^mq^{N-m}.
%\end{aligned}
%\]

\begin{proof}[Proof of Claim 8]
Since $m=m_0+t\in I_3$,  then we have
\[
 t\in\Bigl[\log (r^{-2}n)\sqrt{m_0q},N-m_0\Bigr].
\]
By Lemma~\ref{l9.1}, we have $\mathbb{P}_r(n,m)\leq \mathbb{P}_r(n,m_0)$
for $m_0\in I_3$.  Therefore
\[
  \sum_{m\in I_3}\mathbb{P}_r(n,m) \binom{N}{m}p^mq^{N-m}
  \leq \mathbb{P}_r(n, m_0) \sum_{m\in I_3}
     \binom{N}{m}p^mq^{N-m},
\]
from which the claim follows using Theorem~\ref{t1.1} and
Lemma~\ref{l9.3}.
%
%By Lemma~\ref{l9.3},  we also have
%\begin{align*}
%&\mathbb{P}\biggl[ \binom{N}{m}p^mq^{N-m}-m_0\geq\log (r^{-2}n)\sqrt{m_0q}\biggr]\\
%&<\exp\biggl[- \frac{\bigl(\log(r^{-2}n)\sqrt{m_0q}\bigr)^2}{3m_0}\biggr]\\
%&<\exp\biggl[- \frac{1}{3}\log^2(r^{-2}n)q\biggr],\\
%\end{align*}
%which implies that
%\[
%\sum_{m\in I_3} \binom{N}{m}p^mq^{N-m}<\exp\biggl[- \frac{1}{3}\log^2(r^{-2}n)q\biggr].
%\]
%By Lemma~\ref{l9.1}, we have $\mathbb{P}_r(n,m)<\mathbb{P}_r(n,m_0)$ for $m\in I_3$.
%From Theorem~\ref{t1.1}, we  have
%\begin{align*}
%&\sum_{m\in I_3}\mathbb{P}_r(n,m) \binom{N}{m}p^mq^{N-m}\\
%&<\exp\biggl[- \frac{[r]_2^2m_0^2}{4n^2}- \frac{[r]_2^3(3r^2-15r+20)m_0^3}{24n^4}+
%O\Bigl( \frac{r^6m_0^2}{n^3}\Bigr)\biggr]\exp\biggl[- \frac{1}{3}\log^2(r^{-2}n)q\biggr].
%\end{align*}
%By Claim 4, we finally have
%\[
%\begin{aligned}[b]
%\sum_{m\in I_3}\mathbb{P}_r(n,m) \binom{N}{m}p^mq^{N-m}
%%<&\exp\biggl[- \frac{[r]_2^2m_0^2}{4n^2}- \frac{[r]_2^3(3r^2-15r+20)m_0^3}{24n^4}+
%%O\Bigl( \frac{r^6m_0^2}{n^3}\Bigr)\biggr]\exp\biggl[- \frac{1}{3}\log^2(r^{-2}n)q\biggr]\\
%<&\sum_{m\in I_2}\mathbb{P}_r(n,m) \binom{N}{m}p^mq^{N-m}\exp\biggl[- \frac{[r]_2^4m_0^3}{8n^4}- \frac{1}{3}\log^2(r^{-2}n)q\biggr]\\
%\end{aligned}
%\]
%to complete the proof of Claim 6, where $\exp\bigl[- \frac{[r]_2^4m_0^3}{8n^4}- \frac{1}{3}\log^2(r^{-2}n)q\bigr]
%=O\bigl( \frac{r^6m_0^2}{n^3}\bigr)$.
\end{proof}

\bigskip

%As the equation shown in~\eqref{e9.5}, by Claim 4, Claim 6, Claim 7 and Claim 8, we have
%\[
%\begin{aligned}[b]
%\mathbb{P}\bigl[H_r(n,p)\in \mathcal{L}_r\bigr]
%&=\sum_{m\in I_2}\mathbb{P}_r(n,m) \binom{N}{m}p^mq^{N-m}\Bigl(1+O\Bigl( \frac{r^6m_0^2}{n^3}\Bigr)\Bigr)\\
%&=\sum_{m\in I_2}\mathbb{P}_r(n,m) \binom{N}{m}p^mq^{N-m}\exp\Bigl[O\Bigl( \frac{r^6m_0^2}{n^3}\Bigr)\Bigr]\\
%&=\exp\biggl[- \frac{[r]_2^2m_0^2}{4n^2}+ \frac{[r]_2^3(3r-5)m_0^3}{6n^4}
%+O\Bigl( \frac{\log^3 (r^{-2}n)}{\sqrt{m_0}}+ \frac{r^6m_0^2}{n^3}\Bigr)\biggr]
%\end{aligned}
%\]
%to complete the proof of Theorem~\ref{t9.6}.
To complete the proof of Theorem~\ref{t9.6}, add together Claims 4, 6, 7 and 8.
\end{proof}

Note that in the process of proving Theorem~\ref{t9.6} we also proved
Corollary~\ref{c1.4}.
%From the calculations in the proof of Theorem~\ref{t9.6}, we have a corollary about the distribution on the
%number of edges of $H_r(n,p)$.
%
%\begin{corollary}\label{c9.7}
%Under the condition of Theorem~\ref{t9.6}, the number of edges of $H_r(n,p)$ conditioned on being
%linear converges in distribution to the normal distribution with mean $m_0- \frac{[r]_2^2m_0^2}{2n^2}$
%and variance $m_0$.
%\end{corollary}

The second case of Theorem~\ref{t1.3} is $0<m_0=O(r^{-2}n)$.

\begin{theorem}\label{t9.8}
Assume that $0<m_0=O(r^{-2}n)$. Then
\[
\mathbb{P}\bigl[H_r(n,p)\in \mathcal{L}_r\bigr]=\exp\biggl[- \frac{[r]_2^2m_0^2}{4n^2}+O\biggl( \frac{r^6m_0^2}{n^3}\biggr)\biggr].
\]
\end{theorem}

\begin{proof} Let $X_{i}$
denote the number of pairs of edges with $i$ common vertices in $H_r(n,p)$,
where $2\leq i\leq r-1$. Let $X_{\rm{link}}=\sum_{i=2}^{r-1}X_i$.

Let $M_{i,1},\ldots,M_{i,t_i}$ be all unordered pairs of $r$-sets $\{e_1,e_{2}\}$ in
$[n]$ with $|e_1\cap e_2|=i$, where $2\leq i\leq r-1$ and
\[
t_i= \dfrac12N \binom{r}{i} \binom{n-r}{r-i}.
\]
 Firstly,  we have
 \begin{align*}
\mathbb{E}\left[X_2\right]&=\sum_{j=1}^{t_2}\mathbb{P}\bigl[M_{2,j}
 \text{ is in }H_r(n,p)\bigr]\\
&= \dfrac12N \binom{r}{2} \binom{n-r}{r-2}\biggl( \frac{m_0}{N}\biggr)^2
= \frac{[r]_2^2m_0^2}{4n^2}\Bigl(1+\Bigl( \frac{r}{n}\Bigr)\Bigr).
 \end{align*}
We also have
\[
\mathbb{E}\biggl[\sum_{i\geq 3}X_i\biggr]
=O\biggl(\sum_{i\geq 3}N \binom{r}{i} \binom{n-r}{r-i}\biggl( \frac{m_0}{N}\biggr)^2\biggr)
=O\Bigl( \frac{r^6m_0^2}{n^3}\Bigr).
 \]
Thus,  we have
\[
\mathbb{E}\bigl[X_{\rm{link}}\bigr]=\sum_{i=2}^{r-1}
\sum_{j=1}^{t_i}\mathbb{P}\bigl[M_{i,j}\text{ is in }H_r(n,p)\bigr]
= \frac{[r]_2^2m_0^2}{4n^2}+O\Bigl( \frac{r^6m_0^2}{n^3}\Bigr).
 \]
 Similarly,
 \[
 \sum_{i_1,i_2=2}^{r-1}\sum_{j_1=1}^{t_{i_1}}\sum_{j_2=1}^{t_{i_2}}\mathbb{P}\bigl[M_{i_1,j_1},M_{i_2,j_2}\text{ are in }H_r(n,p)\bigr]=O\Bigl( \frac{r^8m_0^3}{n^4}\Bigr).
\]
By inclusion-exclusion, we conclude that
\[
\mathbb{P}\bigl[X_{{\rm link}}=0\bigr]= 1- \frac{[r]_2^2m_0^2}{4n^2}+O\Bigl( \frac{r^6m_0^2}{n^3}\Bigr),
\]
where $O\bigl( \frac{r^8m_0^3}{n^4}\bigr)=O\bigl( \frac{r^6m_0^2}{n^3}\bigr)$ because $0<m_0=O(r^{-2}n)$.
\end{proof}

Theorems~\ref{t9.6} and~\ref{t9.8} together complete the proof of Theorem~\ref{t1.3}.

\section{Proof of Theorem~\ref{t1.5}}\label{s:10}

%\red{[See the red comment in Theorem~\ref{t1.5}.]}

As in the theorem statement, we will assume  $m=o(r^{-3}n^{\frac32})$
%and $k=o\bigl(\min\{\frac{n^2}{r^5m}, \frac{n^3}{r^6m^2}\bigr\}\bigr)$.
and $k=o\bigl(\frac{n^3}{r^6m^2}\bigr)$.
The bound on $m$ implies that either $m=0$ (a trivial case we will ignore)
or $r=o(n^{\frac12})$.
We will also assume that $k\leq m$, since otherwise the theorem is trivially true
because $[m]_k=0$ if~$k>m$.

Let $K=K(n)$ be a given linear $r$-graph on $[n]$
vertices  with edges $\{e_1,\ldots,e_k\}$.
Consider
$H\in\mathcal{L}_r(n,m)$ chosen uniformly at random. Let $\mathbb{P}[K\subseteq H]$
be the probability that
$H$ contains $K$ as a subhypergraph.
If $\mathbb{P}[K\subseteq H]\neq 0$, then we have
\begin{align}
\mathbb{P}[K\subseteq H]&=\mathbb{P}[e_1,\ldots,e_k\in H] \notag\\
&=\prod_{i=1}^{k}\, \frac{\mathbb{P}[e_1,\ldots,e_i\in H]}{\mathbb{P}[e_1,\ldots,e_i\in H]+\mathbb{P}[e_1,\ldots,e_{i-1}\in H,e_i\notin H]} \notag\\
&=\prod_{i=1}^{k}\, \biggl(1+ \frac{\mathbb{P}[e_1,\ldots,e_{i-1}\in H,e_i\notin H]}
{\mathbb{P}[e_1,\ldots,e_i\in H]}\biggr)^{\!\!-1}.\label{e10.1}
\end{align}
For $i=1,\ldots,k$, let $\mathcal{L}_r(n,m:\overline{e}_i)$ be the set of all linear
hypergraphs in $\mathcal{L}_r(n,m)$ which contain edges
$e_1,\ldots,e_{i-1}$ but not edge $e_i$.
Let $\mathcal{L}_r(n,m:e_i)$ be the set of all linear
hypergraphs in $\mathcal{L}_r(n,m)$ which contain edges
$e_1,\ldots,e_{i}$. We have the ratio
\begin{equation}\label{e10.2}
 \frac{\mathbb{P}[e_1,\ldots,e_{i-1}\in H,e_i\notin H]}{\mathbb{P}[e_1,\ldots,e_i\in H]}
= \frac{|\mathcal{L}_r(n,m:\overline{e}_i)|}{|\mathcal{L}_r(n,m:e_i)|}.
\end{equation}
Note that  $|\mathcal{L}_r(n,m)|\neq 0$
by Theorem~\ref{t1.1}.
%If $H\in\mathcal{L}_r(n,m:\overline{e}_i)$ for some $1\leq i\leq k$,
%then we use \textit{adding $e_i$-switchings} defined below such that
%$|\mathcal{L}_r(n,m:e_j)|\neq 0$ for any $i\leq j\leq k$.
We will show below that none of the denominators in~\eqref{e10.2} are zero.

\medskip

Let $H\in\mathcal{L}_r(n,m:e_i)$ with $1\leq i\leq k$.
An \textit{$e_i$-displacement} is defined in two steps:

\noindent{\bf Step 0.}\ Remove the edge $e_i$ from $H$.
Define $H_0$ with the same vertex set $[n]$
and the edge set $E(H_0)=E(H)\setminus\{e_i\}$.

\noindent{\bf Step 1.}\ Take any $r$-set distinct from $e_i$ of which no
two vertices belong to the same edge
of $H_0$ and add it as an edge to~$H_0$. The new graph is denoted by $H'$.

\begin{lemma}\label{l10.1}
Assume $m=o(r^{-3}n^{ \frac32})$ and $1\leq i\leq k$.
Let $H\in\mathcal{L}_r(n,m-1)$ and let $N_r$ be
the set of $r$-sets distinct from $e_i$ of which no two vertices belong to
the same edge of $H$. Then
\[
|N_r|=\biggl[N- \binom{r}{2}m\binom{n-2}{r-2}\biggr]
%\biggl(1+O\Bigl( \frac{r^4}{n^2}+ \min\Bigl\{ \frac{r^6m^2}{n^3}, \frac{r^5m}{n^2}\Bigr\}\Bigr)\biggr)
\biggl(1+O\Bigl( \frac{r^4}{n^2}+\frac{r^6m^2}{n^3}\Bigr)\biggr)
\]
\end{lemma}

\begin{proof} It is clear that $|N_r|\geq\bigl[N- 1- \binom{r}{2}(m-1) \binom{n-2}{r-2}\bigr]$.
%Note that
%\[
%\sum_{v\in[n]} \binom{\deg (v)}{2}=O\Bigl(rm\min\Bigl\{m, \frac{n-1}{r-1}\Bigr\}\Bigr)
%=O\bigl(\min\{rm^2,mn\}\bigr).
%\]
The proof of Lemma~\ref{l3.5} applies if we  replace the bound on
$\sum_{v\in[n]} \binom{\deg (v)}{2}$ in~\eqref{e3.4} by
$O(rm^2)$ and note that $ \frac{ \binom{r}{2} \binom{n-2}{r-2}}{ \binom{n}{r}}=O ( \frac{r^4}{n^2})$.
\end{proof}

An \textit{$e_i$-replacement} is the inverse of an $e_i$-displacement.
An $e_i$-replacement from $H'\in\mathcal{L}_r(n,m:\overline{e}_i)$
consists of removing any edge in $E(H')-\{e_1,\ldots,e_{i-1}\}$,
then inserting $e_i$. We say that the $e_i$-replacement is \textit{legal} if
$H\in\mathcal{L}_r(n,m:e_i)$, otherwise it is \textit{illegal}.

\begin{lemma}\label{l10.2}
Assume $m=o(r^{-3}n^{ \frac32})$ and $1\leq i\leq k$.
Consider $H'\in\mathcal{L}_r(n,m:\overline{e}_i)$
chosen uniformly at random. Let $E^*$
be the set of  $r$-sets $e^*$ of $[n]$ such that $|e^*\cap e_i|\geq 2$.
Suppose that $n\to\infty$.
Then
\begin{equation*}
\mathbb{P}\bigl[E^*\cap H'\neq\emptyset\bigr]=
 \frac{(m-i+1) \binom{r}{2} \binom{n-r}{r-2}}{N}
  +O\Bigl( \frac{r^6m^2}{n^3}\Bigr).
\end{equation*}
\end{lemma}

\begin{proof}
Fix an $r$-set $e^*\in E^*$.
 Let $\mathcal{L}_r(n,m:\overline{e}_i, e^*)$ %\subseteq\mathcal{L}_r(n,m:\overline{e}_i)
be the set of all the hypergraphs in $\mathcal{L}_r(n,m:\overline{e}_i)$
which contain the edge $e^*$. Let
\begin{align*}
\mathcal{L}_r(n,m:\overline{e}_i, \overline{e}^*)=
\mathcal{L}_r(n,m:\overline{e}_i)-\mathcal{L}_r(n,m:\overline{e}_i, e^*).
\end{align*}
Thus, we have
\begin{align}
\mathbb{P}[e^*\in H']&
= \frac{|\mathcal{L}_r(n,m:\overline{e}_i, e^*)|}{|\mathcal{L}_r(n,m:\overline{e}_i)|}\notag\\
&= \frac{|\mathcal{L}_r(n,m:\overline{e}_i, e^*)|}
{|\mathcal{L}_r(n,m:\overline{e}_i, e^*)|+|\mathcal{L}_r(n,m:\overline{e}_i, \overline{e}^*)|}\notag\\
&= \biggl(1+ \frac{|\mathcal{L}_r(n,m:\overline{e}_i, \overline{e}^*)|}{|\mathcal{L}_r(n,m:\overline{e}_i, e^*)|}\biggr)^{\!-1}.\label{e10.3}
\end{align}

Let $G\in\mathcal{L}_r(n,m:\overline{e}_i, e^*)$ and $S(G)$ be the set
of all ways to move the edge $e^*$ to an $r$-set of $[n]$ distinct from $e^*$ and $e_i$,
of which no two vertices are in any remaining  edges of $G$.
Call the new graph~$G'$.
By the same proof as Lemma~\ref{l10.1}, we have
\[
  S(G)=\biggl[N- \binom{r}{2}m \binom{n-2}{r-2}\biggr]
  \biggl(1+O\Bigl( \frac{r^4}{n^2}+ \frac{r^6m^2}{n^3}\Bigr)\biggr).
\]

Conversely, let $G'\in\mathcal{L}_r(n,m:\overline{e}_i, \overline{e}^*)$
and let $S'(G')$ be the set of all ways to move one edge in $E(G')-\{e_1,\ldots,e_{i-1}\}$
to $e^*$ to make the resulting graph in
$\mathcal{L}_r(n,m:\overline{e}_i, e^*)$. In order to find the expected number of $S'(G')$,
we need to apply the same switching way to $\mathcal{L}_r(n,m:\overline{e}_i, \overline{e}^*)$
with a simple analysis.

Likewise, let $E^{**}$ be the set of $r$-sets $e^{**}$ of $[n]$ such that $|e^{**}\cap e^*|\geq 2$
and fix an $r$-set $e^{**}\in E^{**}$. %\subseteq\mathcal{L}_r(n,m:\overline{e}_i,\overline{e}^*)
 Let $\mathcal{L}_r(n,m:\overline{e}_i, \overline{e}^*, {e}^{**})$
be the set of all the hypergraphs in $\mathcal{L}_r(n,m:\overline{e}_i,\overline{e}^*)$
which contain the edge $e^{**}$ and $\mathcal{L}_r(n,m:\overline{e}_i, \overline{e}^*,\overline{e}^{**})=
\mathcal{L}_r(n,m:\overline{e}_i,\overline{e}^*)-\mathcal{L}_r(n,m:\overline{e}_i, \overline{e}^*,{e}^{**})$.
By the exactly same analysis above, we also have
\begin{align}\label{e10.4}
\mathbb{P}[e^{**}\in G']
&= \biggl(1+ \frac{|\mathcal{L}_r(n,m:\overline{e}_i, \overline{e}^*,\overline{e}^{**})|}{|\mathcal{L}_r(n,m:\overline{e}_i, \overline{e}^*,{e}^{**})|}\biggr)^{\!-1}.%\label{e10.4}
\end{align}
For any hypergraph in $\mathcal{L}_r(n,m:\overline{e}_i, \overline{e}^*,{e}^{**})$, %and $|E^{**}|=\sum_{i=2}^{r-1} \binom{r}{i} \binom{n-r}{r-i}$,
by the same proof as Lemma~\ref{l10.1}, we also have $\bigl[N- \binom{r}{2}m \binom{n-2}{r-2}\bigr]\bigl(1+O\bigl( \frac{r^4}{n^2}+ \frac{r^6m^2}{n^3}\bigr)\bigr)$
ways to move the edge $e^{**}$ to an $r$-set of $[n]$ distinct from ${e}_i$, ${e}^*$ and ${e}^{**}$,
of which no two vertices are in any remaining edges. Similarly, there are at most $m-i+1$ ways %to move one edge to $e^{**}$
to switch a hypergraph from $\mathcal{L}_r(n,m:\overline{e}_i, \overline{e}^*,\overline{e}^{**})$ to
$\mathcal{L}_r(n,m:\overline{e}_i, \overline{e}^*,{e}^{**})$.
As the equation shown in~\eqref{e10.4}, we have
$\mathbb{P}[e^{**}\in G']
= O\bigl( \frac{m}{N}\bigr)$.  Note that $|E^{**}|=\sum_{i=2}^{r-1} \binom{r}{i} \binom{n-r}{r-i}=O\bigl( \binom{r}{2} \binom{n-r}{r-2}\bigr)$,
then $\mathbb{P}[E^{**}\cap G'\neq\emptyset]=O\bigl( \frac{r^4m}{n^2}\bigr)$ and
the expected number of $S'(G')$ is $(m-i+1)\bigl(1-O\bigl( \frac{r^4m}{n^2}\bigr)\bigr)$.

Thus, we have
\begin{align*}
& \frac{|\mathcal{L}_r(n,m:\overline{e}_i, \overline{e}^*)|}{|\mathcal{L}_r(n,m:\overline{e}_i, e^*)|}
= \frac{|S(G)|}{|S'(G')|}\\
&{\qquad}= \frac{N- \binom{r}{2}m \binom{n-2}{r-2}}{m-i+1}
\biggl(1+O\Bigl( \frac{r^4m}{n^2}+ \frac{r^6m^2}{n^3}\Bigr)\biggr)\\
&{\qquad}= \frac{N}{m-i+1}\biggl(1+O\Bigl( \frac{r^4m}{n^2}+\frac{r^6m^2}{n^3}\Bigr)\biggr).
\end{align*}
As the equation shown in~\eqref{e10.3}, we also have
\begin{align}\label{e10.5}
\mathbb{P}[e^*\in H']= \frac{m-i+1}{N}
\biggl(1+O\Bigl( \frac{r^4m}{n^2}+\frac{r^6m^2}{n^3}\Bigr)\biggr).
\end{align}

By inclusion-exclusion,
\begin{align}\label{e10.6}
\sum_{e^*\in E^*}\mathbb{P}[e^*\in H']-\sum_{e_1^*\in E^*, e_2^*\in E^*,e_1^*\cap e_2^*\neq\emptyset}\mathbb{P}[e_1^*,e_2^*\in H']\leq \mathbb{P}\bigl[E^*\cap H'\neq\emptyset\bigr]\leq \sum_{e^*\in E^*}\mathbb{P}[e^*\in H'].
\end{align}
Since $|E^*|=\sum_{i=2}^{r-1} \binom{r}{i} \binom{n-r}{r-i}= \binom{r}{2} \binom{n-r}{r-2}+O\bigl( \frac{r^3n^{r-3}}{(r-3)!}\bigr)$,
as the equation shown in~\eqref{e10.5}, we have
\begin{align}\label{e10.7}
\sum_{e^*\in E^*}\mathbb{P}[e^*\in H']=
 \frac{(m-i+1) \binom{r}{2} \binom{n-2}{r-2}}{N}
 +O\Bigl(\frac{r^6m^2}{n^3}\Bigr)
\end{align}
because $O\bigl( \frac{m}{N} \binom{r}{2} \binom{n-2}{r-2}\bigl( \frac{r^4m}{n^2}
+\frac{r^6m^2}{n^3}\bigr)\bigr)=
O\bigl(\frac{r^6m^2}{n^3}\bigr)$ and
 $O\bigl( \frac{m}{N} \frac{r^3n^{r-3}}{(r-3)!}\bigr)=O\bigl( \frac{r^6m}{n^3}\bigr)$.

Consider
$\sum_{e_1^*\in E^*, e_2^*\in E^*, e_1^*\cap e_2^*\neq\emptyset}\mathbb{P}[e_1^*,e_2^*\in H']$ in the
equation~\eqref{e10.6}.
Note that $H'\in \mathcal{L}_r(n,m:\overline{e}_i)$, then $|e_1^*\cap e_2^*|=1$.
Let $\mathcal{L}_r(n,m:\overline{e}_i, e_1^*, e_2^*)$,
$\mathcal{L}_r(n,m:\overline{e}_i, {e}_1^*, \overline{e}_2^*)$,
$\mathcal{L}_r(n,m:\overline{e}_i, \overline{e}_1^*, e_2^*)$ and $\mathcal{L}_r(n,m:\overline{e}_i, \overline{e}_1^*, \overline{e}_2^*)$
be the set of all linear
hypergraphs in $\mathcal{L}_r(n,m:\overline{e}_i)$ which contain both
$e_1^*$ and  $e_2^*$, only contain $e_1^*$, only contain $e_2^*$ and neither of them,
respectively. Thus, we have
\begin{align}
&\mathbb{P}[e_1^*,e_2^*\in H']\notag\\&
= \frac{|\mathcal{L}_r(n,m:\overline{e}_i, e_1^*, e_2^*)|}{|\mathcal{L}_r(n,m:\overline{e}_i)|}\notag\\
&= \frac{|\mathcal{L}_r(n,m:\overline{e}_i, e_1^*, e_2^*)|}
{|\mathcal{L}_r(n,m:\overline{e}_i, e_1^*, e_2^*)|+|\mathcal{L}_r(n,m:\overline{e}_i, {e}_1^*, \overline{e}_2^*)|+|\mathcal{L}_r(n,m:\overline{e}_i, \overline{e}_1^*, e_2^*)|+|\mathcal{L}_r(n,m:\overline{e}_i, \overline{e}_1^*, \overline{e}_2^*)|}\notag\\
&= \biggl(1+ \frac{|\mathcal{L}_r(n,m:\overline{e}_i, {e}_1^*, \overline{e}_2^*)|}{|\mathcal{L}_r(n,m:\overline{e}_i, e_1^*, e_2^*)|}+ \frac{|\mathcal{L}_r(n,m:\overline{e}_i, \overline{e}_1^*, e_2^*)|}{|\mathcal{L}_r(n,m:\overline{e}_i, e_1^*, e_2^*)|}+ \frac{|\mathcal{L}_r(n,m:\overline{e}_i, \overline{e}_1^*, \overline{e}_2^*)|}{|\mathcal{L}_r(n,m:\overline{e}_i, e_1^*, e_2^*)|}\biggr)^{\!-1}\label{e10.8}.
\end{align}
By the similar analysis above, we have
\begin{align}
 \frac{|\mathcal{L}_r(n,m:\overline{e}_i, {e}_1^*, 
 \overline{e}_2^*)|}{|\mathcal{L}_r(n,m:\overline{e}_i, e_1^*, e_2^*)|}&\geq \frac{\bigl[N- \binom{r}{2}m \binom{n-2}{r-2}\bigr]}{m-i+1}\biggl(1+O\Bigl( \frac{r^4}{n^2}+\frac{r^6m^2}{n^3}\Bigr)\biggr), \notag\\
 \frac{|\mathcal{L}_r(n,m:\overline{e}_i, \overline{e}_1^*, e_2^*)|}{|\mathcal{L}_r(n,m:\overline{e}_i, e_1^*, e_2^*)|}
 &\geq \frac{\bigl[N- \binom{r}{2}m \binom{n-2}{r-2}\bigr]}{m-i+1}\biggl(1+O\Bigl( \frac{r^4}{n^2}+ \frac{r^6m^2}{n^3}\Bigr)\biggr). \label{e10.9}
\end{align}

For any hypergraph in $\mathcal{L}_r(n,m:\overline{e}_i, e_1^*, e_2^*)$, we move
$e_1^*$ and $e_2^*$ away in two steps by
the similar switching operations in Section~\ref{s:5}.
For $e_1^*$ (resp. $e_2^*$), by the same proof as Lemma~\ref{l10.1},
there are  $\bigl[N- \binom{r}{2}m \binom{n-2}{r-2}\bigr]
\bigl(1+O\bigl( \frac{r^4}{n^2}+ \frac{r^6m^2}{n^3}\bigr)\bigr)$ ways to move  $e_1^{*}$
 (resp. $e_2^*$)  to an $r$-set of $[n]$ distinct from ${e}_i$, $e_1^*$ and $e_2^*$
 such that the resulting graph is in $\mathcal{L}_r(n,m:\overline{e}_i, \overline{e}_1^*, \overline{e}_2^*)$.
%of which no two vertices are in any remaining edges.
Similarly, there are at most $2\binom{m-i+1}{2}$ ways %to move one edge to $e^{**}$
to switch a hypergraph from $\mathcal{L}_r(n,m:\overline{e}_i, \overline{e}_1^*, \overline{e}_2^*)$ to
$\mathcal{L}_r(n,m:\overline{e}_i, e_1^*, e_2^*)$.
Thus, we have
\begin{align}\label{e10.10}
 \frac{|\mathcal{L}_r(n,m:\overline{e}_i, \overline{e}_1^*, \overline{e}_2^*)|}{|\mathcal{L}_r(n,m:\overline{e}_i, e_1^*, e_2^*)|}&\geq \frac{\bigl[N- \binom{r}{2}m \binom{n-2}{r-2}\bigr]^2}{2\binom{m-i+1}{2}}\biggl(1+O\Bigl( \frac{r^4}{n^2}+ \frac{r^6m^2}{n^3}\Bigr)\biggr).
\end{align}

By equations~\eqref{e10.8}--\eqref{e10.10} and note that there are
 $O\bigl( \frac{r^3 n^{2r-4}}{(r-2)!^2}\bigr)$
ways to choose the pair $\{e_1^*,e_2^*\}$ such that $|e_1^*\cap e_i|\geq 2$, $|e_2^*\cap e_i|\geq 2$ and $|e_1^*\cap e_2^*|=1$,
then we have
\begin{align}\label{e10.11}
\sum_{e_1^*\in E^*, e_2^*\in E^*,e_1^*\cap e_2^*\neq\emptyset}\mathbb{P}[e_1^*,e_2^*\in H']
=O\biggl( \frac{r^3 n^{2r-4}}{(r-2)!^2} \frac{m^2}{N^2}\biggr)=O\Bigl( \frac{r^7 m^2}{n^4}\Bigr)=O\Bigl( \frac{r^6m^2}{n^3}\Bigr).
\end{align}

To complete the proof of Lemma~\ref{l10.2}, add together the equations~\eqref{e10.6},~\eqref{e10.7} and~\eqref{e10.11}.
\end{proof}

By Lemmas~\ref{l10.1} and~\ref{l10.2}, we have
\begin{lemma}\label{l10.3}
Assume $m=o(r^{-3}n^{ \frac32})$ and $1\leq i\leq k$. Then\\
$(a)$\ Let $H\in\mathcal{L}_r(n,m:{e}_i)$. The number of $e_i$-displacements is
\[
\biggl[N- \binom{r}{2}m \binom{n-2}{r-2}\biggr]\biggl(1+O\Bigl( \frac{r^4}{n^2}
  + \frac{r^6m^2}{n^3}\Bigr)\biggr).
\]
$(b)$\ Consider $H'\in\mathcal{L}_r(n,m:\overline{e}_i)$ chosen uniformly at random.
The expected number of legal $e_i$-replacements is
\[
(m-i+1)\biggl[1- \frac{(m-i+1) \binom{r}{2} \binom{n-r}{r-2}}{N}
 +O\Bigl(\frac{r^6m^2}{n^3}\Bigr)\biggr].
\]
$(c)$
\begin{align*}
 \frac{|\mathcal{L}_r(n,m:\overline{e}_i)|}{|\mathcal{L}_r(n,m:{e}_i)|}
&= \frac{\bigl[N- \binom{r}{2}m \binom{n-2}{r-2}\bigr]}{m-i+1}\\
&{\qquad}\times\biggl(1+ \frac{(m-i+1) \binom{r}{2} \binom{n-r}{r-2}}{N}
+O\Bigl( \frac{r^4}{n^2}+\frac{r^6m^2}{n^3}\Bigr)\biggr).
\end{align*}
\end{lemma}

By Lemma~\ref{l10.3}(c), we have
\begin{align*}
\mathbb{P}[K&\subseteq H]\\
&=\prod_{i=1}^{k} \biggl(1+
\frac{\mathbb{P}[e_1,\ldots,e_{i-1}\in H,e_i\notin H]}{\mathbb{P}[e_1,\ldots,e_i\in H]}\biggr)^{\!\!-1}\\
&=\prod_{i=1}^{k} \frac{m-i+1}{\bigl[N- \binom{r}{2}m \binom{n-2}{r-2}\bigr]}
\biggl(1- \frac{(m-i+1) \binom{r}{2} \binom{n-r}{r-2}}{N}
+O\Bigl( \frac{r^4}{n^2}+ \frac{r^6m^2}{n^3}\Bigr)\biggr)\\
&=\prod_{i=1}^{k} \frac{m-i+1}{\bigl[N- \binom{r}{2}m \binom{n-2}{r-2}\bigr]}
\exp\biggl[- \frac{(m-i+1) \binom{r}{2} \binom{n-r}{r-2}}{N}
+O\Bigl( \frac{r^4}{n^2}+ \frac{r^6m^2}{n^3}\Bigr)\biggr]\\
&= \frac{[m]_k}{N^k}\exp\biggl[ \frac{[r]_2^2k^2}{4n^2}
+O\Bigl( \frac{r^4k}{n^2}+ \frac{r^6m^2k}{n^3}\Bigr)\biggr],
\end{align*}
since $k=o\bigl(\frac{n^3}{r^6m^2}\bigr)$.

%\begin{remark}\label{r10.4}
%Consider $H\in \mathcal{H}_r(n,m)$ chosen uniformly at random. Let
%$\mathbb{P}_{\mathcal{H}_r(n,m)}[K\subseteq H]$ be the probability that $H$ contains $K$ as a subhypergraph. Suppose
%that $n\to\infty$. Then
%\[
%\mathbb{P}_{\mathcal{H}_r(n,m)}[K\subseteq H]
% = \frac{ \binom{N-k}{m-k}}{ \binom{N}{m}}= \frac{[m]_k}{[N]_k}
%= \frac{[m]_k}{N^k}\exp\biggl[ \frac{[k]_2}{2N}
%+O\Bigl( \frac{k^3}{N^2}\Bigr)\biggr].
%\]
%We have the ratio
%$\frac{\mathbb{P}[K\subseteq H]}{\mathbb{P}_{\mathcal{H}_r(n,m)}[K\subseteq H]}$
%is increasing in $k$.
%\end{remark}

\section*{Acknowledgement}\label{s:11}

Fang Tian
was partially supported by the National Natural Science Foundation of China
(Grant No.~11871377) and China Scholarship Council [2017]3192, and is now a
visiting research fellow at the Australian National University. Fang Tian is
immensely grateful to Brendan D. McKay for giving her the opportunity to learn from him,
and thanks him for his problem and useful discussions.


\begin{thebibliography}{s2}
\bibitem{asas00}
A.\,S.~Asratian and N.\,N.~Kuzjurin, On the number of partial Steiner systems.
\textit{ J. Comb. Des.}, {\bf 81}(5) (2000), 347-352.

\bibitem{balgoh17}
J.~Balogh and L.~Li, On the number of linear hypergraphs of large girth.\\
\textit{arXiv:1709.04079}.

%\bibitem{kang1415}
%M. Behrisch, A. Coja-Oghlan, M. Kang, The asymptotic number of connected $d$-uniform hypergraphs. \textit{Combin.
%Probab. Comput.}, {\bf 23} (2014), 367-385. Corrigendum: \textit{Combin. Probab. Comput.}, {\bf 24} (2015), 373-375.

\bibitem{vlaejoc}
V.~Blinovsky and C.~Greenhill, Asymptotic enumeration of sparse uniform
hypergraphs with given degrees.
\textit{Eur. J. Combin.}, {\bf 51} (2016), 287-296.

\bibitem{valelec}
V.~Blinovsky and C.~Greenhill, Asymptotic enumeration of sparse uniform linear
hypergraphs with given degrees.
\textit{Electron. J. Comb.}, {\bf 23}(3) (2016), P3.17.

%\bibitem{bollobas16}
%B. Bollob\'{a}s and O. Riordan, Counting connected hypergraphs via the probabilistic method. {\it
%Comb. Probab. Comput.}, {\bf 25} (2016), 21-75.

\bibitem{cher52}
H. Chernoff,
A measure of asymptotic efficiency for tests of a hypothesis bases on the sum of observations. {\it
Annals of Mathematical Statistics.}, {\bf 23} (1952), 493-507.

\bibitem{dudek13}
A. Dudek, A. Frieze, A. Ruci\'{n}ski and M. \v{S}ileikis,
Approximate counting of regular hypergraphs.
\textit{Inform. Process. Lett.}, {\bf 113} (2013), 785-788.

\bibitem{grable96}
D.\,A. Grable and K.\,T. Phelps, Random methods in design theory: a survey.
\textit{J. Comb. Des.}, {\bf 4}(4) (1996), 255-273.

\bibitem{green06}
C. Greenhill, B.\,D. McKay and X. Wang, Asymptotic enumeration of sparse
$0-1$ matrices with irregular row and column
sums. \textit{J. Comb. Theory A}, {\bf 113} (2006), 291-324.

\bibitem{green08}
C. Greenhill and B.\,D. McKay, Asymptotic enumeration of sparse nonnegative
integer matrices with specified row and column
sums. \textit{Adv. Appl. Math.}, {\bf 41} (2008), 459-481.

\bibitem{green13}
C. Greenhill and B.\,D. McKay, Asymptotic enumeration of sparse multigraphs with given
degrees. \textit{SIAM J. Discrete Math.},
{\bf 27} (2013), 2064-2089.

%\bibitem{mika97}
%M. Karo\'{n}ski and T. {\L}uczak, The number of connected sparsely edged uniform
% hypergraphs. \textit{Discrete Math.}, {\bf 171} (1997),
%153-167.

%\bibitem{mckay02}
%B.\,D. McKay, I.\,M. Wanless and N.\,C. Wormald,
%Asymptotic enumeration of graphs with a given upper bound on the maximum degree.
%\textit{Comb. Probab. Comput.}, {\bf 11} (2002), 373-392.

\bibitem{rodl85}
V. R\"{o}dl,  On a packing and covering problem, \textit{Eur. J. Combin.},
{\bf 5} (1985), 69-78.

%\bibitem{sato14}
%C.\,M. Sato and N.\,C. Wormald, Asymptotic enumeration of sparse connected $3$-uniform hypergraphs. arXiv:1401.7381.

%\bibitem{wong89}
%R. Wong, Asymptotic approximation of integrals. \textit{Academic Press,  Boston}, 1989.

\end{thebibliography}
\end{document}